\documentclass[fontsize=12pt,a4paper,headings=normal,
twoside=false,leqno,parskip=half-,abstract=true]{scrartcl}
\usepackage[english]{babel}
\usepackage[utf8]{inputenc}
\usepackage{hyperref}
\hypersetup{
 pdftitle={Sturm 3-ball global attractors 2: Design of Thom-Smale complexes},
 pdfauthor={Bernold Fiedler, Carlos Rocha},
 colorlinks=true,
 linkcolor=blue,
 citecolor=blue,
 filecolor=blue,
 urlcolor=blue}

\usepackage{graphicx}
\usepackage[format=plain,labelfont=bf,font=small]{caption}
\usepackage{xcolor}
\usepackage[arrow, matrix, curve]{xy}

\usepackage{amsmath,amsthm}
\usepackage{amssymb} 
\newcommand{\transv}{\mathrel{\text{\tpitchfork}}}
\makeatletter
\newcommand{\tpitchfork}{%
  \vbox{
    \baselineskip\z@skip
    \lineskip-.52ex
    \lineskiplimit\maxdimen
    \m@th
    \ialign{##\crcr\hidewidth\smash{$-$}\hidewidth\crcr$\pitchfork$\crcr}
  }%
}
\makeatother
\usepackage{latexsym}
\usepackage[notref,notcite,color,final
]{showkeys}
\definecolor{refkey}{rgb}{1,0,0}
\definecolor{labelkey}{rgb}{1,0,0}

\usepackage[textwidth=2cm,textsize=small,backgroundcolor=none]{todonotes}

  \mathchardef\ordinarycolon\mathcode`\:
  \mathcode`\:=\string"8000
  \begingroup \catcode`\:=\active
    \gdef:{\mathrel{\mathop\ordinarycolon}}
  \endgroup

\theoremstyle{plain}
\newtheorem{thm}{Theorem}[section]
\newtheorem{lem}[thm]{Lemma}

\newtheorem{prop}[thm]{Proposition}
\newtheorem{cor}[thm]{Corollary}
\newtheorem{defi}[thm]{Definition}

\begin{document}

\title{{\LARGE{Sturm 3-ball global attractors 2:\\Design of Thom-Smale complexes}}}

\author{
 \\
{~}\\
Bernold Fiedler* and Carlos Rocha**\\
\vspace{2cm}}

\date{version of \today}
\maketitle
\thispagestyle{empty}

\vfill

*\\
Institut für Mathematik\\
Freie Universität Berlin\\
Arnimallee 3\\ 
14195 Berlin, Germany\\
\\
**\\
Center for Mathematical Analysis, Geometry and Dynamical Systems\\
Instituto Superior T\'ecnico\\
Universidade de Lisboa\\
Avenida Rovisco Pais\\
1049--001 Lisbon, Portugal\\


\newpage
\pagestyle{plain}
\pagenumbering{roman}
\setcounter{page}{1}

\begin{abstract}
This is the second of three papers on the geometric and combinatorial characterization of global Sturm attractors which consist of a single closed 3-ball.
The underlying scalar PDE is parabolic,
$$ u_t = u_{xx} + f(x,u,u_x)\,, $$
on the unit interval $0 < x<1$ with Neumann boundary conditions.
Equilibria are assumed to be hyperbolic.\\
\newline
Geometrically, we study the resulting Thom-Smale dynamic complex with cells defined by the fast unstable manifolds of the equilibria.
The Thom-Smale complex turns out to be a regular cell complex.
Our geometric description involves a bipolar orientation of the 1-skeleton, a hemisphere decomposition of the boundary 2-sphere by two polar meridians, and a meridian overlap of certain 2-cell faces in opposite hemispheres.\\
\newline
The combinatorial description is in terms of the Sturm permutation, alias the meander properties of the shooting curve for the equilibrium ODE boundary value problem.
It involves the relative positioning of extreme 2-dimensionally unstable equilibria at the Neumann boundaries $x=0$ and $x=1$, respectively, and the overlapping reach of polar serpents in the shooting meander.\\
\newline
In the first paper we showed the implications
$$ \text{Sturm attractor}\quad \Longrightarrow \quad \text{Thom-Smale complex} \quad \Longrightarrow \quad \text{meander}\,.$$
The present part 2, closes the cycle of equivalences by the implication
$$ \text{meander} \quad \Longrightarrow \quad \text{Sturm attractor}\,.$$
In particular this cycle allows us to construct a unique Sturm 3-ball attractor for any prescribed Thom-Smale complex which satisfies the geometric properties of the bipolar orientation and the hemisphere decomposition.
Many explicit examples and illustrations will be discussed in part 3.
The present 3-ball trilogy, however, is just another step towards the still elusive geometric and combinational characterization of all Sturm global attractors in arbitrary dimensions.

\end{abstract}

\newpage

\tableofcontents


\newpage
\pagenumbering{arabic}
\setcounter{page}{1}

\section{Introduction}
\label{sec1}

\numberwithin{equation}{section}
\numberwithin{figure}{section}

For a general introduction we first follow \cite{firo3d-1} and the references there.
\emph{Sturm global attractors} $\mathcal{A}_f$ are the global attractors of scalar parabolic equations
	\begin{equation}
	u_t = u_{xx} + f(x,u,u_x)
	\label{eq:1.1}
	\end{equation}
on the unit interval $0<x<1$.
Just to be specific we consider Neumann boundary conditions $u_x=0$ at $x=0,1$.
Standard semigroup theory provides local solutions $u(t,x)$ for $t \geq 0$ and given initial data at time $t=0$, in suitable Sobolev spaces $u(t, \cdot) \in X \subseteq C^1 ([0,1], \mathbb{R})$.
Under suitable dissipativeness assumptions on $f \in C^2$, any solution eventually enters a fixed large ball in $X$.
In fact that large ball of initial conditions itself limits onto the maximal compact and invariant subset $\mathcal{A}_f$ which is called the global attractor. See \cite{he81, pa83, ta79} for a general PDE background, and \cite{bavi92, chvi02, edetal94, ha88, haetal02, la91, ra02, seyo02, te88} for global attractors in general.

Equilibria $v = v(x)$ are time-independent solutions, of course, and hence satisfy the ODE
	\begin{equation}
	0 = v_{xx} + f(x,v,v_x)
	\label{eq:1.3}
	\end{equation} 
for $0\leq x \leq 1$, again with Neumann~boundary.
Here and below we assume that all equilibria $v$ of \eqref{eq:1.1}, \eqref{eq:1.3} are \emph{hyperbolic}, i.e. without eigenvalues (of) zero (real part) of their linearization.
Let $\mathcal{E} = \mathcal{E}_f \subseteq \mathcal{A}_f$ denote the set of equilibria.
Our generic hyperbolicity assumption and dissipativeness of $f$ imply that $N$:= $|\mathcal{E}_f|$ is odd.

It is known that \eqref{eq:1.1} possesses a Lyapunov~function, alias a variational or gradient-like structure, under separated boundary conditions;  see \cite{ze68, ma78, mana97, hu11, fietal14}. In particular, the global attractor consists of equilibria and of solutions $u(t, \cdot )$, $t \in \mathbb{R}$, with forward and backward limits, i.e.
	\begin{equation}
	\underset{t \rightarrow -\infty}{\mathrm{lim}} u(t, \cdot ) = v\,,
	\qquad
	\underset{t \rightarrow +\infty}{\mathrm{lim}} u(t, \cdot ) = w\,.
	\label{eq:1.2}
	\end{equation}
In other words, the $\alpha$- and $\omega$-limit sets of $u(t,\cdot )$ are two distinct equilibria $v$ and $w$.
We call $u(t, \cdot )$ a \emph{heteroclinic} or \emph{connecting} orbit, or \emph{instanton},  and write $v \leadsto w$ for such heteroclinically connected equilibria.

We attach the name of \emph{Sturm} to the PDE \eqref{eq:1.1}, and to its global attractor $\mathcal{A}_f$ because of a crucial nodal property of its solutions which we express by the \emph{zero number} $z$.
Let $0 \leq z (\varphi) \leq \infty$ count the number of (strict) sign changes of $\varphi : [0,1] \rightarrow \mathbb{R}, \, \varphi \not\equiv 0$.
Then
	\begin{equation}
	t \quad \longmapsto \quad z(u^1(t, \cdot ) - u^2(t, \cdot ))\,
	\label{eq:1.4}
	\end{equation}
is finite and nonincreasing with time $t$, for $t>0$ and any two distinct solutions $u^1$, $u^2$ of \eqref{eq:1.1}.
Moreover $z$ drops strictly with increasing $t$, at any multiple zero of $x \longmapsto u^1(t_0 ,x) - u^2(t_0 ,x)$; see \cite{an88}.
See Sturm \cite{st1836} for a linear autonomous version. For a first introduction see also \cite{ma82, brfi88, fuol88, mp88, brfi89, ro91, fisc03, ga04} and the many references there.

\begin{figure}[t!]
\centering \includegraphics[width=\textwidth]{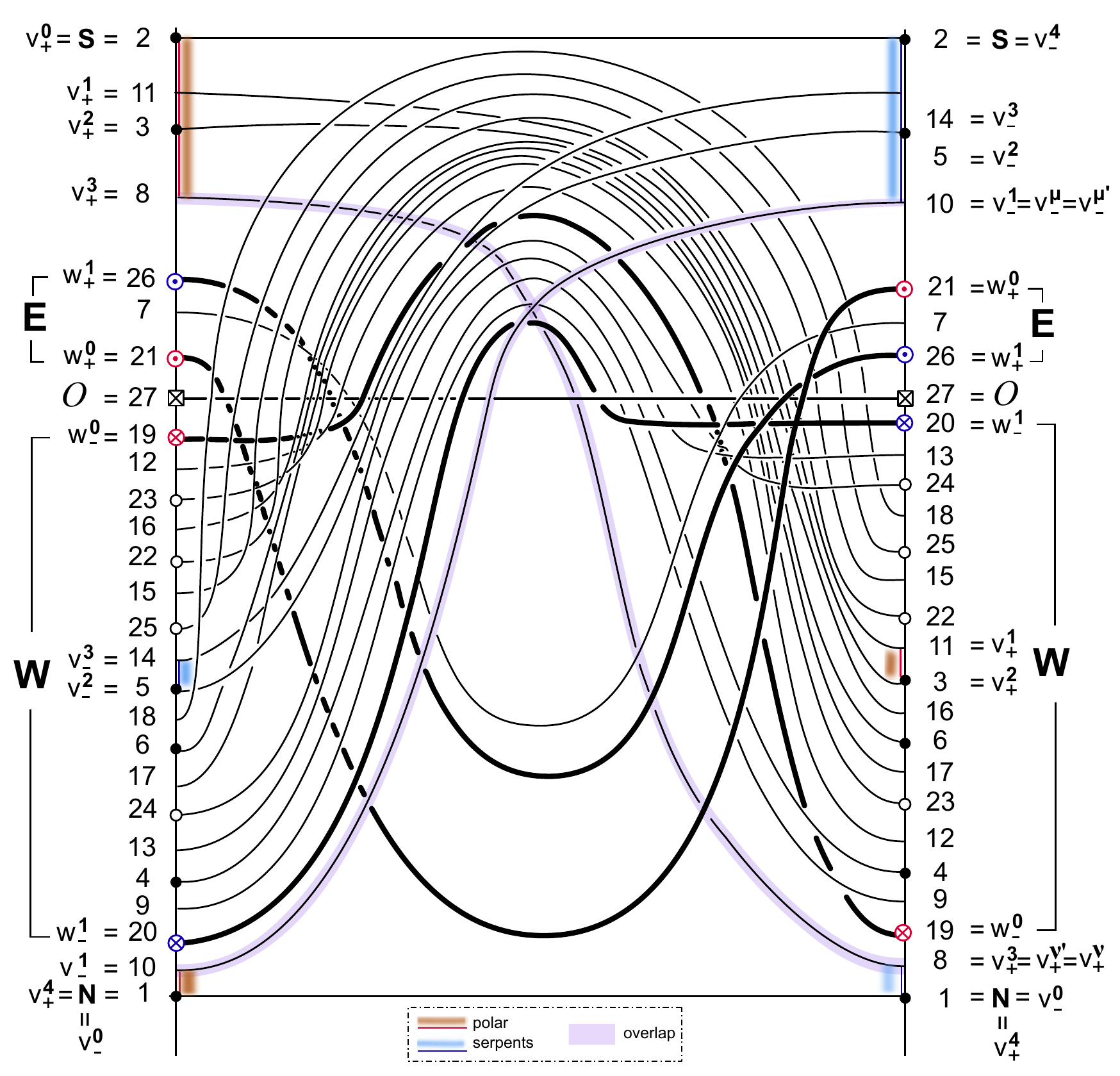}
\caption{\emph{
A sketch of the $27$ spatial profiles $v(x)$, for all equilibria in a solid Sturm octahedron $\mathcal{A}_f$.
The equilibria $1, \ldots , 27$ are ordered by $h_0^f$, $h_1^f$ along the left, right vertical axis $x=0,\, 1$, respectively.
Dots $\bullet$ indicate stable sink equilibria $1, \ldots , 6$ with Morse index $i=0$.
Circles $\circ$ indicate $8$ source equilibria $19, \ldots , 26$ with unstable dimension $i=2$.
The $i=3$ central equilibrium of the solid octahedron is $\mathcal{O}=27$.
The remaining equilibria $7, \ldots , 18$ indicate $i=1$ saddles.
See also figs.~\ref{fig:1.2} and \ref{fig:1.3}. For the notations $v_{\pm}^j,\ w_\pm^\iota$ see also fig.~\ref{fig:1.5}.
}}
\label{fig:1.1}
\end{figure}

The dynamic consequences of the Sturm structure are enormous.
In a series of papers, we have given a combinatorial description of Sturm global attractors $\mathcal{A}_f$; see \cite{firo96, firo99, firo00}.
Define the two \emph{labeling bijections} $h^f_0, h^f_1$: $\lbrace 1, \ldots, N \rbrace \rightarrow \mathcal{E}_f$ of the equilibria such that
	\begin{equation}
	h^f_\iota (1) < h^f_\iota (2) < \ldots < h^f_\iota (N) \qquad \text{at}
	\qquad x=\iota = 0,1\,.
	\label{eq:1.5}
	\end{equation}
See figs.~\ref{fig:1.1} and \ref{fig:5.1} for specific examples.
	
Our combinatorial description is based on the \emph{Sturm~permutation} $\sigma_f \in S_N$ which was introduced by Fusco and Rocha in \cite{furo91} and is defined as
	\begin{equation}
	\sigma_f:= (h^f_0)^{-1} \circ h^f_1\,.
	\label{eq:1.6}
	\end{equation}
Using a shooting approach to the ODE boundary value problem \eqref{eq:1.3}, the Sturm~permutations $\sigma_f \in S_N$ have been characterized as \emph{dissipative Morse meanders} in \cite{firo99}; see also \eqref{eq:1.19a}--\eqref{eq:1.22b} below for details.
In \cite{firo96} we have shown how to determine which equilibria $v$, $w$ possess a heteroclinic orbit connection \eqref{eq:1.2}, explicitly and purely combinatorially from $\sigma_f$.
A remaining puzzle were different, and even nonconjugate, Sturm permutations which still give rise to $C^0$ orbit-equivalent Sturm attractors; see also \cite[fig.~5.2]{firo3d-1}.
We will address this puzzle in theorem~\ref{thm:2.7} below.

Already at this elementary level, let us mention the four \emph{trivial equivalences} generated by the two commuting involutions $x \mapsto 1-x$ and $u \mapsto -u$; see \cite[definition 2.3]{firo3d-1}. Evidently, the first involution interchanges $h_0$ with $h_1$, and hence replaces the Sturm permutation $\sigma=h_0^{-1} \circ h_1$ by its inverse $\sigma^{-1}$. The second involution reverses the direction of the boundary orders $h_0, h_1$. This replaces $\sigma$ by its conjugate $\kappa \sigma \kappa$ under the flip $\kappa(j):=N+1-j$. Trivially, trivial equivalences give rise to trivially $C^0$ orbit-equivalent Sturm attractors. It is the remaining nontrivial equivalences, most of all, which theorem~\ref{thm:2.7} aims at.

\begin{figure}[t!]
\centering \includegraphics[width=\textwidth]{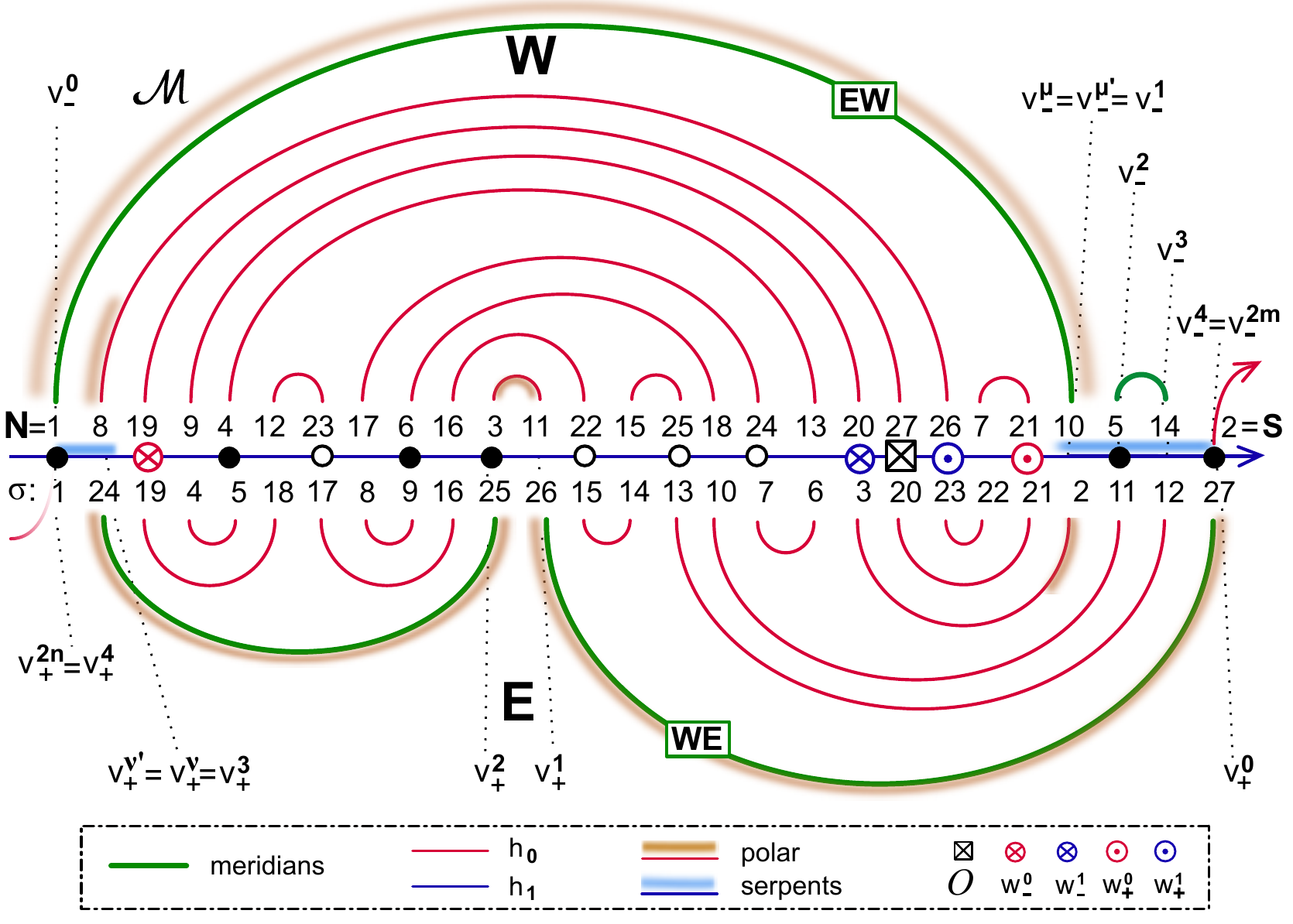}
\caption{\emph{
The Sturm meander $\mathcal{M}_f$ of the solid octahedron $\mathcal{A}_f$ of fig.~\ref{fig:1.1}.
Equilibrium labels above the horizontal $h_1$-axis, and Sturm permutation $\sigma_f = (h_0^f)^{-1} \circ h_1^f$ below.
Note how the shooting curve $h_0^f$ and the horizontal axis $h_1^f$ follow the equilibrium labels according to their enumerations in fig.~\ref{fig:1.2}.
Note consistency of all Morse numbers $i_v$ with Morse indices $i(v)$, for all equilibria $v$, according to fig.~\ref{fig:1.1}.
}}
\label{fig:1.2}
\end{figure}

For an explicit example of a Sturm permutation $\sigma_f$ which defines a solid octahedral Sturm global attractor $\mathcal{A}_f$ see figs.~\ref{fig:1.1} -- \ref{fig:1.3} and \cite[section~6]{firo3d-1}.
Fig.~\ref{fig:1.1} sketches the spatial profiles $v = v(x)$ for the $N =27$ equilibria $v \in \mathcal{E}_f$.
The boundary label maps $h_\iota^f$ are, specifically,
\begin{equation}
	\begin{aligned}
	h_0: \,\, 1\;\; &\text{10 20 9 4 13 24 17 6 18 5 14 25 15
	 22 16 23 12 19 27 21 7 26 8 3 11 2}\,;\\
	h_1: \,\, 1\;\; &\text{8  19 9  4  12 23 17 6  16 3  11 22 15
	25 18 24 13 20 27 26 7  21 10 5  14 2}\,.
	\end{aligned}
	\label{eq:1.5a}
	\end{equation}
Fig.~\ref{fig:1.2} depicts a stylized \emph{shooting meander} $\mathcal{M}_f$ associated to the octahedral Sturm permutation $\sigma_f$ which results from the boundary labels $h_\iota^f$ of the equilibria $v \in \mathcal{E}_f$ at $x = \iota = 0,1$, in ascending order.
By \eqref{eq:1.6} and \eqref{eq:1.5a},
	\begin{equation}
	\begin{aligned}
	\sigma 
	&=  \lbrace 1,\text{24, 19, 4, 5, 18, 17, 8, 9, 16, 25, 26, 15, 14,}\\
	&\phantom{=\lbrace 1,} \text{ 13, 10, 7, 6, 3, 20, 23, 22, 21, 2, 11, 12, 27}\rbrace =\\
	&= \text{(2 24) (3 19) (6 18) (7 17) (10 16) (11 25) (12 26) (13 15) (21 23)}\,.
	\end{aligned}
	\label{eq:1.6a}
	\end{equation}
Indeed, the $(v,v_x)$ phase plane of ODE \eqref{eq:1.3} at $x=1$ features the horizontal $v$-axis with equilibrium order $h_1^f$, as a Neumann boundary condition.
The meander curve $\mathcal{M}_f$ is the image, at $x=1$, which results, by shooting, from the Neumann initial conditions at $x=1$.
Hence the intersections of $\mathcal{M}_f$ with the horizontal axis represent the equilibrium set $\mathcal{E}_f$.
The ascending labeling $h_0^f$ of equilibria, at $x=0$, is the ordering of these intersections along $\mathcal{M}_f$.
The ascending labeling $h_1^f$ of equilibria, at $x=1$, is the ordering of these same intersections along the horizontal axis.

In fact it is the Sturm property of \eqref{eq:1.4} which implies the Morse-Smale property, for hyperbolic equilibria.
Indeed unstable and stable manifolds $W^u(v)$, $W^s(w)$, which intersect precisely along heteroclinic orbits $v \leadsto w$, are automatically transverse:
$W^u(v) \transv W^s(w)$.
See \cite{he85, an86}.
In the Morse-Smale setting, Henry already observed, that a heteroclinic orbit $v \leadsto w$ is equivalent to $w$ belonging to the boundary $\partial W^u(v)$ of the unstable manifold $W^u(v_-)$; see \cite{he85}.

More geometrically, global Sturm attractors $\mathcal{A}_f$ and $\mathcal{A}_g$ with the same Sturm permutation $\sigma_f = \sigma_g$ are $C^0$ orbit-equivalent \cite{firo00}.
Only for $C^1$-small perturbations, from $f$ to $g$, this global fact follows from $C^0$ structural stability of Morse-Smale systems; see e.g. \cite{pasm70} and \cite{pame82}.

For planar Sturm~attractors $\mathcal{A}_f$, i.e. for equilibrium sets $\mathcal{E}_f$ with a maximal Morse index two \cite{br90, jo89, ro91}, a slightly more geometric approach had been initiated in the planar Sturm trilogy \cite{firo08, firo09, firo10}.
It was clarified which planar graphs $\mathcal{H}$ do arise as connection graphs $\mathcal{H}=\mathcal{H}_f$ of planar Sturm attractors $\mathcal{A}_f$, and which ones do not.
Meanwhile, a \emph{Schoenflies~theorem} has also been proved to hold for the closure $\overline{W}^u(v) \subseteq X$ of the unstable manifold $W^u$ of any hyperbolic equilibrium $v$; see \cite{firo13}.
In particular $\overline{W}^u(v)$ is the homeomorphic Euclidean embedding of a closed unit ball $\overline{B}^{i(v)}$ of dimension $i(v)$.
In \cite{firo14} this allowed us to reformulate the combinatorial results of \cite{firo08, firo09, firo10}, in a more geometric and topological language, as follows.

We consider \emph{finite~regular} CW-\emph{complexes}
	\begin{equation}
	\mathcal{C} = \bigcup\limits_{v\in \mathcal{E}} c_v\,,
	\label{eq:1.7}
	\end{equation}
i.e. finite disjoint unions of \emph{cell~interiors} $c_v$ with additional gluing properties.
We think of the labels $v\in \mathcal{E}$ as \emph{barycenter} elements of $c_v$.
For CW-complexes we require the closures $\overline{c}_v$ in $\mathcal{C}$ to be the continuous images of closed unit balls $\overline{B}_v$ under \emph{characteristic maps}.
We call $\mathrm{dim}\,\overline{B}_v$ the dimension of the (open) cell $c_v$. 
For positive dimensions of $\overline{B}_v$ we require $c_v$ to be the homeomorphic images of the interiors $B_v$. 
For dimension zero we write $B_v := \overline{B}_v$ so that any 0-cell $c_v= B_v$ is just a point.
The \emph{m-skeleton} $\mathcal{C}^m$ of $\mathcal{C}$ consists of all cells of dimension at most $m$.
We require $\partial c_v := \overline{c}_v \setminus c_v \subseteq \mathcal{C}^{m-1}$ for any $m$-cell $c_v$.
Thus, the boundary $(m-1)$-sphere $S_v := \partial B_v = \overline{B}_v \setminus B_v$ of any $m$-ball  $B_v$, $m>0$, maps into the $(m-1)$-skeleton,
	\begin{equation}
	\partial B_v \quad \longrightarrow \quad \partial c_v \subseteq \mathcal{C}^{m-1}\,,
	\label{eq:1.8}
	\end{equation}
for the $m$-cell $c_v$, by restriction of the continuous characteristic map.
The map \eqref{eq:1.8} is called the \emph{attaching} (or \emph{gluing}) \emph{map}.
For \emph{regular} CW-complexes, in contrast, the characteristic maps $ \overline{B}_v  \rightarrow  \overline{c}_v $ are required to be homeomorphisms, up to and including the \emph{attaching} (or \emph{gluing}) \emph{homeomorphism}. We moreover require $\partial{c_v}$  to be a sub-complex of $\mathcal{C}^{m-1}$, then. 
See \cite{frpi90} for a background on this terminology.

\begin{figure}[t!]
\centering \includegraphics[width=\textwidth]{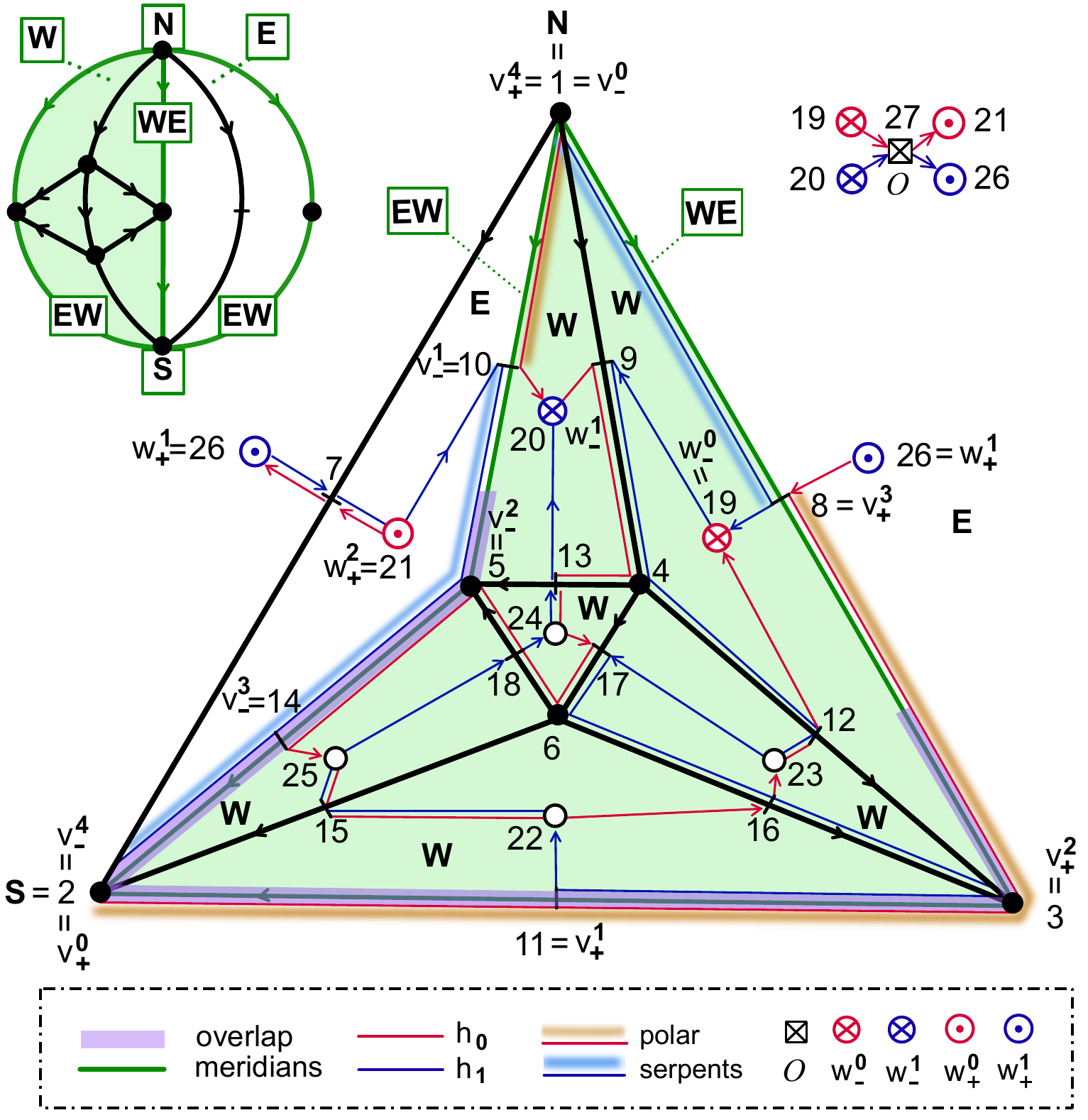}
\caption{\emph{
Sketch of the labeling bijections $(h_0^f,h_1^f)$ in the bipolar dynamic 3-cell complex of the solid octahedron $\mathcal{A}_f$ of fig.~\ref{fig:1.1}.
The backwards face $1\,2\,3$ of the octahedron is exterior.
Note the hemisphere decomposition $\mathbf{W}$, $\mathbf{E}$ by the two meridians $\mathbf{EW}$ and $\mathbf{WE}$ from pole $\mathbf{N}$ to pole $\mathbf{S}$.
The paths $h_\iota^f$ with respect to equilibrium labels in the figure are listed in \eqref{eq:1.5a}.
See also \eqref{eq:1.6a} for the resulting Sturm permutation $\sigma = h_0^{-1} \circ h_1$, and fig.~\ref{fig:1.2} for the resulting Sturm meander.
}}
\label{fig:1.3}
\end{figure}

\begin{figure}[t!]
\centering \includegraphics[width=\textwidth]{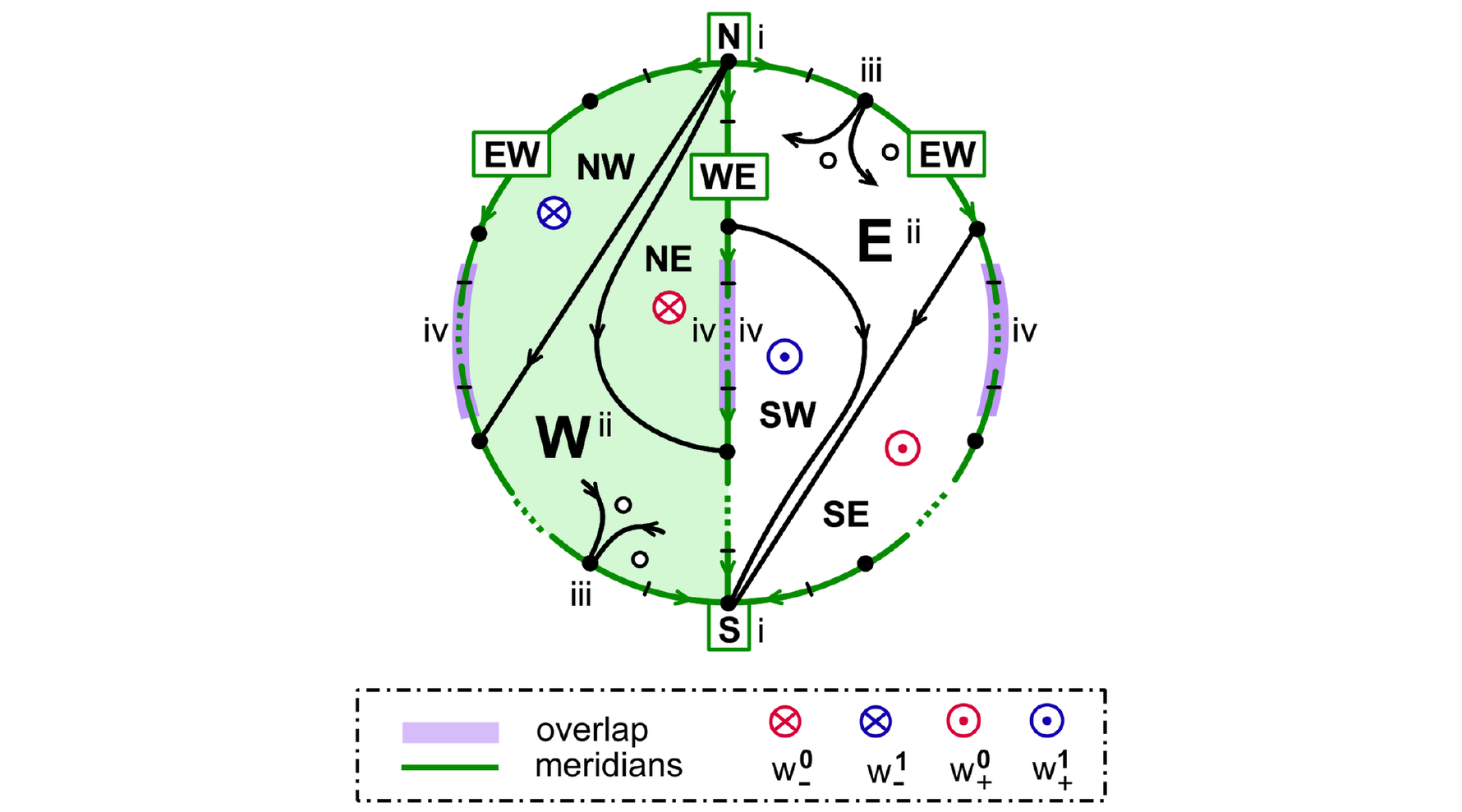}
\caption{\emph{
A 3-cell template. Shown is the $S^2$ boundary of the single 3-cell $c_\mathcal{O}$ with poles $\mathbf{N}$, $\mathbf{S}$, hemispheres $\mathbf{W}$ (green), $\mathbf{E}$ and separating meridians $\mathbf{EW}$, $\mathbf{WE}$ (green).
The right and the left boundaries denote the same $\mathbf{EW}$ meridian and have to be identified.
Dots $\bullet$ are sinks, and small circles $\circ$ are sources.
Note the hemisphere decomposition~(ii), the edge orientations~(iii) at meridian boundaries, and the meridian overlaps~(iv) of the $\mathbf{N}$-adjacent meridian faces $\otimes = w_-^\iota$ with their $\mathbf{S}$-adjacent counterparts $\odot =w_+^\iota$.
For $w_\pm^\iota$ see also \eqref{eq:1.19}.
For a specific octahedron example see fig.~\ref{fig:1.3}.
}}
\label{fig:1.4}
\end{figure}

The disjoint dynamic decomposition
	\begin{equation}
	\mathcal{A}_f = \bigcup\limits_{v \in \mathcal{E}_f} W^u(v)
	\label{eq:1.9}
	\end{equation}
of the global attractor $\mathcal{A}_f$ into unstable manifolds $W^u$ of equilibria $v$ is called the \emph{Thom-Smale complex} or \emph{dynamic complex}; see for example \cite{fr79, bo88, bizh92}.
In our Sturm setting \eqref{eq:1.1} with hyperbolic equilibria $v \in \mathcal{E}_f$, the Thom-Smale complex is a finite regular CW-complex.
The open cells $c_v$ are the unstable manifolds $W^ u (v)$ of the equilibria $v \in \mathcal{E}_f$.
The proof is closely related to the Schoenflies result of \cite{firo13}; see \cite{firo14}.
We can therefore define the \emph{Sturm~complex} $\mathcal{C}_f$ to be the regular Thom-Smale dynamic complex 
	\begin{equation}
	\mathcal{C}_f := \bigcup_{v \in \mathcal{E}_f}\, W^u(v)
	\label{eq:1.7a}
	\end{equation}
of the Sturm global attractor $\mathcal{A}_f$, provided all equilibria $v \in \mathcal{E}_f$ are hyperbolic.
Again we call the equilibrium $v \in \mathcal{E}_f$ the \emph{barycenter} of the cell $c_v=W^u(v)$.

A planar Sturm complex $\mathcal{C}_f$, for example, is the Thom-Smale complex of a planar Sturm global attractor $\mathcal{A}_f$ for which all equilibria $v \in \mathcal{E}_f$ have Morse indices $i(v) \leq 2$.
See section~\ref{sec3} for a detailed discussion, based on our planar Sturm trilogy \cite{firo08, firo09, firo10}.
See fig.~\ref{fig:1.3} for the Sturm complex of the solid octahedron attractor $\mathcal{A}_f$ defined by the Sturm permutation $\sigma_f$ of \eqref{eq:1.6a} and figs.~\ref{fig:1.1}, \ref{fig:1.2}.

Our main objective, in the present trilogy of papers, is a geometric and combinatorial characterization of those global Sturm attractors, which are the closure
	\begin{equation}
	\mathcal{A}_f = \text{clos } W^u (\mathcal{O})
	\label{eq:1.10}
	\end{equation}
of the unstable manifold $W^u$ of a single equilibrium $v = \mathcal{O}$ with Morse index $i(\mathcal{O}) =3$.
We call such an $\mathcal{A}_f$ a 3-\emph{ball Sturm attractor}.
Recall that we assume all equilibria $v \in \mathcal{E}_f$ to be hyperbolic:
\emph{sinks} have Morse index $i=0$, \emph{saddles} have $i=1$, and \emph{sources}  $i=2$.
This terminology also applies when viewed within the flow-invariant and attracting boundary 2-sphere
	\begin{equation}
	\Sigma^2 = \partial W^u(\mathcal{O}):= \left(
	\text{clos } W^u(\mathcal{O})\right) \smallsetminus
	W^u (\mathcal{O})\,.
	\label{eq:1.11}
	\end{equation}
Correspondingly we call the associated cells $c_v = W^u(v)$ of the dynamic cell complex, or of any regular cell complex, \emph{vertices}, \emph{edges}, and {\emph{faces}.
The graph of vertices and edges, for example, defines the 1-skeleton $\mathcal{C}^1$ of the 3-ball cell complex $\mathcal{C} = \bigcup_v \, c_v$.

For a geometric characterization of 3-ball Sturm attractors $\mathcal{A}_f$ in \eqref{eq:1.10}, by their dynamic complexes \eqref{eq:1.9}, we now drop all Sturmian PDE interpretations.
Instead we define 3-cell templates, abstractly, in the class of regular cell complexes and without any reference to PDE or dynamics terminology.
See fig.~\ref{fig:1.4} for an illustration.

\begin{defi}\label{def:1.1}
A finite disjoint union $\mathcal{C} = \bigcup_{v \in \mathcal{E}} c_v$ of cells $c_v$ is called a \emph{3-cell template} if $\mathcal{C}$ is a regular cell complex and the following four conditions all hold.
\begin{itemize}
\item[(i)] $\mathcal{C} = \text{clos } c_{\mathcal{O}}= S^2 \,\dot{\cup}\, c_{\mathcal{O}}$ is the closure of a single 3-cell $c_{\mathcal{O}}$.
\item[(ii)] The 1-skeleton $\mathcal{C}^1$ of $\mathcal{C}$ possesses a \emph{bipolar orientation} from a pole vertex $\mathbf{N}$ (North) to a pole vertex $\mathbf{S}$ (South), with two disjoint directed \emph{meridian paths} $\mathbf{WE}$ and $\mathbf{EW}$ from $\mathbf{N}$ to $\mathbf{S}$.
The circle of meridians decomposes the boundary sphere $S^2$ into remaining \emph{hemisphere} components $\mathbf{W}$ (West) and $\mathbf{E}$ (East), both open in $S^2$.
\item[(iii)] Edges are oriented towards the meridians, in $\mathbf{W}$, and away from the meridians, in $\mathbf{E}$, at end points on the meridians other than the poles $\mathbf{N}$, $\mathbf{S}$.
\item[(iv)] Let $\mathbf{NE}$, $\mathbf{SW}$ denote the unique faces in $\mathbf{W}$, $\mathbf{E}$, respectively, which contain the first, last edge of the meridian $\mathbf{WE}$ in their boundary.
Then the boundaries of $\mathbf{NE}$ and $\mathbf{SW}$ overlap in at least one shared edge of the meridian $\mathbf{WE}$.

Similarly, let $\mathbf{NW}$, $\mathbf{SE}$ denote the unique faces in $\mathbf{W}$, $\mathbf{E}$, adjacent to the first, last edge of the other meridian $\mathbf{EW}$, respectively.
Then their boundaries overlap in at least one shared edge of $\mathbf{EW}$.
\end{itemize}
\end{defi}

We recall here that an edge orientation of the 1-skeleton $\mathcal{C}^1$ is called bipolar if it is without directed cycles, and with a single ``source'' vertex $\mathbf{N}$ and a single ``sink'' vertex $\mathbf{S}$ on the boundary of $\mathcal{C}$.
Here ``source'' and ``sink'' are understood, not dynamically but, with respect to edge orientation.
To avoid any confusion with dynamic $i=0$ sinks and $i=2$ sources, below, we call $\mathbf{N}$ and $\mathbf{S}$ the North and South pole, respectively.

With the above notation and definition we can now formulate the main result of the present paper.

\begin{thm}\label{thm:1.2}
Let $\mathcal{C} = \bigcup_{v \in \mathcal{E}} c_v$ be a finite disjoint union of cells.
Then $\mathcal{C} = \mathcal{C}_f$ is the Thom-Smale dynamic cell complex of a 3-ball Sturm attractor $\mathcal{A}_f$ if, and only if, $\mathcal{C}$ is a 3-cell template.
More precisely, there exists a cell-preserving homeomorphism
	\begin{equation}
	\Phi: \qquad \mathcal{C} = \bigcup\limits_{v \in \mathcal{E}} c_v
	\quad \longrightarrow \quad 
	\bigcup\limits_{v \in \mathcal{E}_f} W^u(v)
	= \mathcal{C}_f = \mathcal{A}_f
	\label{eq:hom}
	\end{equation}
with $\Phi (c_v) = W^u\left(\Phi(v)\right)$.
\end{thm}

Here $\Phi$ also identifies the abstract labels $v \in \mathcal{E}$ of the cells $c_v$ with the generating equilibria $\Phi(v) \in \mathcal{E}_f$ of the unstable manifolds $W^u(v)$ of Morse index dimension $\dim c_v = i(v)$.

In \cite{firo14} we have proved a precursor of theorem~\ref{thm:1.2}:
any finite regular cell complex which is the closure of a single 3-cell is, in fact, the dynamic complex of a suitable Sturm 3-ball.
This requires condition~(i) of definition~\ref{def:1.1}, only.
The full geometric characterization of Sturm 3-balls as 3-cell templates, in theorem~\ref{thm:1.2}, is much more detailed, of course.
It turns out that any finite regular 2-sphere complex possesses a bipolar orientation, with edge adjacent poles, and a hemisphere decomposition, with a single Western face, which defines a 3-cell template.
Therefore theorem~\ref{thm:1.2} refines \cite{firo14}.

In section~\ref{sec2} we translate the geographic language of definition~\ref{def:1.1}, for 3-cell templates, into the broader concept of signed hemisphere decompositions.
At the heart of this is a convenient notational variant of the zero number $z$.
We write
	\begin{equation}
	z(\varphi) = j_{\pm}
	\label{eq:1.4+}
	\end{equation}
to indicate $j$ strict sign changes of $\varphi$, by $j$, and $\pm \varphi (0) >0$, by the index $\pm$.
For example $z(\pm \varphi_j) = j_{\pm}$, for the $j$-th Sturm-Liouville eigenfunction $\varphi_j$.
By the Schoenflies result \cite{firo13} and \cite[proposition~3.1]{firo3d-1} this provides a disjoint \emph{signed hemisphere decomposition}
	\begin{equation}
	\partial W^u(v) =
	\bigcup\limits_{0 \leq j< i(v)} 
	\Sigma _\pm^j(v)
	\label{eq:1.12}
	\end{equation}
of the boundary sphere $\Sigma^{i(v)-1} = \partial W^u(v)$ of any unstable manifold, such that
	\begin{equation}
	\Sigma_\pm^j (v) := 
	\bigcup_{\substack{v \leadsto w\\z(w-v) = j_\pm}} W^u(w)\,.
	\label{eq:1.13}
	\end{equation}
For the fast unstable manifolds $W^k$ of $v$ with dimensions $1 \leq k \leq i(v)$, we obtain  analogously
	\begin{equation}
	\partial W^k(v) = 
	\bigcup\limits_{0 \leq j< k}
	\Sigma _\pm^j(v)\,.
	\label{eq:1.14}
	\end{equation}
See \eqref{eq:2.5}--\eqref{eq:2.12} for details.
With the abbreviation $\Sigma_\pm^j$:= $\Sigma_\pm^j(\mathcal{O})$, the translation table between the signed hemispheres decomposition \eqref{eq:1.12}, \eqref{eq:1.13} of $\partial W^u(\mathcal{O}) = \mathcal{A}_f \smallsetminus W^u(\mathcal{O})$, for the Sturm 3-ball $\mathcal{A}_f$ in theorem~\ref{thm:1.2}, and the geographic 3-cell template $\mathcal{C}$ of definition~\ref{def:1.1}, is as follows:
	\begin{equation}
	\begin{aligned}
	(\Sigma_-^0, \Sigma_+^0) \quad &\mapsto \quad (\mathbf{N}, 					\mathbf{S})\\
	(\Sigma_-^1, \Sigma_+^1) \quad &\mapsto \quad (\mathbf{EW}, 					\mathbf{WE})\\
	(\Sigma_-^2, \Sigma_+^2) \quad &\mapsto \quad (\mathbf{W}, \mathbf{E})
	\,.
	\end{aligned}
	\label{eq:1.15}
	\end{equation}
In theorem~\ref{thm:2.6} below we refine theorem~\ref{thm:1.2}, such that the homeomorphism $\Phi$ respects a signed hemisphere decomposition, not only for $\partial W^u(\mathcal{O})$ but, for the sphere boundary $\partial W^u(v)$ of any unstable manifold in $\mathcal{A}_f$.
In theorem~\ref{thm:2.7} we will show how the Sturm permutation $\sigma_f$, and therefore the Sturm global attractor $\mathcal{A}_f$ itself (up to $C^0$ orbit equivalence), is determined uniquely by the signed hemisphere decompositions \eqref{eq:1.12}, \eqref{eq:1.13}.

As an elementary example, in section~\ref{sec3}, we review and adapt our results from the planar trilogy \cite{firo08, firo09, firo10} to the present setting of signed hemispheres.
Our focus is on the equivalence of boundary bipolar orientations with the above language of signed hemisphere decompositions and fast unstable manifolds.
In particular we recall, and justify, the face transition rules of \cite[definition~2.2]{firo3d-1} for ZS-pairs $(h_0,h_1)$ in bipolar planar cell complexes, in corollary~\ref{cor:3.2}, using the language of signed hemisphere complexes.

In \cite[theorem~5.2]{firo3d-1} of part~1 we have associated a certain Sturm global attractor $\mathcal{A}_f$ to any abstractly given 3-cell template $\mathcal{C}$.
In fact we have constructed abstract paths $h_\iota$ in $\mathcal{C}$, for $\iota =0,1$, by recipe or decree ex cathedra, such that the abstract permutation
	\begin{equation}
	\sigma := h_0^{-1} \circ h_1
	\label{eq:1.16}
	\end{equation}
was a dissipative Morse meander and hence, by \cite{firo96}, a Sturm permutation $\sigma = \sigma_f$ for some concrete nonlinearity $f$.

Let us now recall this terminology in some detail.
Abstractly, a \emph{meander} is an oriented planar $C^1$ Jordan curve $\mathcal{M}$ which crosses a positively oriented horizontal axis at finitely many points.
The curve $\mathcal{M}$ is assumed to run from Southwest to Northeast, asymptotically, and all $N$ crossings are assumed to be transverse; see \cite{ar88, arvi89}.
Note $N$ is odd.
Enumerating the $N$ crossing points $v \in \mathcal{E}$ along the meander $\mathcal{M}$ and along the horizontal axis, respectively, we obtain two labeling bijections
	\begin{equation}
	h_0,h_1: \quad \lbrace 1, \ldots , N \rbrace \rightarrow \mathcal{E}\,.
	\label{eq:1.19a}
	\end{equation}
Define the \emph{meander permutation} $\sigma \in S_N$ as 
	\begin{equation}
	\sigma := h_0^{-1} \circ h_1.
	\label{eq:1.19b}
	\end{equation}
We call the meander $\mathcal{M}$ \emph{dissipative} if
	\begin{equation}
	\sigma(1) =1, \quad \sigma(N) =N
	\label{eq:1.19c}
	\end{equation}
are fixed under $\sigma$.
We define \emph{Morse numbers} $i_v$ for the intersections $v \in \mathcal{E}$ of the meander $\mathcal{M}$ with the horizontal $h_1$-axis, recursively, by
	\begin{equation}
	\begin{aligned}
	i_{h_0(1)} &:=\quad i_{h_0(N)} :=\quad 0\,,\\
	i_{h_0(j+1)} &:=\quad i_{h_0(j)}+(-1)^{j+1} 
	\text{sign} (\sigma^{-1}(j+1)-\sigma^{-1}(j))\,.
	\end{aligned}
	\label{eq:1.17}
	\end{equation}
Equivalently, by recursion along $h_1$:
	\begin{equation}
	\begin{aligned}
	i_{h_1(1)} &:=\quad i_{h_1(N)} :=\quad 0\,,\\
	i_{h_1(j+1)} &:=\quad i_{h_1(j)}+(-1)^{j+1} 
	\text{sign} (\sigma (j+1)-\sigma (j))\,.
	\end{aligned}
	\label{eq:1.18}
	\end{equation}
Note how the enumeration of intersections $v\in \mathcal{E}$ by $h_\iota$: $\{ 1, \ldots , N\} \rightarrow \mathcal{E}$ depends on $h_\iota$, of course, but the Morse numbers $i_v$ only depend on the Sturm permutation $\sigma$ which defines the meander $\mathcal{M}$.

We call the meander $\mathcal{M}$ \emph{Morse}, if
	\begin{equation}
	i_v \geq 0\,,
	\label{eq:1.24}
	\end{equation}
for all $v \in \mathcal{M}$.

We call $\mathcal{M}$ \emph{Sturm meander}, if $\mathcal{M}$ is a dissipative Morse meander; see \cite{firo96}. 
Conversely, given any permutation $\sigma \in S_N$, we can define an associated curve $\mathcal{M}$ of arches over the horizontal axis which switches sides at the intersections $\mathcal{E}=\{1,\ldots,N\}$ on the axis, in the order of $\sigma$. This fixes the labeling $h_1=\mathrm{id}$ and $h_0=\sigma^{-1}$.
A \emph{Sturm permutation} $\sigma$ is a permutation such that the associated curve $\mathcal{M}$ is a Sturm meander.
The main paradigm of \cite{firo96} is the equivalence of Sturm meanders $\mathcal{M}$ with shooting curves of the Neumann ODE problem \eqref{eq:1.3}.
In fact, the Neumann shooting curve is a Sturm meander, for any dissipative nonlinearity $f$ with hyperbolic equilibria.
Conversely, for any permutation $\sigma$ of a Sturm meander $\mathcal{M}$ there exist dissipative $f$ with hyperbolic equilibria such that $\sigma = \sigma_f$ is the Sturm permutation of $f$.
In particular, the intersections $v$ of the meander $\mathcal{M}$ with the horizontal $v$-axis are the boundary values of the equilibria $v \in \mathcal{E}_f$ at $x=1$, and the Morse number
	\begin{equation}
	i_v = i(v)
	\label{eq:1.22b}
	\end{equation}
is the Morse index of $v$.
For that reason we have used closely related notation to describe either case.

In particular, \eqref{eq:1.22b} extends the terminology of \emph{sinks} $i_v =0$, \emph{saddles} $i_v =1$, and \emph{sources} $i_v=2$ to abstract Sturm meanders.
We insist, however, that our above definition~\eqref{eq:1.19a}--\eqref{eq:1.24} is completely abstract and independent of this ODE/PDE interpretation.

For example, consider the case $i_{\mathcal{O}} =3$ of a single intersection $v= \mathcal{O}$ with Morse number $3$.
Suppose $i_v \leq 2$ for all other Morse numbers.
Then \eqref{eq:1.17} implies $i=2$ for the two $h_0$-\emph{neighbors} $h_0(h_0^{-1} (\mathcal{O})\pm1)$ of $\mathcal{O}$ along the meander $\mathcal{M}$.
In other words, these neighbors are both sources.
The same statement holds true for the two $h_1$-\emph{neighbors} $h_1(h_1^{-1}(\mathcal{O})\pm 1)$ of $\mathcal{O}$ along the horizontal axis.
To fix notation, we denote these $h_\iota$-neighbors by
	\begin{equation}
	w_\pm^\iota:= h_\iota (h_\iota^{-1} (\mathcal{O}) \pm 1)\,,
	\label{eq:1.19}
	\end{equation}
for $\iota = 0,1$.
The $h_\iota$-\emph{extreme sources} are the first and last source intersections $v$ of the meander $\mathcal{M}$ with the horizontal axis, in the order of $h_\iota$.

Reminiscent of cell template terminology, we call the extreme sinks $\mathbf{N} = h_0(1) = h_1(1)$ and $\mathbf{S} = h_0(N) = h_1(N)$ the (North and South) \emph{poles} of the Sturm meander $\mathcal{M}$.
A \emph{polar} $h_\iota$-\emph{serpent}, for $\iota =0,1$, is a set of $v =h_\iota (m) \in \mathcal{E}$ for a maximal interval of integers $m$ which contains a pole, $\mathbf{N}$ or $\mathbf{S}$, such that
	\begin{equation}
	i_{h_\iota(m)} \in \lbrace 0,1\rbrace
	\label{eq:1.20}
	\end{equation}
for all $m$.
To visualize the serpent we often include the meander or axis path joining $v$ in the serpent.
See figs.~\ref{fig:1.2} and \ref{fig:1.5} for examples.
We call $\mathbf{N}$-polar serpents and $\mathbf{S}$-polar serpents anti-polar to each other.
An \emph{overlap} of anti-polar serpents simply indicates a nonempty intersection.
For later reference, we call a polar $h_\iota$-serpent \emph{full} if it extends all the way to the saddle which is $h_{1-\iota}$-adjacent to the opposite pole.
Full $h_\iota$-serpents always overlap with their anti-polar $h_{1-\iota}$-serpent, of course, at least at that saddle.

\begin{figure}[t!]
\centering \includegraphics[width=\textwidth]{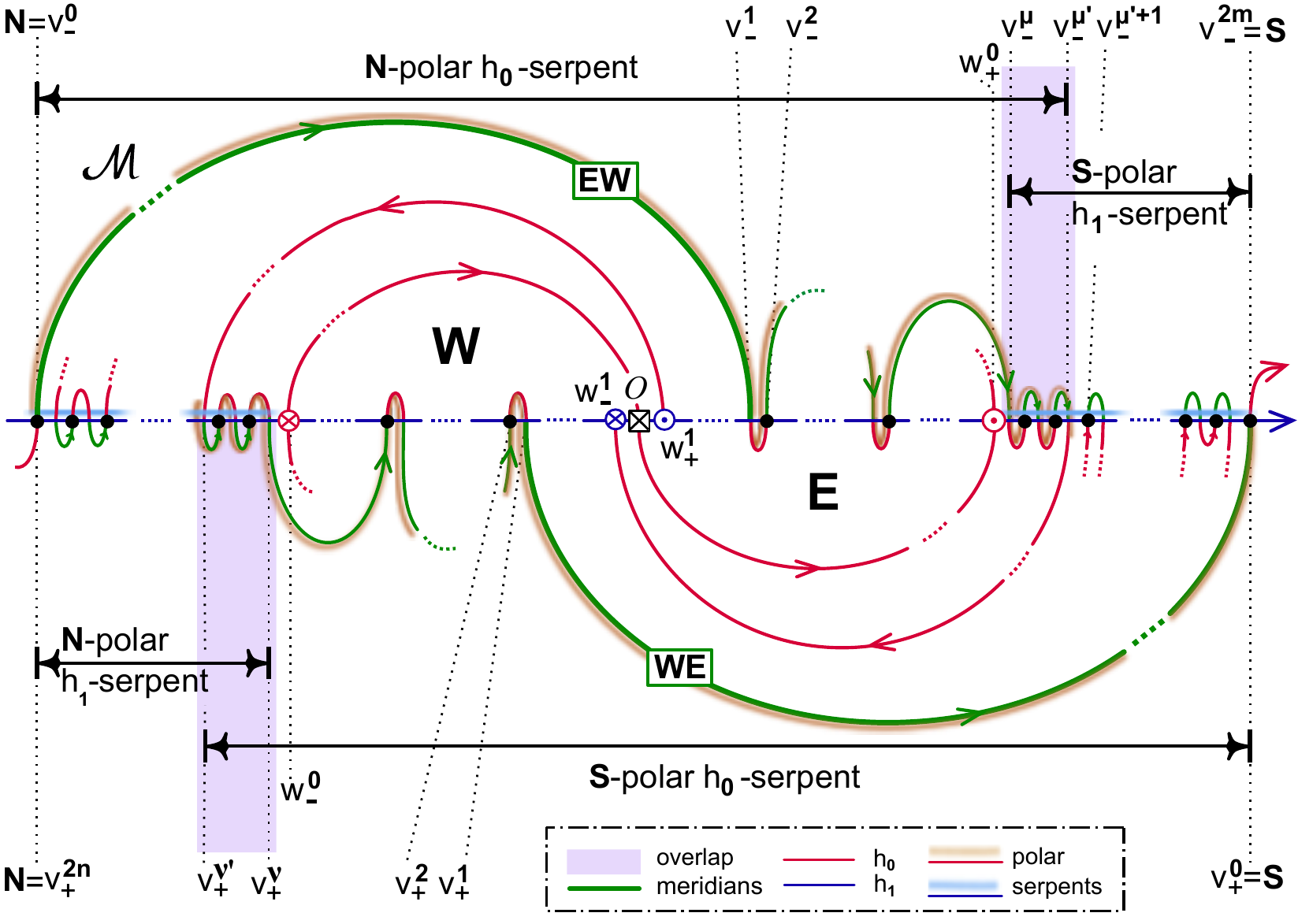}
\caption{\emph{
A 3-meander template.
Note the $\mathbf{N}$-polar $h_1$-serpent $\mathbf{N} = v_+^{2n} \ldots  v_+^\nu$ terminated at $v_+^\nu$ by the source $w_-^0$ which is, both, $h_1$-extreme minimal and the lower $h_0$-neighbor of $\mathcal{O}$.
This serpent overlaps the anti-polar, i.e. $\mathbf{S}$-polar, $h_0$-serpent $v_+^{\nu '} \ldots  v_+^\nu \ldots  v_+^0  = \mathbf{S}$, from $v_+^{\nu '}$ to $v_+^\nu$.
Similarly, the $\mathbf{N}$-polar $h_0$-serpent $\mathbf{N} = v_-^0 \ldots  v_-^{\mu '}$ overlaps the anti-polar, i.e. $\mathbf{S}$-polar, $h_1$-serpent $v_-^\mu  \ldots  v_-^{\mu '} \ldots  v_-^{2n} = \mathbf{S}$, from $v_-^\mu$ to $v_-^{\mu '}$.
The $h_1$-neighbors $w_\pm^1$ of $\mathcal{O}$ are  the $h_0$-extreme sources, by the two polar $h_0$-serpents.
Similarly, the $h_0$-neighbors $w_\pm^0$ of $\mathcal{O}$ define the $h_1$-extreme sources.
Compare also the octahedral example of fig.~\ref{fig:1.2}.
}}
\label{fig:1.5}
\end{figure}

\begin{defi}\label{def:1.3}
An abstract Sturm meander $\mathcal{M}$ with axis intersections $v \in \mathcal{E}$ is called a \emph{3-meander template} if the following four conditions hold, for $\iota = 0,1$.

\begin{itemize}
\item[(i)] $\mathcal{M}$ possesses a single axis intersection $v = \mathcal{O}$ with Morse number $i_{\mathcal{O}} =3$, and  no other Morse number exceeds $2$.
\item[(ii)] Polar $h_\iota$-serpents overlap with their anti-polar $h_{1-\iota}$-serpents in at least one shared vertex.
\item[(iii)] The intersection $v=\mathcal{O}$ is located between the two intersection points, in the order of $h_{1-\iota}$, of the polar arc of any polar $h_\iota$-serpent.
\item[(iv)] The $h_\iota$-neighbors $w^\iota_\pm$ of $v=\mathcal{O}$ are the $i=2$ sources which terminate the polar $h_{1-\iota}$-serpents.
\end{itemize}
\end{defi}

See fig.~\ref{fig:1.5} for an illustration of 3-meander templates.
Property (iv), for example, asserts that the $h_\iota$-neighbor sources $w_\pm^\iota$ of $\mathcal{O}$ are the $h_{1-\iota}$-extreme sources, for $\iota = 0,1$. For the Sturm boundary orders $h^f_\iota$
this is a useful exercise in polar serpents; see \cite[lemma~4.3(iii)]{firo3d-1}.

In \cite[theorem~5.2]{firo3d-1} we have established the passage
	\begin{equation}
	\text{3-cell template} \quad \Longrightarrow \quad
	\text{3-meander template}\,,
	\label{eq:1.21}
	\end{equation}
based on the above construction.
The 3-meander template $\mathcal{M}$ and its Sturm permutation $\sigma = h_0^{-1}h_1$, in turn, define a Sturm nonlinearity $f$ such that $\sigma_f = \sigma$.
Let $\mathcal{A}_f$ denote the Sturm global attractor of $f$.

In theorem~\ref{thm:5.1} below, we claim that $\mathcal{A}_f$ is in fact a Sturm 3-ball.
We prepare the proof, in section~\ref{sec4}, by a formal \emph{scoop of noses and signed hemispheres}, which does not affect heteroclinic connectivity in the closure of the opposite hemisphere; see \eqref{eq:4.4} and definition~\ref{def:4.2}.

We prove the refined version, theorem~\ref{thm:2.6}, of theorem~\ref{thm:1.2}, and uniqueness theorem~\ref{thm:2.7} on the Sturm permutations of prescribed Sturm 3-cell templates, in the final section~\ref{sec7}.
This is based on the crucial identity
	\begin{equation}
	h_\iota^f = h_\iota
	\label{eq:1.22}
	\end{equation}
between the labeling orders $h_\iota^f$: $\{ 1, \ldots , N\} \rightarrow \mathcal{E}_f$ of equilibria $v \in \mathcal{E}_f$, according to the order of their boundary values $v(x)$ at $x=\iota =0,1$, and the SZS labeling paths $h_\iota$ in the abstract Sturm complex $\mathcal{C} = \mathcal{C}_f$ of the cells $c_v = W^u(v)$, for all $v \in \mathcal{E} = \mathcal{E}_f$.
More precisely we will prove \eqref{eq:1.22} for the scoops $\check{h}_\iota$ and the paths $h_\iota^\pm$ defined by the abstract planar signed hemisphere complexes; see lemma~\ref{lem:6.1}.
In particular, the signed hemisphere complexes $\mathcal{C}_f^s$ of Sturm 3-ball attractors are in one-to-one correspondence with 3-template cell complexes, which are signed complexes $\mathcal{C}^s$, via the translation table \eqref{eq:1.15}.
This shows that any prescribed 3-cell template $\mathcal{C}$ can be realized as the signed hemisphere complex $\mathcal{C}_f^s$ of a Sturm 3-ball attractor $\mathcal{A}_f$.
It also shows how $\mathcal{C}$ determines $\sigma = \sigma_f$ uniquely; see theorem~\ref{thm:2.7}.
Moreover it closes the cycle of implications
	\begin{equation}
	\begin{xy}
	\xymatrix{
	 &  \text{Sturm attractor} \ar@{}[ld]^(.35){}="a"^(.65){}="b" 
	\ar@{=>} "a";"b"		&\\
    \text{Thom-Smale complex} \ar@{}[rr]^(.35){}="a"^(.65){}="b" 
	\ar@{=>} "a";"b"	  &  &  \text{meander} 
	\ar@{}[lu]^(.35){}="a"^(.65){}="b" 
	\ar@{=>} "a";"b"	}
	\end{xy}
	\label{eq:1.23}
	\end{equation}
for Sturm 3-balls.


\textbf{Acknowledgments.}
With great pleasure we express our profound gratitude to  Wal{\-}dyr~M.~Oliva, whose deep geometric insights and friendly challenges remain a visible inspiration to us since so many years.
Extended mutually delightful hospitality by the authors is mutually acknowledged.
Suggestions concerning the Thom-Smale complex were generously provided by Jean-Michel Bismut.
Gustavo~Granja has generously shared his deeply topological view point, precise references included. 
Anna~Karnauhova has contributed all illustrations with great patience, ambition, and her inimitable artistic touch.
Typesetting was expertly accomplished by Ulrike~Geiger. This work
was partially supported by DFG/Germany through SFB~647 project C8 and by FCT/Portugal through project UID/MAT/04459/2013.


\section{Signed hemispheres}
\label{sec2}

The basic tool in the proof of our main theorem~\ref{thm:1.2}, and its refinements, is a detailed analysis of the signed zero number
	\begin{equation}
	z(\varphi) = j_\pm\,,
	\label{eq:2.1}
	\end{equation}
which denotes $z(\varphi) =j$ and $\pm \varphi(0) >0$; see \eqref{eq:1.4+}.
In definition~\ref{def:2.1} below, this is used to define configurations of Sturm equilibria $v \in \mathcal{E}_f$ which we call signed hemisphere templates.
We recall how to derive the relevant information from Sturm permutations $\sigma_f$, directly and explicitly.
For independent readability later on, we also discuss Morse indices $i(v)$ and (signed) connection graphs $\mathcal{H}_f^{(s)}$, briefly.
Proposition~\ref{prop:2.2} recalls, from \cite{firo3d-1}, how signed zero numbers relate to the hemisphere decomposition by boundaries $\partial W^j$ of fast unstable manifolds $W^j$.	
In proposition~\ref{prop:2.3} we return to the planar and 3-ball cases, to summarize how the boundary label paths $h_\iota^f$ of the equilibrium orders \eqref{eq:1.5} at $x=\iota = 0,1$ traverse edges $W^u(v)$ of $i(v)=1$ saddles, faces $W^u(v)$ of $i(v)=2$ sources, and the 3-ball $W^u(\mathcal{O})$ of $i(\mathcal{O})=3$, in the Thom-Smale dynamic complex $\mathcal{C}_f$ of a Sturm 3-ball $\mathcal{A}_f$.
We compare this description with the formal definition of formal ZS-pairs and SZS-pairs $(h_0,h_1)$ in 3-ball templates.
Compare \cite[definitions~2.2, 5.1]{firo3d-1} and definitions~\ref{def:2.4}, \ref{def:2.5} below.
Noting the equivalence of proposition~\ref{prop:2.3} and definition~\ref{def:2.5}, in section~\ref{sec7}, will prove theorem~\ref{thm:2.6} which refines our main theorem~\ref{thm:1.2}:
we establish the existence of a Sturm 3-ball attractor $\mathcal{A}_f$ such that the signed Thom-Smale complex $\mathcal{C}_f^s$ of $\mathcal{A}_f$ coincides with any prescribed 3-cell template \eqref{eq:1.15}.
The equivalence is by a cell-preserving signed homeomorphism $\Phi^s$, as in \eqref{eq:hom}, which also preserves the additional sign structure.
We conclude, in theorem~\ref{thm:2.7}, by stating uniqueness of the Sturm permutation $\sigma_f$, as defined by the prescribed 3-cell template.

Let $\mathcal{A}_f$ be any Sturm global attractor.
Recall how $\mathcal{A}_f$ comes with boundary label paths $h_\iota^f$, the Sturm permutation $\sigma_f=(h_0^f)^{-1} \circ h_1^f$ and its meander $\mathcal{M}_f$, the set $\mathcal{E}_f$ of (hyperbolic) equilibria, and heteroclinic orbits $w \leadsto v$ between certain equilibria $w, v \in \mathcal{E}_f$.
We write
	\begin{equation}
	v\, \leadsto_\pm \, w\,,
	\label{eq:2.2}
	\end{equation}
if $v \leadsto w$ and $\pm(w-v)>0$ at $x=0$, respectively.
The directed \emph{connection graph} $\mathcal{H}_f$ consists of the equilibrium vertices $\mathcal{E}_f$ and directed edges $w \leadsto v$, indicating heteroclinic orbits between equilibria of adjacent Morse indices $i(w) = i(v)+1$.
Due to a cascading principle, general heteroclinic orbits $w \leadsto v$, between not necessarily adjacent Morse levels $i$, are equivalently represented by di-paths in $\mathcal{H}_f$; see \cite{brfi89, firo96} and the summary in \cite{firo3d-1}.
The \emph{signed connection graph} $\mathcal{H}_f^s$, analogously, features signed directed edges $\leadsto_\pm$, instead.

Fix any unstable equilibrium $v\in \mathcal{E}_f$, with Morse index $i(v)= \dim W^u(v) >0$.
We decompose the heteroclinic targets $v\leadsto w$ according to their signed zero number \eqref{eq:2.1} as
	\begin{equation}
	\begin{aligned}
	\mathcal{E}_\pm^j(v)
	&:= \{ w \in \mathcal{E}_f\,|\, v\leadsto_\pm w\,,\, z(w-v)=j\}\\
	&\phantom{:}=\{ w \in \mathcal{E}_f\,|\, v\leadsto w\,,\, z(w-v)=j_\pm\}\,.
	\end{aligned}
	\label{eq:2.3}
	\end{equation}
Here $0 \leq j < i(v)$, because $j=z(u-v) < i(v)$ for all $u \in \text{clos } W^u(v)\smallsetminus\{v\}$; see \cite{brfi86}.

\begin{defi}\label{def:2.1}
We call the partitions $\mathcal{E}_\pm^j(v)$, $0 \leq j < i(v)$, of the equilibria $w \in \partial W^u(v)$, the \emph{signed hemisphere template} of the Sturm attractor $\mathcal{A}_f$.

In the special case of a Sturm 3-ball $\mathcal{A}_f$ we call these partitions the \emph{signed 2-hemisphere template}.
\end{defi}

The relevant Morse and Sturm data $i(v)$ and $z(w-v)$ can easily be derived, explicitly, from the labeling paths $h_\iota = h_\iota^f \in S_N$ in \eqref{eq:1.5} and the Sturm permutation $\sigma= \sigma_f = (h_0^f)^{-1} \circ h_1^f$, as follows.
Recursively, the Morse numbers $i_v$, $v\in \mathcal{E}_f$ have been defined in \eqref{eq:1.16}.
Then \cite{furo91} have shown that
	\begin{equation}
	i(v) = i_v
	\label{eq:2.3b}
	\end{equation}
for all $v \in \mathcal{E}_f$.
Similarly, define the zero numbers $z_{v_1v_2}$ for $v_1, v_2 \in \mathcal{E}_f$, recursively, as
	\begin{equation}
	\begin{aligned}
	z_{vv} &:= i(v)\,,\\
	z_{h_0(j+1)h_0(k)} &:= z_{h_0(j)h_0(k)} + \tfrac{1}{2}(-1)^{j+1}
	\cdot\\
	&\phantom{: }\cdot
	\left[ \text{sign}\left(\sigma^{-1}(j+1)-
	\sigma^{-1}(k)\right)- 
	\text{sign}\left(\sigma^{-1}(j)-\sigma^{-1}(k)\right)\right]\,.
	\end{aligned}
	\label{eq:2.3c}
	\end{equation}
Then \cite{ro91,firo96} have shown that 
	\begin{equation}
	z(w-v) =z_{wv}\,.
	\label{eq:2.3d}
	\end{equation}
for equilibria $w \neq v$. The signed version of \eqref{eq:2.3d} follows easily from $\text{sign}(h_0^{-1}(w) -h_0^{-1}(v))$.

Definition~\ref{def:2.1} in fact provides partitions of the equilibria $w \in \partial W^u(v)$, with the exception of those $w$ which are never the target of any heteroclinic orbit $v \leadsto w$ from some equilibrium $v$ with higher Morse index $i(v) >i(w)$.
In the case of signed 2-hemisphere templates, this only excludes the 3-ball equilibrium $w=\mathcal{O}$ with $i(w)=3$.
To see this we invoke the Morse-Smale property again; see section~\ref{sec1}.
Indeed all equilibria $w \in \partial W^u(v)$ are then targets of heteroclinic orbits $v \leadsto w$.
This shows the equivalence of the connection graph $\mathcal{H}_f$ with the incidence relations,
	\begin{equation}
	v \leadsto w 
	\quad \Longleftrightarrow \quad 
	c_w \subseteq \partial c_v\,.
	\label{eq:2.4}
	\end{equation}
in the Sturm complex of cells $c_w \subseteq \partial c_v$.
For example, any equilibrium $w \neq \mathcal{O}$ satisfies $w \in \partial W^u(\mathcal{O})$, and is therefore the target of a heteroclinic orbit $\mathcal{O} \leadsto w$.

Our definition~\ref{def:2.1} of signed hemisphere templates differs slightly from the corresponding notion in \cite[definition~1.1]{firo3d-1}.
To clarify this point we have to recall first how the Schoenflies result \cite{firo13} provides a disjoint hemisphere decomposition
	\begin{equation}
	\partial W^u(v) = \bigcup\limits_{0 \leq j < i(v)}
	\Sigma _{\pm}^j
	\label{eq:2.5}
	\end{equation}
of the topological boundary $\partial W^u$:= $ \text{clos } W^u(v) \smallsetminus W^u(v)$ of the unstable manifold $W^u(v)$, for any hyperbolic equilibrium $v$.
The construction of the disjoint hemispheres $\Sigma_{\pm}^j = \Sigma_{\pm}^j(v)$ can be summarized as follows.
For $1 \leq j \leq i(v)$, let $W^j$ denote the $j$-dimensional fast unstable manifold of $v$.
The tangent space to $W^j$ at $v$ is spanned by the eigenfunctions $\varphi_0, \ldots , \varphi_{j-1}$ of the linearization of \eqref{eq:1.3} at $v$, for the first $j$ eigenvalues $\lambda_0 > \ldots > \lambda_{j-1} >0$.
Consider any orbit $u(t, \cdot) \in W^{j+1} \smallsetminus W^j$, $t \in \mathbb{R}$.
Then
	\begin{equation}
	\lim_{t \rightarrow -\infty} \left( u \left( t,\cdot\right) -v 				\right)\, 	/\, |u\left(t, \cdot\right) -v| = \pm \varphi_j\,; 
	\label{eq:2.6}
	\end{equation}		
by normalization of $\varphi_j$ in the appropriate norm of the phase space $X \hookrightarrow C^1$.
Here and below we fix signs such that $\varphi_j(0) >0$.
In particular, the signed zero number $z$ of \eqref{eq:1.4} satisfies
	\begin{equation}
	\lim_{t \rightarrow - \infty} z\left( u \left( t, \cdot\right) -v 			\right) = z(\pm\varphi_j) = j_\pm\,.
	\label{eq:2.7}
	\end{equation}
See \cite{brfi86} for further details on the construction of $W^j$.

The \emph{signed hemispheres} $\Sigma_\pm^j$ are defined, recursively, by the disjoint unions
	\begin{equation}
	\Sigma^j:= \partial W^{j+1} = \Sigma_-^j \, \cup \, \Sigma_+^j 			\, \cup\,  \Sigma^{j-1}\,,
	\label{eq:2.8}
	\end{equation}
for $0 \leq j < i(v)$, with the convention $\Sigma ^{-1}$:= $ \emptyset$.
The hemisphere closures
	\begin{equation}
	\text{clos } \Sigma_\pm^j = \Sigma_\pm^j \, \dot{\cup}\, 						\Sigma^{j-1}
	\label{eq:2.9}
	\end{equation}
can be obtained as $\omega$-limit sets of protocap hemispheres which are $C^1$-small, nearly parallel, perturbations of $\text{clos } W^j$ in $\text{clos } W^{j+1}$, in the eigendirections $\pm \varphi_j$, respectively.
In particular \eqref{eq:2.6}, \eqref{eq:2.7} hold in the interior of the protocaps, and for any heteroclinic orbit $v \leadsto w \in \Sigma_\pm^j$.
See \cite{firo13} for complete details.

The following proposition was proved in \cite[proposition~3.1]{firo3d-1}, again with the abbreviations $\Sigma_\pm^j = \Sigma_\pm^j(v)$.

\begin{prop}\label{prop:2.2}
With the above notation the following statements hold true for equilibria $v,w,w_1, w_2,$ and all $0 \leq j < i(v)$:
\begin{align}
w \in \Sigma^j 
\quad &\Longrightarrow \quad i(w) \leq j\;\tag{i}\\
w \in \Sigma^j 
\quad &\Longrightarrow \quad z(w-v) \leq j\;\tag{ii}\\
w \in \Sigma_\pm^j 
\quad &\Longrightarrow \quad z(w-v) =j_\pm\;\tag{iii}\\
w_1, w_2 \in \text{clos } \Sigma_+^j  \text{ or } 
w_1,w_2 \in \text{clos } \Sigma_-^j
\quad &\Longrightarrow \quad z(w_1-w_2) \leq j-1\,.\tag{iv}
\end{align}
\end{prop}

In \cite[definition~1.1]{firo3d-1} the sets $\mathcal{E}_\pm^j(v)$ of the signed hemisphere templates \eqref{eq:2.3} had been defined as
	\begin{equation}
	\Tilde{\mathcal{E}}_\pm^j(v) := \mathcal{E}_f \cap \Sigma_\pm^j(v)\,,
	\label{eq:2.10}
	\end{equation}
instead.
By proposition~\ref{prop:2.2}(iii), the sets $\mathcal{E}_\pm^j(v)$ and $\Tilde{\mathcal{E}}_\pm^j(v)$ coincide, for each $j$.

Conversely, we can describe the signed hemispheres $\Sigma_\pm^j(v)$ directly, via the signed hemisphere template \eqref{eq:2.3} of equilibrium sets $\mathcal{E}_\pm^j(v)$.
Indeed \eqref{eq:1.13} now reads
	\begin{equation}
	\Sigma_\pm^j (v) = \bigcup\limits_{w \in \mathcal{E}_\pm^j(v)}
	W^u(w)\,.
	\label{eq:2.10a}
	\end{equation}
This allows us to define a \emph{signed Sturm complex} $\mathcal{C}_f^s$, as a refinement of the Sturm complex $\mathcal{C}_f$ with (regular) Thom-Smale cells $W^u(v)$, $v \in \mathcal{E}_f$.
We simply keep track, in $\mathcal{C}_f^s$, which cells $W^u(w)$ of $\mathcal{C}_f$ belong to which hemisphere $\Sigma_\pm^j(v)$ in the signed hemisphere decomposition of $\mathcal{C}_f$.

We now focus on the case of a Sturm 3-ball $\mathcal{A}_f$.
Our next proposition describes, in terms of the dynamic cell decompositions of $\mathcal{A}_f$ by the Thom-Smale cells $c_v \in W^u(v)$, how the labeling bijections $h_\iota^f$, $\iota=0,1$, traverse each cell.
Let $0<n$:= $i(v) \leq 3$ be the Morse index of $v$.
For fixed $n$, consider sequences $\mathbf{s} = s_0 \ldots s_{n-1}$ of $n$ symbols $s_i \in \left\lbrace \pm \right\rbrace$. In fact, let us restrict to the four cases of constant and alternating sequences of signs $s_i$. For any such prescribed sequence $\mathbf{s} = s_0 \ldots s_{n-1}$ let $w=w(\mathbf{s}) \in \Sigma_{s_{n-1}}^{n-1}(v)$ denote the unique equilibrium such that $v\leadsto w $ starts a heteroclinic cascade
	\begin{equation}
	v \leadsto v_{n-1} \leadsto \ldots \leadsto v_0
	\label{eq:2.12}
	\end{equation}
with $w=v_{n-1}$ and $v_i \in \Sigma_{s_i}^i(v)$ of descending Morse indices $i(v_i)=i=n-1, \ldots , 0$.
Equivalently, by \eqref{eq:2.4}, we may express the same definition on the level of Thom-Smale cells as
	\begin{equation}
	c_{v_i} \subseteq \partial c_{v_{i+1}} \cap \Sigma_{s_i}^i(v)\,,
	\label{eq:2.11}
	\end{equation}
with $v=v_n,\ w=w_{n-1}$, and $v_i \in \Sigma_{s_i}^i(v)$ of ascending Morse indices $i(v_i)=i=0, \ldots , n-1$. 

Again, we do not claim existence of $w(\mathbf{s})$ except in the four cases of constant and alternating signs $s_i$.
Uniqueness of $w(\mathbf{s})$, for given symbol sequence $\mathbf{s}$, can be proved by induction on $n$. For some $v$, however, certain equilibria $w(\mathbf{s})$ with different symbol sequences may happen to coincide.

\begin{prop}\label{prop:2.3}
Fix $0< n$:= $i(v) \leq 3$, $\iota =0,1$, and assume $v \in \mathcal{E}_f$ is not already directly preceded, or directly followed, by an equilibrium of higher Morse index than $n=i(v)$, along the labeling bijection $h_\iota^f: \lbrace 1, \ldots , N \rbrace \rightarrow \mathcal{E}_f$.
Then the unclaimed parts of $h_\iota^f$ through $v$ follow the template table

\begin{center}
\begin{tabular}{|c||c|c|}
\hline
	&	$h_0^f$	&	$h_1^f$
\\ \hline \hline
$n=i(v)=1$	&	$\ldots w(-)\,v\,w(+) \ldots$	&	$\ldots w(-)\,v\,w(+) \ldots$	\\ \hline
$n=i(v)=2$	&	$\ldots w(+-)\,v\,w(-+) \ldots$	&	$\ldots w(++)\,v\,w(--) \ldots$
\\ \hline
$n=i(v)=3$	&	$\ldots w(-+-)\,v\,w(+-+) \ldots$	&	$\ldots w(---)\,v\,w(+++) \ldots$
\\ \hline
\end{tabular}
\end{center}
\end{prop}

\begin{proof}[\textbf{Proof.}]
By adjacency \eqref{eq:1.17}, \eqref{eq:1.18} of Morse indices for $h_\iota$-adjacent equilibria, we only have to consider the case $i(w) = n-1$, $i(v) = n$ for the $w$-entries in the table.
In particular, the unique heteroclinic orbits $u(t)$:= $v\leadsto w$ imply $z(w-v) = (n-1)_{s_{n-1}}$ with
	\begin{equation}
	s_{n-1}=
	\left\lbrace 
	\begin{aligned}
	&\text{sign}\left(h_0(w)-h_0(v)\right)\,, 
	\qquad &\text{for}\quad \iota=0\,,\\
	&(-1)^{n-1} \text{ sign}\left(h_1(w)-h_1(v)\right)\,,
	\qquad &\text{for}\quad \iota=1\,.	
	\end{aligned}
	\right.
	\label{eq:2.13}
	\end{equation}
This fixes the last entries $s_{n-1}$ in the arguments $\mathbf{s}$ of $w(\mathbf{s})$ in the table, and takes care of the trivial case $n=1$.

For $n=2$ let $w$ denote the direct $h_0^f$-successor of $v$.
We may assume $i(w)=i(v)-1=1$, or else nothing has been claimed.
Hence $w \in \Sigma_+^1(v)$.
We have to show $w=w(-+)$, i.e. $\mathbf{N}(v)$:= $\Sigma_-^0(v) = \Sigma_-^0(w)$=: $\mathbf{N}(w)$.
Suppose, indirectly, that $\mathbf{N}(w) \neq \mathbf{N}(v)$.
Then
	\begin{equation}
	\mathbf{N}(v) < \mathbf{N}(w) < w\,.
	\label{eq:2.14}
	\end{equation}
Indeed, the right inequality holds by definition, for all $0 \leq x \leq 1$.
Moreover $w \in \Sigma_+^1(v)$ implies $\mathbf{N}(w) \in \text{clos }\Sigma_+^1(v) \geq \mathbf{N}(v)$, by invariance.
Hence $\mathbf{N}(w)\neq \mathbf{N}(v)$ implies the left inequality of \eqref{eq:2.14}.
Because $w$ is the direct $h_0^f$-successor of $v \neq \mathbf{N}(w)$, we can also conclude	
	\begin{equation}
	\mathbf{N}(w) < v
	\label{eq:2.15}
	\end{equation}
at $x=0$.
Since $w \in \Sigma_+^1(v)$ implies $z(w-v) =1_+$, the same inequality \eqref{eq:2.15} holds at $x=1$, because $\mathbf{N}(w)<w<v$ there.
Since $\mathbf{N}(w) \in \text{clos } \Sigma_+^1(v) \subseteq \Sigma^1(v)$ implies $z(\mathbf{N}(w)-v) \leq 1$, by proposition~\ref{prop:2.2}(ii), we conclude that \eqref{eq:2.15} holds for all $0 \leq x \leq 1$.
But then $z$-dropping \eqref{eq:1.4} and $\mathbf{N}(w) > \mathbf{N}(v)$ block the heteroclinic orbit $u(t,\cdot)$: $v\leadsto \mathbf{N}(v) = \Sigma_-^0(v)$, which exists by definition.
Indeed $z(u(t,\cdot)-\mathbf{N}(w)) = 0_\pm$ for large $\pm t > 0$ would have to drop below zero when $u(t_0, \cdot)=\mathbf{N}(w)$ at the Neumann boundary $x=0$.
This contradiction shows $\Sigma_-^0(w) = \mathbf{N}(w) = \mathbf{N}(v)=\Sigma_-^0(v)$ and hence confirms $w=w(-+)$.
The remaining cases for $n=i(v)=2$ are omitted because they are analogous, thanks to  the four trivial equivalences generated by $x\mapsto1-x$ and $u\mapsto -u$; see our introduction and \cite[definition~2.3]{firo3d-1} .

The above idea of blocking heteroclinic orbits by elementary arguments on $z$-dropping goes back to \cite{brfi88, brfi89}.
For a refined version due to Wolfrum see lemma~\ref{lem:5.2} below.

It remains to address the case $n=i(v)=3$, i.e., $v= \mathcal{O}$.
The four trivial equivalences, again, reduce the problem to showing that $w(-+-)$ is the $h_0$-predecessor of $\mathcal{O}$.
We invoke \cite[theorem~4.1]{firo3d-1}.
There, it was shown that the Thom-Smale complex $c_v= W^u (v)$ of any Sturm 3-ball is in fact a 3-cell template with the translation table \eqref{eq:1.15} between signed hemispheres and geographic terminology.
In fig.~\ref{fig:1.4} of the general 3-cell template, this identifies $w(-+-)$ as the source
	\begin{equation}
	w_-^0 = w(-+-)
	\label{eq:2.16}
	\end{equation}
of the face $\mathbf{NE}$.
Indeed $\mathbf{NE}$ is the unique face of $\mathbf{W}= \Sigma_-^2(\mathcal{O})$, which is adjacent to the unique 1-cell $W^u(v_1)$ of $\mathbf{WE}=\Sigma_+^1(\mathcal{O})$ which, in turn, is itself adjacent to the unique 0-cell $v_0=\mathbf{N}= \Sigma_-^0(\mathcal{O})$. See \eqref{eq:2.11} and \eqref{eq:2.12}.

We show that $\mathcal{O}$ is the direct $h_0$-successor of $w_-^0$.
We first claim
	\begin{equation}
	\Sigma_+^1(w_-^0) \subseteq \mathbf{WE}= \Sigma_+^1(\mathcal{O})\,.
	\label{eq:2.17}
	\end{equation}
By definition~\ref{def:1.1}(iii), the non-meridian edges of the cell boundary $\partial c_{w_-^0}$ are oriented towards the unique boundary minimum $\Sigma_+^0(w_-^0) \subseteq \mathbf{WE} \cup \mathbf{S}$.
Hence one of the hemisphere boundaries $\Sigma_\pm^1(w_-^0)$ must be entirely contained in the meridian $\mathbf{WE}=\Sigma_+^1(\mathcal{O})$. For the boundary $\Sigma_-^1(w_-^0)$, this is impossible because $w_-^0 \in \Sigma_-^2(\mathcal{O})$ implies $\Sigma_-^1(w_-^0) < \mathcal{O}$ at $x=0$, rather than $\Sigma_-^1(w_-^0) > \mathcal{O}$.
This proves claim \eqref{eq:2.17}.

Next suppose, indirectly, that $\mathcal{O}$ is not the direct $h_0$-successor of $w_-^0$.
Then the current proposition applies to $v$:= $w_-^0$ with Morse index $i(v)=2$.
This identifies the direct $h_0$-successor of $v=w_-^0$ to be the unique equilibrium $w$ with 1-cell $c_w \subseteq \Sigma_+^1(w_-^0)$ adjacent to $\Sigma_-^0(w_-^0) = \mathbf{N}$.
By \eqref{eq:2.17} this implies
	\begin{equation}
	w \in \Sigma_+^1(w_-^0) \subseteq \Sigma_+^1(\mathcal{O})\,.
	\label{eq:2.18}
	\end{equation}
Evaluation at $x=0$, in proposition~\ref{prop:2.2}(iii), provides the right inequality of
	\begin{equation}
	w_-^0 < \mathcal{O} < w\,,
	\label{eq:2.19}
	\end{equation}
at $x=0$.
Likewise, the left inequality at $x=0$ follows from $w_-^0 =w(-+-) \in \Sigma_-^2(\mathcal{O})$; see \eqref{eq:2.16}.
Therefore $w$ cannot be the direct $h_0$-successor of $w_-^0$.

This contradiction to the definition of $w$ proves the proposition.		
\end{proof}

\begin{figure}[t!]
\centering \includegraphics[width=0.43\textwidth]{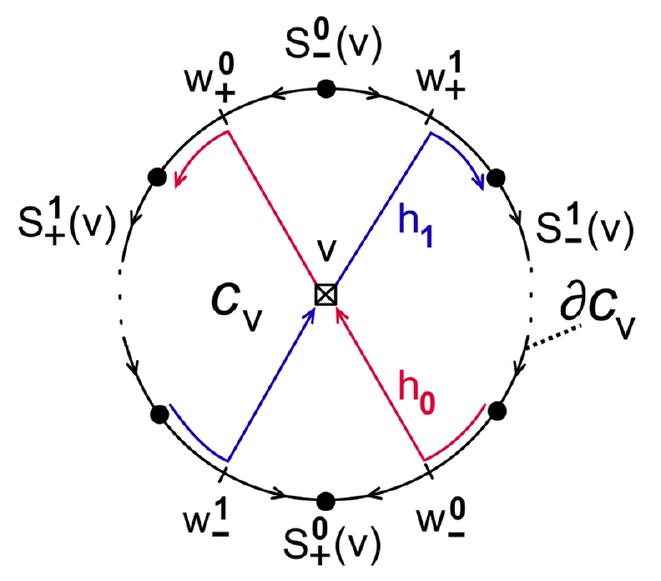}
\caption{\emph{
Traversing a face vertex $v$ by a ZS-pair $(h_0, h_1)$.
Note the resulting shapes ``Z'' of $h_0$ (red) and ``S'' of $h_1$ (blue).
The paths $h_\iota$ may also continue into neighboring faces, beyond $w_\pm^\iota$, without turning into the face boundary $\partial c$.
}}
\label{fig:2.1}
\end{figure}

In \cite[definitions~2.2 and 5.1]{firo3d-1} we have introduced formal ZS-pairs and SZS-pairs of paths $(h_0,h_1)$ associated to bipolar planar cell complexes and 3-cell templates, respectively.
In the following sections we will see how these formal recipes coincide, precisely, with the template table of proposition~\ref{prop:2.3}, for the traversals of the Sturm paths $(h_0^f, h_1^f)$ through the dynamic cells $c_v=W^u(v)$ of planar and 3-ball Sturm attractors, in terms of their signed hemisphere decompositions.

More precisely, let us first recall \cite[definition~2.2]{firo3d-1}; see also \cite{firo09}.
Let $\mathcal{C} = \bigcup_{v\in\mathcal{E}} c_v \subseteq \mathbb{R}^2$ be a finite regular planar cell complex with boundary bipolar orientation of the 1-skeleton $\mathcal{C}^1$.
Let $v$ indicate any source, i.e. the barycenter of a 2-cell face $c_v$ in $\mathcal{C}$.
By planarity of $\mathcal{C}$ it turns out that the bipolar orientation of $\mathcal{C}^1$ defines unique orientation extrema on the boundary circle $\partial c_v$ of the 2-cell $c_v$.
Let $w_-^0$ be the barycenter on $\partial c_v$ of the edge to the right of the minimum, and $w_+^0$ the edge barycenter to the left of the maximum.
Similarly, let $w_-^1$ be the edge barycenter to the left of the minimum, and $w_+^1$ to the right of the maximum.
See fig.~\ref{fig:2.1}.

\begin{figure}[t!]
\centering \includegraphics[width=\textwidth]{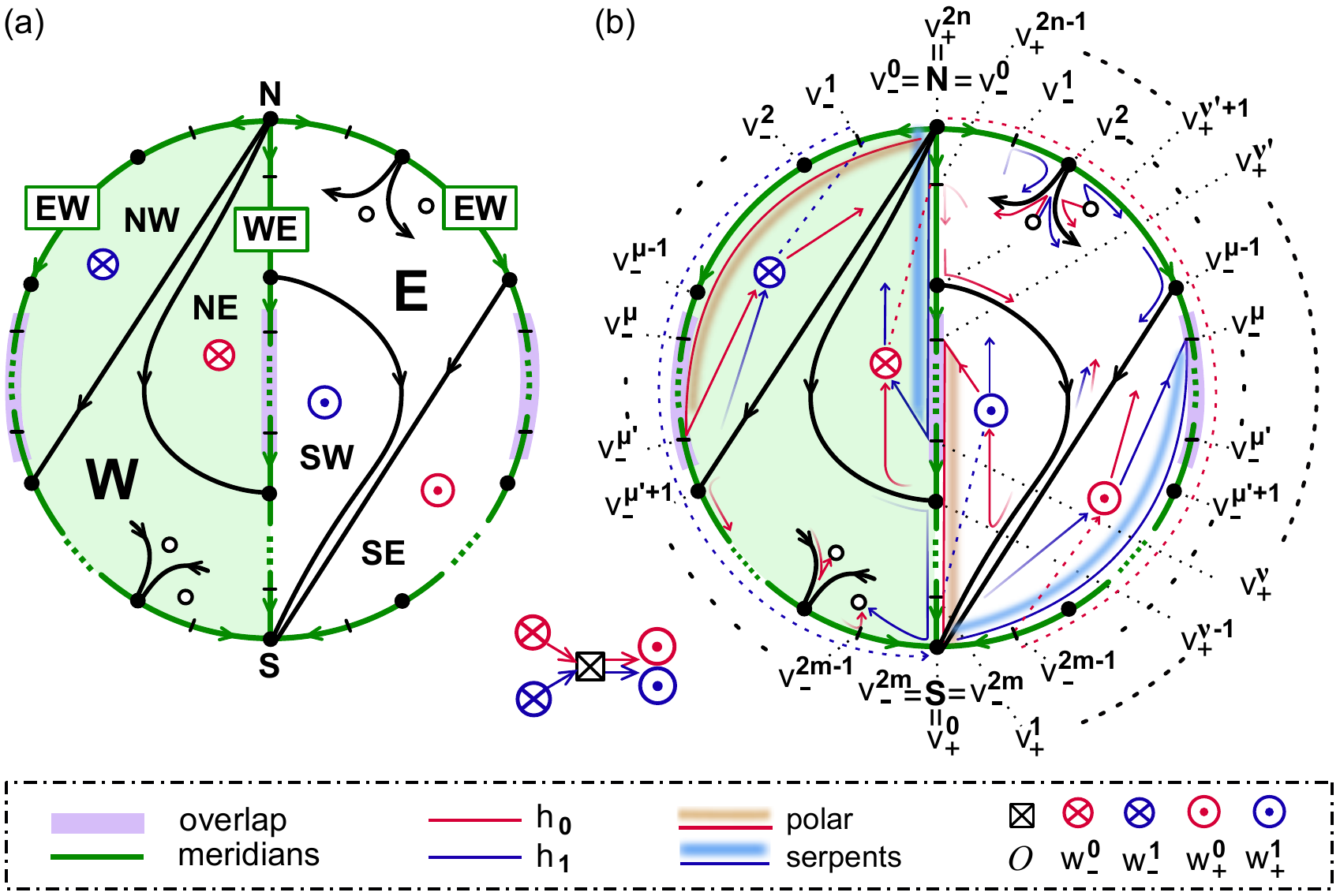}
\caption{\emph{
The SZS-pair $(h_0,h_1)$ in a 3-cell template $\mathcal{C}$, with poles $\mathbf{N}, \mathbf{S}$, hemispheres $\mathbf{W}, \mathbf{E}$  and meridians $\mathbf{EW}, \mathbf{WE}$.
Left, (a): schematics of the 3-cell template, as in fig.~\ref{fig:1.4}.
Right, (b): schematics of the SZS-pair.
Dashed lines indicate the $h_\iota$-ordering of vertices in the closed hemisphere, when $\mathcal{O}$ and the other hemisphere are ignored, according to definition~\ref{def:2.5}(i).
The actual paths $h_\iota$ tunnel, from $w_ -^\iota \in  \mathbf{W}$ through the 3-cell barycenter $\mathcal{O}$, and re-emerge at $w_+^\iota \in  \mathbf{E}$, respectively.
Note the boundary overlap of the faces $\mathbf{NW}, \mathbf{SE}$ of $w_-^1, w_+^0$ from $v_-^{\mu-1}$ to $v_-^{\mu ' +1}$ on the $\mathbf{EW}$ meridian.
Similarly, the boundaries of the faces $\mathbf{NE}, \mathbf{SW}$ of $w_-^0, w_+^1$ overlap from $v_+^{\nu -1}$ to $v_+^{\nu ' +1}$ along $\mathbf{WE}$.
}}
\label{fig:2.2}
\end{figure}

\begin{defi}\label{def:2.4}
The paths of labeling bijections $h_0, h_1$: $\lbrace 1, \ldots , N \rbrace \rightarrow \mathcal{E}$ are called a \emph{ZS-pair} $(h_0, h_1)$ in the finite, regular, planar and bipolar cell complex $\mathcal{C} = \bigcup_{v \in \mathcal{E}} c_v$ if the following three conditions all hold true:

\begin{itemize}
\item[(i)] $h_0$ traverses any face $c_v$ as $\ldots w_-^0 vw_+^0 \ldots$
\item[(ii)] $h_1$ traverses any face $c_v$ as $\ldots w_-^1 vw_+^1 \ldots$
\item[(iii)] both $h_\iota$ follow the bipolar orientation of the 1-skeleton $\mathcal{C}^1$, if not already defined by (i), (ii).
\end{itemize}

We call $(h_0,h_1)$ an \emph{SZ-pair}, if $(h_1, h_0)$ is a ZS-pair, i.e. if the roles of $h_0$ and $h_1$ in the rules (i) and (ii) of the face traversals are reversed.
\end{defi}

This definition enters the variant of unique SZS-Pairs $(h_0,h_1)$, \cite[definition~5.1]{firo3d-1}, associated to 3-cell templates, as follows.
See fig.~\ref{fig:2.2} for an illustration.

\begin{defi}\label{def:2.5}
Let $\mathcal{C} = \bigcup_{v \in \mathcal{E}} c_v$ be a 3-cell template with oriented 1-skeleton $\mathcal{C}^1$, poles $\mathbf{N}, \mathbf{S}$, hemispheres $\mathbf{W}, \mathbf{E}$, and meridians $\mathbf{EW}$, $\mathbf{WE}$.
A path pair $(h_0, h_1)$ of labeling bijections $h_\iota$: $ \lbrace 1, \ldots , N \rbrace \rightarrow \mathcal{E}$ is called the SZS-\emph{pair assigned to} $\mathcal{C}$ if the following two conditions hold.
\begin{itemize}
\item[(i)] The restrictions of range $h_\iota$ to $\text{clos } \mathbf{W}$ form an SZ-pair $(h_0, h_1)$, in the closed Western hemisphere.
The analogous restrictions form a ZS-pair $(h_0,h_1)$ in the closed Eastern hemisphere $\text{clos } \mathbf{E}$.
See definition~\ref{def:2.4}.
\item[(ii)] In the notation of fig.~\ref{fig:2.2}, and for each $\iota =0,1$, the paths $h_\iota$ traverse $\mathcal{O}$ in the orders $\ldots w_-^\iota \mathcal{O} w_+^\iota \ldots$, respectively.
\end{itemize}
\end{defi}

In \cite[theorem~5.2]{firo3d-1} we have show that the permutation
	\begin{equation}
	\sigma:= h_0^{-1} \circ h_1
	\label{eq:2.20}
	\end{equation}
associated to the SZS-pair $(h_0,h_1)$ of any 3-cell template $\mathcal{C}$ is a Sturm meander, i.e. $\sigma$ is a dissipative Morse meander in the sense of \cite{firo96}.
In particular there exists a dissipative nonlinearity $f$ with hyperbolic equilibria in \eqref{eq:1.1}, such that the Sturm permutation $\sigma_f$ coincides with the formal permutation $\sigma$ associated to (the SZS-pair $(h_0,h_1)$ of) the arbitrarily prescribed 3-cell template $\mathcal{C}$:
	\begin{equation}
	\sigma_f = \sigma \,.
	\label{eq:2.21}
	\end{equation}
Moreover, $\sigma_f$ comes with the associated Sturm global attractor $\mathcal{A}_f$, equilibria $\mathcal{E}_f$ and the Thom-Smale regular cell complex $\mathcal{C}_f = \bigcup_{v\in \mathcal{E}_f} c_v$, $c_v = W^u(v)$; see \eqref{eq:1.7a}.

Roughly speaking our main theorem~\ref{thm:1.2} claims
	\begin{equation}
	\mathcal{C}_f = \mathcal{C}\,,
	\label{eq:2.22}
	\end{equation}
by a cell preserving homeomorphism \eqref{eq:hom}.
To refine this statement, in view of the signed hemisphere decompositions \eqref{eq:2.3}, \eqref{eq:2.10} of the equilibria $\mathcal{E}_f$ into $\mathcal{E}_\pm^j(v)$, and of the sphere boundaries $\Sigma^j(v) = \partial W^{j+1}(v)$ into signed hemispheres $\Sigma_\pm^j(v)$, we now define a \emph{formal hemisphere decomposition} $\mathcal{C}^s$ on any 3-cell template $\mathcal{C}$.
Let $c_v$ denote any cell of $\mathcal{C}$, with dimension $i_v = \dim c_v >0$.
If $i_v=3$, i.e. for $v=\mathcal{O}$, we define the formal hemispheres $S_\pm^j(\mathcal{O})$, $j=0, \ldots ,3$, analogously to the hemispheres $\Sigma_\pm^j(\mathcal{O})$ in the translation table \eqref{eq:1.15}.
If $i_v=1$, i.e. for edge saddles $v$, we define $S_+^0(v)$ as the head vertex and $S_-^0(v)$ as the tail vertex of the edge $c_v$ under the bipolar orientation of the 1-skeleton $\mathcal{C}^1$.
For $i_v=2$ faces, we define $S_-^0(v)$ as the max and $S_+^0(v)$ as the min vertex on the circle boundary $\partial c_v \subseteq \mathcal{C}^1$, under its (downward) bipolar orientation; see fig.~\ref{fig:2.1}.
For face sources $v\in \mathbf{E} =: S_+^2(\mathcal{O})$, we define the remaining right part of the boundary $\partial c_v$ as $S_-^1(v)$, and the left part as $S_+^1(v)$.
For $v\in \mathbf{W} =: S_-^2(\mathcal{O})$, we flip these sides of $S_\pm^1(v)$, so that $S_-^1(v)$ is left and $S_+^1(v)$ right.
In summary, the formal hemisphere decomposition $\mathcal{C}^s$ of $\mathcal{C}$ consists of $\mathcal{C}$ itself, together with the sign information on
	\begin{equation}
	c_w\subseteq S_\pm^j(v)\,, \quad \text{for any} \quad
	c_w \subseteq \partial c_v\,, \, 0\leq j<i_v\,.
	\label{eq:2.23}
	\end{equation}

In other words, definition~\ref{def:1.1}(ii) of bipolarity and meridians in a 3-cell template $\mathcal{C}$ is equivalent to the definition of a formal hemisphere decomposition $\mathcal{C}^s$.
The translation table for the hemispheres $S_\pm^j(\mathcal{O})$ is completely analogous to \eqref{eq:1.15} with the identification
	\begin{equation}
	S_\pm^j(\mathcal{O}) = \Sigma_\pm^j(\mathcal{O})\,,
	\label{eq:2.23a}
	\end{equation}		
$0 \leq j \leq 2$.
In particular \cite[theorem~4.1]{firo3d-1} has already identified the dynamic Sturm complex $\mathcal{C}_f$ associated to any signed 2-hemisphere template $\mathcal{E}_\pm^j(v)$, $v\in \mathcal{E}_f$, as a 3-cell template $\mathcal{C}^s$ with formal hemisphere decomposition $S_\pm^j(v)$ given by bipolarity, the meridians, and the identification \eqref{eq:2.23a}.
The following theorem addresses the converse of this construction.

\begin{thm}\label{thm:2.6}
Let $\mathcal{C}^s = \bigcup_{v\in \mathcal{E}}$ be a 3-cell template with associated formal hemisphere decomposition $S_\pm^j(v)$ as in \eqref{eq:2.23} above.
Let $\mathcal{C}_f$ be the Sturm dynamic complex \eqref{eq:1.7a} associated to $\mathcal{C}^s$, by the above construction \eqref{eq:2.20}, \eqref{eq:2.21} of an SZS-pair $(h_0,h_1)$.
Let $\Sigma_\pm^j(v)$, $v \in \mathcal{E}_f$, $0\leq j<i(v)$ be the signed hemisphere decomposition \eqref{eq:2.5} on $\mathcal{C}_f$.

Then there exists a cell-preserving \emph{homeomorphism}
	\begin{equation}
	\Phi^s: \qquad \mathcal{C}^s = \bigcup\limits_{v \in \mathcal{E}} c_v
	\quad \longrightarrow \quad 
	\bigcup\limits_{v \in \mathcal{E}_f} W^u(v)
	= \mathcal{C}_f^s = \mathcal{A}_f
	\label{eq:homs}
	\end{equation}
with $\Phi^s (c_v) = W^u(\Phi^s(v))$.
Moreover $\Phi^s$ is \emph{signed}, i.e. $\Phi^s$ also preserves the signed hemisphere structure
	\begin{equation}
	\Phi^s\left(S_\delta^j(v)\right) =
	\Sigma_\delta^j\left(\Phi^s(v)\right)\,,
	\label{eq:2.24}
	\end{equation}
for all $v\in \mathcal{E}$, $0\leq j<i_v$, and $\delta = \pm$.

In short the SZS-pair $(h_0,h_1)$ designs a Sturm global attractor $\mathcal{A}_f$ such that the Thom-Smale complex $\mathcal{C}_f^s$ coincides with the given 3-cell template $\mathcal{C}^s$, including the signed hemisphere structure.
\end{thm}

Along the proof of the signed realization theorem~\ref{thm:2.6}, we can also settle the longstanding puzzle on different, not even conjugate, Sturm permutations $\sigma_f, \sigma_g$ with apparently equivalent Sturm attractors $\mathcal{A}_f=\mathcal{A}_g$ -- at least for Sturm 3-balls, and hence also for planar attractors.

\begin{thm}\label{thm:2.7}
Let $\mathcal{C}_f$ and $\mathcal{C}_g$ be two Sturm 3-ball dynamic complexes, alias 3-cell templates.
Assume there exists a cell-preserving homeomorphism
	\begin{equation}
	\Phi^s: \qquad \mathcal{A}_f=\mathcal{C}_f^s 
	= \bigcup\limits_{v \in \mathcal{E}_f} W_f^u(v)
	\quad \longrightarrow \quad 
	\bigcup\limits_{v \in \mathcal{E}_g} W_g^u(v)
	= \mathcal{C}_g^s = \mathcal{A}_g\,,
	\label{eq:homsfg}
	\end{equation}
with $\Phi^s(W_f^u(v)) = W_g^u(\Phi^s(v))$.
Assume $\Phi^s$ is \emph{signed}, i.e. $\Phi^s$ also preserves the signed hemisphere decompositions
	\begin{equation}
	\Phi^s\left(\Sigma_\delta^j(v)\right) =
	\Sigma_\delta^j\left(\Phi^s(v)\right)\,,
	\label{eq:2.25}
	\end{equation}
for all $v\in \mathcal{E}_f$, $0\leq j<i(v)$, and $\delta = \pm$.

Then the Sturm permutations of $f$ and $g$ coincide:
	\begin{equation}
	\sigma_f = \sigma_g\,.
	\label{eq:2.26}
	\end{equation}
Moreover, $\Phi^s$ can be chosen to respect all fast unstable manifolds,
	\begin{equation}
	\Phi^s\left(W_f^{j+1}(v)\right)=
	W_g^{j+1}\left(\Phi^s(v)\right)\,,
	\label{eq:2.27}
	\end{equation}
$0 \leq j<i(v)$, together with their signed versions.
\end{thm}

For an example we refer to \cite[(5.6) and fig.~5.2]{firo3d-1}.
Any cell-preserving homeomorphism $\Phi$, in that example, would have to interchange the respective 2-hemispheres of $\mathcal{O}$,
	\begin{equation}
	\Phi: \quad \Sigma_{f,\pm}^2(\mathcal{O}_f) 
	\quad \longrightarrow\quad
	\Sigma_{g, \mp}^2(\mathcal{O}_g)\,.
	\label{eq:2.28}
	\end{equation}
This accounts for different Sturm permutations $\sigma_f \neq \sigma_g$, which are not trivially conjugate either.
See also \cite[fig.~4.6]{firo3d-3}.


\section{Planar Sturm attractors}
\label{sec3}

As a prelude to the proof of theorem~\ref{thm:2.6} for 3-ball Sturm global attractors we recall the case of planar disks, in theorem~\ref{thm:3.1}.
See \cite[section~2]{firo3d-1} for details.
A central construction, in definition~\ref{def:2.4} above, assigns a ZS-Hamiltonian pair of paths $h_0, h_1$: $\lbrace 1, \ldots , N \rbrace \rightarrow \mathcal{E}$ through the vertices $v\in\mathcal{E}$ of the cells $c_v$ of a prescribed planar bipolar cell complex $\mathcal{C}$.
The construction of $h_0, h_1$ ensures that the permutation $\sigma$:= $h_0^{-1} \circ h_1 \in S_N$ is Sturm, $\sigma = \sigma_f$ and hence defines a Sturm meander $\mathcal{M}_f$.
Moreover, the associated Sturm global attractor $\mathcal{A}_f$ is planar with Thom-Smale cell complex $\mathcal{C}_f$ as prescribed by $\mathcal{C}$.
See theorem~\ref{thm:3.1}.
We then refine the analysis of the cell complex equality $\mathcal{C}_f = \mathcal{C}$, in the planar case.
In fact $\mathcal{C}_f = \mathcal{C}$ is understood in terms of a cell-to-cell homeomorphism $\Phi$: $\mathcal{C}\rightarrow \mathcal{C}_f$.
We refine this to a signed homeomorphism $\Phi^s$: $\mathcal{C}^s \rightarrow \mathcal{C}_f^s$ between signed cell complexes.
In other words, $\Phi^s(S_\delta^j(v)) = \Sigma_\delta^j(\Phi^s(v))$ maps corresponding hemispheres of $\mathcal{C}^s$ and $\mathcal{C}_f^s$ onto each other, for all equilibria $v$, signs $\delta = \pm$, and dimensions $0 \leq j<i(v)$; see \eqref{eq:3.2} and corollary~\ref{cor:3.2}.
In particular we show how the disk orientations of the planar embedding $\mathcal{C} \subseteq \mathbb{R}^2$, together with the bipolar orientation of the 1-skeleton $\mathcal{C}^1$, already fix a signed hemisphere structure of $\mathcal{C}^s$, and hence determine the boundary orders $h_\iota^f = h_\iota$ and the Sturm permutation $\sigma_f= h_0^{-1} h_1$ uniquely.
See \eqref{eq:3.8}--\eqref{eq:3.12}.
For a topological disk $\mathcal{C}$, we recall how the remaining freedom of sign choices when passing to $\mathcal{C}^s$ amounts to trivially equivalent global attractors $\mathcal{A}_f = \mathcal{C}_f$, under $x \mapsto 1-x$ and $u \mapsto-u$, once the target sink equilibria of the one-dimensional fast unstable manifolds $W^1(v)$ have been fixed, for all $i=2$ source equilibria $v$.

\begin{figure}[t!]
\centering \includegraphics[width=\textwidth]{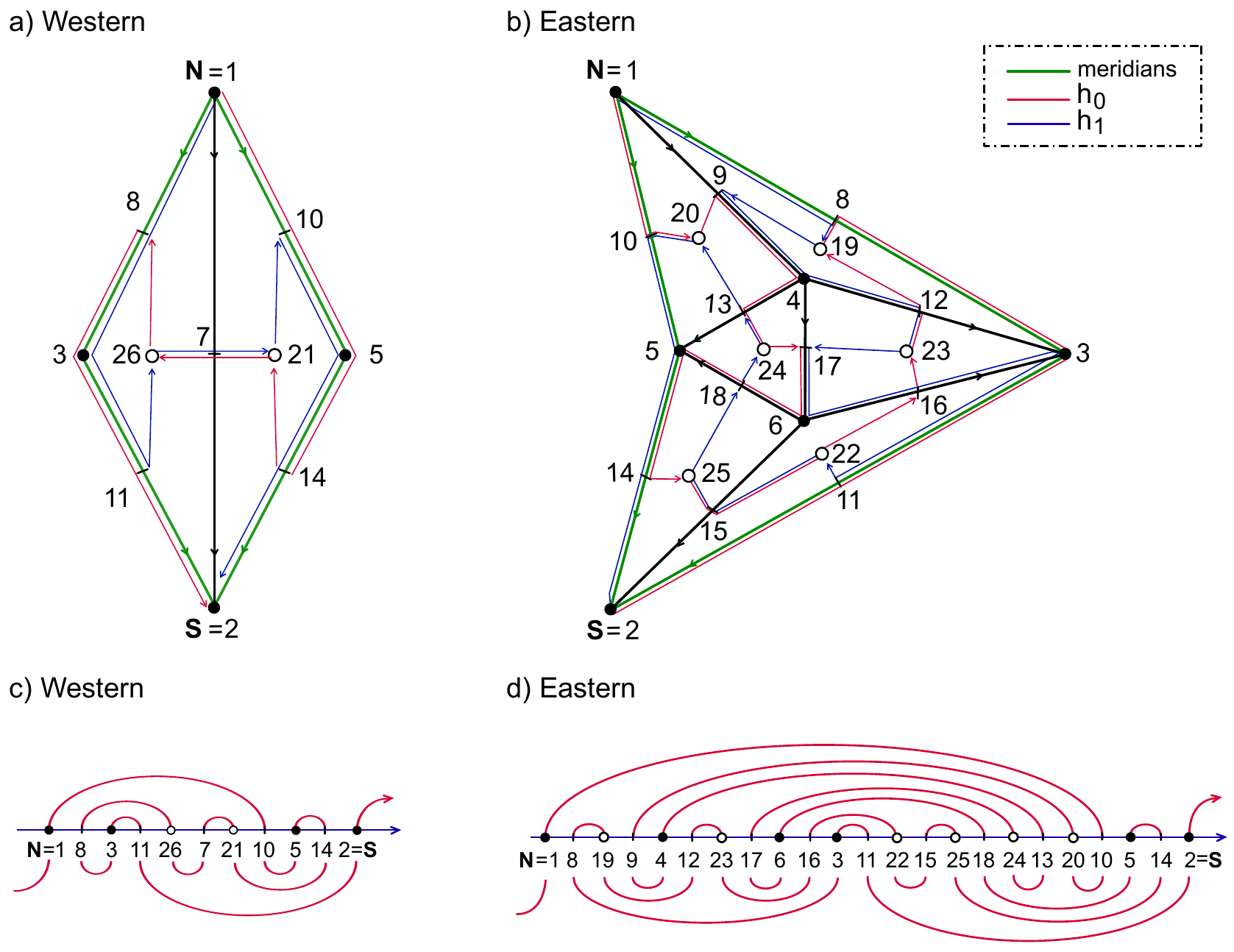}
\caption{\emph{
The closed hemisphere Thom-Smale complexes, alias topological disk attractors, of the 3-cell octahedron from fig.~\ref{fig:1.3}.
Western hemisphere (a), and Eastern hemisphere (b).
See (c), (d) for the associated meanders, respectively.
Vertex annotations, viz. equilibrium labels, correspond to fig.\ref{fig:1.3}.
}}
\label{fig:3.1}
\end{figure}

We first consider planar Sturm global attractors $\mathcal{A}_f$ and complexes $\mathcal{C}$ which are \emph{topological disks}.
By this we mean that $\mathcal{A}_f$, $\mathcal{C}$ are allowed to contain several sources of Morse index $i=2$, alias faces, but $\mathcal{A}_f$, $ \mathcal{C}$ are homeomorphic to the standard closed disk.
We recall definition~\ref{def:2.1} of the signed hemisphere template $\mathcal{E}_\pm^j(v)$ of $\mathcal{A}_f$, according to equilibria in the hemisphere decomposition $\Sigma_\pm^j(v)$ of $\partial W^u(v)$, for all equilibria $v \in \mathcal{E}_f$ and $0 \leq j < i(v)$.
In \cite[theorem~2.4]{firo3d-1} we proved the following theorem.

\begin{thm}\label{thm:3.1}
\begin{itemize}
\item [] \phantom{linebreak}
\item[(i)] Let $(h_0,h_1)$ be the ZS-pair of a given planar bipolar topological disk complex $\mathcal{C} \subseteq \mathbb{R}^2$ with poles $\mathbf{N}$, $\mathbf{S}$ on the circular boundary of $\mathcal{C}$.
Then the Sturm permutation $\sigma_f = \sigma$:= $h_0^{-1}h_1$ defines a topological disk Sturm global attractor $\mathcal{A}_f$ with dynamic complex $\mathcal{C}_f=\mathcal{C}$, and hence a unique signed hemisphere template $\mathcal{E}_\pm^j(v)$.
\item[(ii)] Conversely, let $\mathcal{E}_\pm^j(v)$ be the signed hemisphere template of a given planar Sturm global attractor $\mathcal{A}_f$.
Then $\mathcal{E}_\pm^j(v)$ defines a unique bipolar orientation of the planar Thom-Smale complex $\mathcal{C}_f$ of $\mathcal{A}_f$, and hence a unique ZS-pair $h_\iota$:= $h_\iota^f$, $\iota = 0,1$.
\end{itemize}
\end{thm}

See fig.~\ref{fig:3.1}(b) for an illustration of theorem~\ref{thm:3.1}, featuring the ZS-pair $(h_0,h_1)$ for the given orientation of the Eastern hemisphere part of the solid octahedron from fig.~\ref{fig:1.3}.
In fig.~\ref{fig:3.1}(a) the SZ-pair $(h_0,h_1)$ is illustrated for the Western hemisphere of the same example.

Since theorem~\ref{thm:3.1} will play a central role in our proof of theorems~\ref{thm:2.6} and \ref{thm:2.7}, let us comment on the precise interpretation of the equality $\mathcal{C}=\mathcal{C}_f$ here; see \cite{firo08} for further details.
As in the 3-ball case of theorem~\ref{thm:2.6}, equality is understood in the sense of a cell preserving homeomorphism
	\begin{equation}
	\Phi^s: \qquad \mathcal{C} = \bigcup\limits_{v\in \mathcal{E}} c_v
	\quad \longrightarrow \quad 
	\bigcup\limits_{v \in \mathcal{E}_f} W^u(v)
	= \mathcal{C}_f = \mathcal{A}_f
	\label{eq:3.1}
	\end{equation}
with $\Phi^s(c_v) = W^u(v)$, which also preserves the signed hemisphere structure
	\begin{equation}
	\Phi^s\left(S_\delta^j(v)\right) =
	\Sigma_\delta^j\left(\Phi^s(v)\right)\,,
	\label{eq:3.2}
	\end{equation}
for all $v \in \mathcal{E}$, $0\leq j<i_v$, and $\delta = \pm$.
First, this requires a bijective identification
	\begin{equation}
	\Phi^s: \qquad \mathcal{E} 
	\quad \longrightarrow \quad \mathcal{E}_f\,,
	\label{eq:3.3}
	\end{equation}
for the restriction of $\Phi^s$ to the barycenters $v \in \mathcal{E}$ of the cells $c_v \in \mathcal{C}$.
Recalling \cite[lemma~5.2]{firo08}, this identification is defined by the ZS-pair $(h_0,h_1)$ in $\mathcal{E}$ and the boundary orders $(h_0^f, h_1^f)$ in $\mathcal{E}_f$ as
	\begin{equation}
	\Phi^s:= h_\iota^f \circ h_\iota^{-1}\,.
	\label{eq:3.4a}
	\end{equation}
Since $h_0^{-1} \circ h_1 = \sigma = \sigma_f = (h_0^f)^{-1} \circ h_1^f$, the two choices $\iota = 0,1$ define the same bijection $\Phi^s$ in \eqref{eq:3.3}, \eqref{eq:3.4a}.
We therefore use the same symbol $v$ to denote $v \in \mathcal{E}$ and $\Phi^s(v) \in \mathcal{E}_f$.
With this convention we obtain
	\begin{equation}
	h_\iota = h_\iota^f
	\label{eq:3.4b}
	\end{equation}
for $\iota = 0,1$.

In \cite[lemma~5.3]{firo08} we have shown that the vertex identification \eqref{eq:3.3}, \eqref{eq:3.4a} between $\mathcal{C}$ and $\mathcal{C}_f$ already defines an isomorphism between the \emph{filled graph} $G_2$ of $\mathcal{C}$ and the (unsigned) connection graph $\mathcal{H}_f$ of $\mathcal{C}_f$.
Here the filled graph $G_2$ of $\mathcal{C}$ is augmented by the edges from any face center $v$, of 2-dimensional cells $c_v$ in $\mathcal{C}$, to all saddles $w$ of edges $c_w \subseteq \partial c_v$, in addition to the bipolar 1-skeleton $\mathcal{C}^1$.
Sometimes $G_2$ is called the \emph{quadrangulation} of $\mathcal{C}^1$ to emphasize the partitions of $c_v$ into quadrangles.
The graph isomorphism preserves orientation on $\mathcal{C}^1$.
By transitivity and cascading of heteroclinic connectivity in the Sturm attractor $\mathcal{A}_f =\mathcal{C}_f$ we also conclude
	\begin{equation}
	\dim c_v = \dim W^u(v) = i(v) = i_v\,,
	\label{eq:3.5}
	\end{equation}
i.e. the vertex identification \eqref{eq:3.3} preserves cell dimension.
More precisely, the graph isomorphism $\Phi^s$: $G_2 \rightarrow \mathcal{H}_f$ ensures the left equivalence in 
	\begin{equation}
	c_w \subseteq \partial c_v
	\quad \Longleftrightarrow \quad
	v \leadsto w
	\quad \Longleftrightarrow \quad
	W^u(w) \subseteq \partial W^u(v)\,.
	\label{eq:3.6}
	\end{equation}
The right equivalence follows from Morse-Smale transversality in the Thom-Smale complex, as we recall from the introduction.
This allows us to define the homeomorphism $\Phi^s$ by induction	over the cell dimensions $i(v)$ as follows.

The identification of $0$-cells $c_v$, alias sink vertices $i(v)=0$, takes care of the case $i(v)=0$.
Once the homeomorphism $\Phi^s$: $\mathcal{C}^{i-1} \rightarrow\mathcal{C}_f^{i-1}$ has been constructed for the $(i-1)$-skeleta, $i \geq 1$, we can define the extension
	\begin{equation}
	\Phi^s: \qquad \mathcal{C}^i 
	\quad \longrightarrow \quad \mathcal{C}_f^i\,,
	\label{eq:3.7}
	\end{equation}
separately on each closed cell $\overline{c}_v$ of dimension $i(v) = i$.
	Indeed, we may simply extend $\Phi^s$, already defined on the sphere boundary $S^{i-1}(v) = \partial c_v \subseteq \mathcal{C}^{i-1}$ of any regular cell $\overline{c}_v \subseteq \mathcal{C}^i$, radially inwards towards the cell center $v$.
	
The construction of a \emph{signed} homeomophism $\Phi^s$, however, requires a little extra care.
On 1-cell edges $c_w$ of the bipolar 1-skeleton $\mathcal{C}^1$ we observe how the graph isomorphism $\Phi^s$ in \eqref{eq:3.3}, \eqref{eq:3.4a} maps tails $S_-^0(w)$ and heads $S_+^0(w)$ of the bipolar orientation to the signed hemi``sphere'' boundaries  $\Sigma_-^0(w)$ and $\Sigma_+^0(w)$ of the edge $W^u(w)$, respectively.
See our definition of $S_\pm^j$ above \eqref{eq:2.23}.
Indeed, we first note that both $h_\iota^f$ traverse the sink equilibrium $\Sigma_-^0(w)$ before $\Sigma_+^0(w) > w > \Sigma_-^0(w)$, simply because both $h_\iota^f$ proceed according to the boundary order at $x=\iota =0,1$.

We show next how each of the paths $h_\iota$ in $\mathcal{C}$, likewise, traverses the tail vertex $S_-^0(w)$ before the head vertex $S_+^0(w)$.
Indeed the ZS-rules of definition~\ref{def:2.4} for the ZS-pair $(h_0,h_1)$ of Hamiltonian paths in $G_2$ ensure that both $h_\iota$ traverse the vertex $S_-^0(w)$ before $S_+^0(w)$.
In fact, each $h_\iota$ defines an extension of the partial bipolar order on $\mathcal{C}^1 \cap \mathcal{E}$ to a total order of all vertices $\mathcal{E}$ of $G_2$.
To see this we just observe that each $h_\iota$ defines a polar Jordan curve from $\mathbf{N}$ to $\mathbf{S}$ in the planar complex $\mathcal{C}$; see also fig.~\ref{fig:2.1}.
Therefore $\Phi^s$ is signed, automatically, on the 1-skeleta $\mathcal{C}^1$ and $\mathcal{C}_f^1$.

It remains to understand why $\Phi^s$ is also signed on each face closure $\overline{c}_v\rightarrow \text{clos } W^u(v)$, for $i=2$ sources $v$.
Again we note how both $h_\iota^f$ traverse $\Sigma_-^0(v)$ before, and $\Sigma_+^0(v)$ after, any other equilibrium in $\text{clos }W^u(v)$.
The same holds true for the paths $h_\iota$ in the closed 2-cell $\overline{c}_v$, with respect to the boundary minimum $S_-^0(v)$ and the boundary maximum $S_+^0(v)$, respectively, under the bipolar orientation of the 1-skeleton $\mathcal{C}^1$.
Therefore $\Phi^s$ maps the vertices $S_\delta^0(v)$ to the equilibria $\Sigma_\delta^0(v)$, respectively, for $\delta = \pm$.

Our choice of the ZS-pair $(h_0,h_1)$ for the labeling maps $h_\iota$, in the cell complex $\mathcal{C}$, and the identification $h_\iota = h_\iota^f$ in \eqref{eq:3.4a}, \eqref{eq:3.4b} further imply
	\begin{equation}
	\Phi^s\left(S_\delta^1(v)\right) =
	\Sigma_\delta^1(v)\,,
	\label{eq:3.8}
	\end{equation}
for any $i=2$ source $v$ and any $\delta = \pm$.
Indeed, the boundary order $h_0^f$ traverses $w(+-)$, and hence all equilibria in $\Sigma_-^1(v)$, before the face center $v$; see proposition~\ref{prop:2.3}. Similarly, $h_0^f$ traverses $w(-+)$, as well as all other equilibria in $\Sigma_+^1(v)$, after $v$. In the exact same way, the abstract path $h_0$ traverses all vertices in the right boundary $S_-^1(v)$ of the face $c_v$ before, and the left boundary $S_+^1(v)$ after, the face center $v$ itself.
This proves \eqref{eq:3.8}.
It also shows that the homeomorphism $\Phi^s$, defined by radial extension above, is already a signed homeomorphism, i.e. $\Phi^s$: $\mathcal{C} \rightarrow\mathcal{C}_f$ preserves the signed hemisphere decompositions of $\mathcal{C}$ and $\mathcal{C}_f$.
It is useful to rethink the above observations, based on $h_1 = h_1^f$ instead $h_0= h_0^f$ -- with identical results.

Next consider two planar Sturm attractors $\mathcal{A}_f$ and $\mathcal{A}_g$ which are topological disks.
Suppose $\mathcal{A}_f$ and $\mathcal{A}_g$ possess the same signed hemisphere decompositions $\Sigma_{f, \pm}^j(v)$, $\Sigma_{g,\pm}^j(v)$ of their Sturm complexes $\mathcal{C}_f$ and $\mathcal{C}_g$.
By this we mean a bijection $\mathcal{E}_f \rightarrow \mathcal{E}_g$ of equilibria $v_f \mapsto v_g$, with isomorpic connection graphs $\mathcal{H}_f \cong \mathcal{H}_g$, such that the \emph{signed} zero numbers coincide,
	\begin{equation}
	z(w_f-v_f) = z(w_g-v_g)\,,
	\label{eq:3.9}
	\end{equation}
whenever $v_f \leadsto w_f$, alias $v_g \leadsto w_g$.
By the above arguments, we then have a signed homeomorphism
	\begin{equation}
	\begin{aligned}
	\Phi^s :\quad &\mathcal{C}_f^s
	\longrightarrow\mathcal{C}_g^s\\
	\Phi^s :=\quad &\Phi_g^s \circ \left(\Phi_f^s\right)^{-1}
	\end{aligned}
	\label{eq:3.10}
	\end{equation}
of their signed Sturm complexes, which preserves the respective signed hemisphere decompositions:
	\begin{equation}
	\Phi^s\left(\Sigma_{f, \pm}^j(v_f)\right)=
	\Sigma_{g, \pm}^j(v_g)\,.
	\label{eq:3.11}
	\end{equation}
Moreover, \eqref{eq:3.4b} implies $h_\iota^f = h_\iota^g$, for $\iota =0,1$, and hence the Sturm permutations $\sigma= h_0^{-1}\circ h_1$ coincide,
\begin{equation}
	\sigma_f = \sigma_g\,.
	\label{eq:3.12}
	\end{equation}
In this sense, theorems~\ref{thm:2.6} and \ref{thm:2.7} hold true for planar Sturm attractors which are topological disks.

Let us add a word about orientation.
Suppose we had chosen an SZ-pair $(h_0,h_1)$ in the planar topological disk $\mathcal{C}$, instead of a ZS-pair.
Then we should define the left, rather than the right, boundary of all faces $c_v$ to be $S_-^1(v)$.
The right boundaries would then become $S_+^1(v)$, instead.
By the above arguments, the homeomorphism $\Phi^s$ would then remain signed.
Effectively this amounts to a homeomorphic description of the Sturm complex $\mathcal{C}_f$ by a planar complex of the opposite orientation.
Comparing the separated Western and Eastern hemispheres of the solid octahedron in fig.~\ref{fig:1.3}, as depicted in fig.~\ref{fig:3.1}(a), (b), the hemisphere descriptions differ by precisely this orientation reversal.
This is due to the fact that we present both hemispheres $\Sigma_\pm^2(\mathcal{O})$ of $\Sigma^2(\mathcal{O})$ in the same coordinate frame.
Note however, how the identified meridians $\mathbf{WE} = \Sigma_+^1(\mathcal{O})$ and $\mathbf{EW} = \Sigma_-^1(\mathcal{O})$ of fig.~\ref{fig:1.3} and table \eqref{eq:1.15} entirely consist of edges in $\Sigma_+^1(v)$ and $\Sigma_-^1(v)$, respectively, in either planar orientation.
In \cite[fig.~5.2]{firo3d-1}, we have presented an example of orientation reversal in a Sturm 3-ball.

We can now extend theorem~\ref{thm:3.1} to general planar Sturm attractors $\mathcal{A}_f$ which are not topological disks. Such attractors consists of a linear chain of a number $d \geq 0$ of topological disks with intermediate one-dimensional chains, glued on.
This possibly includes a prepended and/or appended one-dimensional spike.
The chains consist of alternating sinks and saddles, each chain with a first and, possibly identical, last sink.

For a single topological disk, the orientation reversal of the planar embedding of a single 2-cell face reverses the orientation of all other cells.
For several disk components, $d >1$, we may choose the orientation of each cell, individually.
In general, this will lead to cell-homeomorphic planar Sturm attractors $\mathcal{A}_f$ with different Sturm permutations $\sigma_f$.
Fixing the signed hemisphere decomposition, alias the ZS-rule for face traversing pairs $(h_0,h_1)$, alias the right/left rule for $S_\pm^1(v)$ in cell faces, will still determine the signed Sturm complex $\mathcal{C}_f^1(v)$ and the Sturm permutation $\sigma_f$ uniquely.

These remarks prove the following variant of theorem~\ref{thm:3.1}, in terms of signed planar cell complexes.

\begin{cor}\label{cor:3.2}
\begin{itemize}
\item [] \phantom{linebreak}
\item[(i)] Let $(h_0,h_1)$ be the ZS-pair of any given planar bipolar complex $\mathcal{C}^s \subseteq \mathbb{R}^2$ with poles $\mathbf{N}$, $\mathbf{S}$ on the boundary of $\mathcal{C}_f$.
This identifies $\mathcal{C}$ as a signed complex $\mathcal{C}^s$; see \eqref{eq:2.23}.
Then the Sturm permutation $\sigma_f = \sigma$:= $h_0^{-1}h_1$ defines a unique signed Sturm complex
	\begin{equation}
	\mathcal{C}_f^s=\mathcal{C}^s\,,
	\label{eq:3.13}
	\end{equation}
in the sense of \eqref{eq:2.3}, \eqref{eq:2.10a}.
Equality in \eqref{eq:3.13} is understood by a signed homeomorphism $\Phi^s$ as in \eqref{eq:3.1}, \eqref{eq:3.2} above.
\item[(ii)] Conversely, let $\mathcal{C}_f^s$ be the signed Sturm complex of a given planar Sturm attractor $\mathcal{A}_f$.
Then the signed hemisphere decomposition $\mathcal{C}_f^s$ defines a planar embedding of $\mathcal{C}_f^s$, with unique orientation of each disk component of $\mathcal{C}_f^s$, such that the boundary orders $h_\iota$:= $h_\iota^f$, $\iota = 0,1$ are a ZS-pair.
\end{itemize}
\end{cor}

We conclude this section by recalling the role of the fast unstable manifolds $W^{uu}(v) = W^1(v)$ of $i=2$ sources $v$ in 2-cells $c_v = W^u(v)$.
Their role is usually ignored in the study of Thom-Smale dynamic complexes.
Our goal is to clarify the extent to which these fast unstable manifolds already determine the sign information in the signed Sturm complex $\mathcal{C}_f^s$, given just the Sturm complex $\mathcal{C}_f$ itself.
Since $z(\cdot -v) = 0_\pm$ on $W^1(v) \smallsetminus \lbrace v\rbrace$, these manifolds are heteroclinic orbits
	\begin{equation}
	v \leadsto \Sigma_\pm^0(v)\,;
	\label{eq:3.14}
	\end{equation}
see proposition~\ref{prop:2.2}(iii).
In particular their targets identify the bipolar extrema $\Sigma_\pm^0(v)$ in the circular cell boundary $\partial W^u(v) = \partial c_v$, up to sign.
Flipping this sign in one single 2-cell flips all signs, in unison.
This defines the bipolar orientation on the 1-skeleton $\mathcal{C}_f^1$, up to global sign reversal.

In a planar Sturm attractor it only remains to determine the 1-hemispheres $\Sigma_\pm^1(v)$, for a complete specification of the signed Sturm complex $\mathcal{C}_f^s$.
For the case $d=1$ of a single topological disk, only, this follows globally from the bipolar orientation, up to a global simultaneous swap of all $\Sigma_+^1(v)$ with their respective counterparts $\Sigma_-^1(v)$.

Both the global reversal of the bipolar orientation and the global orientation flip of the planar embedding can be achieved by the trivial equivalences $x \mapsto 1-x$, $u \mapsto -u$; see the introduction and \cite[corollary~2.5]{firo3d-1}.
In conclusion, the Sturm complex $\mathcal{C}_f$ determines its signed version $\mathcal{C}_f^s$ uniquely, up to trivial equivalences, for the case $d=1$ of a single topological disk.
By corollary~\ref{cor:3.2}, this determines the realizing Sturm permutation $\sigma = \sigma_f$ of the prescribed (unsigned) Sturm complex $\mathcal{C} = \mathcal{C}_f$ uniquely, up to a flip conjugation $k \sigma k$ and taking inverses, once the target equilibria of the fast unstable manifolds $W^1(v)$ are specified.

This planar result neither extends to the case $d \geq 2$ of planar Sturm attractors $\mathcal{A}_f$ with multiple topological disk components, nor to 3-ball Sturm attractors.


\section{Noses and scoops}
\label{sec4}

In this section we study noses $\lbrace v_1, v_2\rbrace \in \mathcal{E}$ of concrete and abstract Sturm permutations $\sigma_f = \sigma$.
Abstractly, let $h_\iota$: $\lbrace 1, \ldots , N \rbrace \rightarrow \mathcal{E}$ be labeling maps such that $\sigma$:= $h_0^{-1} h_1$ is Sturm.
Then we call the pair $\lbrace v_1, v_2\rbrace$ a \emph{nose} if the elements $v_j$ are adjacently labeled by both maps $h_\iota$, i.e.
	\begin{equation}
	| h_\iota^{-1} (v_1) - h_\iota^{-1}(v_2) | =1\,,
	\label{eq:4.1}
	\end{equation}
for $\iota = 0,1$.
\begin{figure}[t!]
\centering \includegraphics[width=\textwidth]{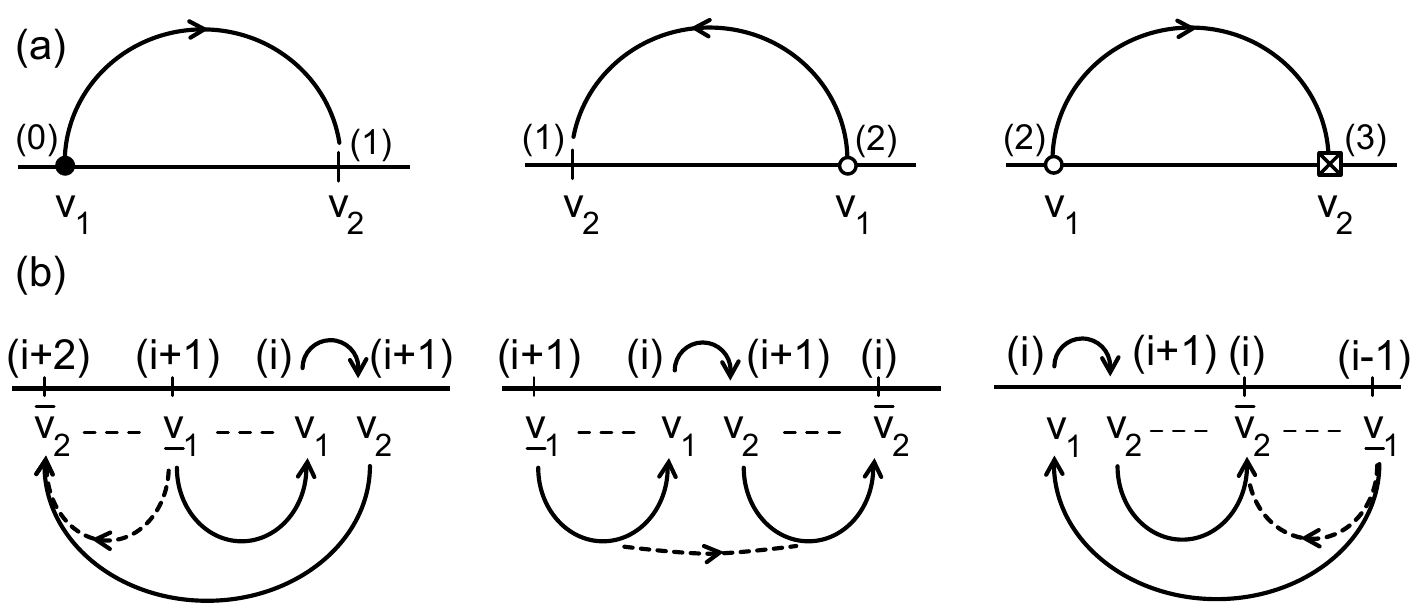}
\caption{\emph{
Upper arc noses $\{ v_1, v_2\}$, such that $h_0^{-1}(v_2) = h_0^{-1}(v_1)+1$.
(a) All cases with Morse numbers $i_{v_j} \leq 3$ (in parentheses).
Note how the meander $\mathcal{M}$ crosses the horizontal $h_1$-axis upwards, at even Morse numbers $(i_{v_1})$, and downwards at odd $(i_{v_2})$.
(b) Configurations with $i(v_2) =i(v_1)+1$, including $h_0$-predecessors $\underline{v}_1$ of $v_1$ and successors $\overline{v}_2$ of $v_2$.
Lower shortcut arcs in $\check{\mathcal{M}}$, from $\underline{v}_1$ to $\overline{v}_2$ after retraction of the nose $\{ v_1, v_2\}$ are dashed.
}}
\label{fig:4.1}
\end{figure}
We exclude the polar cases of $h_\iota^{-1}(v_1)$ or $h_\iota^{-1}(v_2)$ in $\lbrace 1, N \rbrace$ just for simplicity of notation in the nose retractions below.
The naming comes from the resulting arc configuration in the meander $\mathcal{M}$ associated to $\sigma$.
See also \cite{firo99}.
See fig.~\ref{fig:4.1} for the list of upper nose examples, i.e. $\mathcal{M}$-arcs above the horizontal $h_1$-axis, with Morse numbers $i \leq 3$.
Without loss of generality we fix
	\begin{equation}
	h_0^{-1} (v_2) = h_0^{-1}(v_1) +1\,.
	\label{eq:4.2}
	\end{equation}
By \eqref{eq:1.17}, \eqref{eq:1.18} the More numbers $i_{v_j}$ are also adjacent,
	\begin{equation}
	| i(v_1) - i(v_2) | =1\,,
	\label{eq:4.3}
	\end{equation}
and of the opposite even/odd parity compared to either label $h_\iota^{-1}(v_j)$.
The meander itself crosses the horizontal $h_1$-axis upwards, at odd labels, and downwards, at even labels.

A \emph{nose retraction} passes from $h_\iota, \sigma$ to $\check{h}_\iota, \check{\sigma}$, simply skipping a nose $\lbrace v_1, v_2\rbrace$ and its labels.
Thus $\check{h}_\iota$: $\lbrace 1, \ldots , N-2\rbrace \rightarrow \check{\mathcal{E}}$:= $\mathcal{E} \smallsetminus \lbrace v_1, v_2\rbrace$ and
	\begin{equation}
	\check{\sigma} := \check{h}_0^{-1} \circ \check{h}_1\,.
	\label{eq:4.4}
	\end{equation}
The associated meander $\check{\mathcal{M}}$ of $\check{\sigma}$ connects the intersection $\underline{v}_1$:= $h_0(h_0^{-1}(v_1)-1)$ of the $v_1$-predecessor to the $v_2$-successor $\overline{v}_2$:= $h_0(h_0^{-1}(v_2)+1)$ by a direct arc of $\check{\mathcal{M}}$, in the half plane opposite to the arc $v_1v_2$.
The shortcut $\check{h}_0$: $ \ldots \underline{v}_1 \overline{v}_2 \ldots$ is dashed in fig.~\ref{fig:4.1}(b).

In proposition~\ref{prop:4.1} below we show that nose retractions do not affect the Sturm property, Morse indices, or signed zero numbers of the remaining elements.
We caution the reader, however, that the remaining heteroclinic orbits of the connection graph $\mathcal{H}$ may well be affected.
In definition~\ref{def:4.2} we introduce certain sequences of successive nose retractions, called scoops.
In proposition~\ref{prop:4.3}, these scoops reduce permutations $\sigma$ of 3-meander templates to Sturm permutations $\sigma_\pm$ of planar Sturm attractors $\mathcal{A}_\pm$.
In section~\ref{sec5}, we will identify $\mathcal{A}_\pm$ as the closed hemispheres of the Sturm 3-ball attractor of $\sigma$ itself.

\begin{prop}\label{prop:4.1}
Let $\sigma \in S_N$ be any Sturm permutation, and let $\check{\sigma} \in S_{N-2}$ arise by nose retraction of $\lbrace v_1,v_2\rbrace$ from $\sigma$; see \eqref{eq:4.4}.

Then $\check{\sigma}$ is again a Sturm permutation.
The Morse indices $i_v$ and the signed zero numbers $z_{wv}$ of $v \neq w$ are all inherited from $\sigma$, without any change, for $v, w \in \check{\mathcal{E}}$:= $\mathcal{E} \smallsetminus \lbrace v_1,v_2\rbrace$.
\end{prop}

\begin{proof}[\textbf{Proof.}]
Without loss of generality, and to simplify language, suppose $\lbrace v_1, v_2\rbrace$ is an upper arc nose.
Else apply the trivial equivalence $u \mapsto -u$, which rotates all Sturm meanders by $180^\circ$.
By the labeling \eqref{eq:4.2} this implies $i$:= $i_{v_1}$ is even and $i_{v_2} = i_{v_1} \pm 1$ is odd.

We first show how $\check{\sigma}, \check{\mathcal{M}}$ define a meander.
In the meander $\mathcal{M}$ associated to $\sigma$ we only consider the case of a right oriented, and hence right turning, upper nose arc from $v_1$ to $v_2$.
Then $i_{v_2}= i_{v_1}+1$.
The other case, $i_{v_2} = i_{v_1}-1$, is analogous and will be omitted.
See fig.~\ref{fig:4.1}(b) for the resulting arc configurations and the Morse numbers of $h_0$: $\ldots \underline{v}_1 v_1 v_2 \overline{v}_2 \ldots $.
The nose vertices $v_1$ and $v_2 = h_1(h_1^{-1}(v_1)+1)$ are also $h_1$-adjacent, by definition \eqref{eq:4.1}.
The dashed lower arc shortcuts $\check{h}_0$: $\ldots \underline{v}_1 \overline{v}_2 \ldots$ which skip the retracted nose $\lbrace v_1, v_2 \rbrace$, therefore define a meander $\check{M}$.
In particular the permutation $\check{\sigma}$ defined by $\check{\mathcal{M}}$ is a meander.
Moreover $\check{\sigma}$ is dissipative by our exclusion of polar noses $\lbrace v_1, v_2\rbrace$.

To show preservation of Morse numbers under nose retraction we again consult the three cases of fig.~\ref{fig:4.1}(b), only. We compare the recursion \eqref{eq:1.17} for the passage from $i_{\overline{v}_1}$ to $i_{\overline{v}_2}$ before and after nose retraction of $\lbrace v_1 , v_2\rbrace$.
By induction from $j=0$ to $j_1 = h_0^{-1} (\overline{v}_1)$, the Morse numbers $i_{\overline{v}_1}$ coincide.
By inspection of fig.~\ref{fig:4.1}(b), the resulting Morse numbers $i_{\overline{v}_2} = i_{h_0(j_1+3)} = i_{\check{h}_0(j_1+1)}$ coincide in all cases.
This proves preservation of Morse numbers.
In particular $\check{\sigma}$ is Morse, as $\sigma$ is, which proves $\check{\sigma}$ is Sturm.

We prove preservation of the signed zero numbers $\check{z}_{vw} = z_{vw}$ under nose retraction of $\lbrace v_1, v_2\rbrace$, next.
Since nose retraction does not alter the $h_0$-order $<_0$ of the remaining vertices in $\check{\mathcal{E}}$, it is sufficient to prove preservation of the unsigned zero numbers.
In view of the explicit recursions \eqref{eq:2.3c} and preservation of Morse numbers, it is sufficient to prove
	\begin{equation}
	\check{z}_{\overline{v}_2 w} - \check{z}_{\underline{v}_1 w} =
	z_{\overline{v}_2 w} - z_{\underline{v}_1 w}
	\label{eq:4.5}
	\end{equation}
for $h_0$: $\ldots \underline{v}_1 v_1 v_2 \overline{v}_2 \ldots$ and any $w <_0 v_1$.
Here $\check{z}$ refers to $\check{h}_\iota$, $\check{\sigma}$ after nose retraction of $\lbrace v_1, v_2\rbrace$.
With the notation $k,j_1$ for $h_0^{-1}(w)$, $h_0^{-1}(v_1)$, and the abbreviations $\zeta_j, \check{\zeta}_j$ and $s_j$ for unsigned $z_{h_0(j)h_0(k)}$, $\check{z}_{h_0(j)h_0(k)}$, and $\tfrac{1}{2} \text{sign}(\sigma^{-1} (j) - \sigma^{-1}(k))$, respectively, claim \eqref{eq:4.5} reads
	\begin{equation}
	\check{\zeta}_{j_1} - \check{\zeta}_{j_1-1} =
	\zeta_{j_1+2} - \zeta_{j_1-1}\,.
	\label{eq:4.6}
	\end{equation}
To prove claim \eqref{eq:4.6}, we first note that recursion \eqref{eq:2.3c} asserts
	\begin{equation}
	\zeta_{j+1}-\zeta_j = (-1)^{j+1} \left( s_{j+1} - s_j\right)\,.
	\label{eq:4.7}
	\end{equation}
Note $s_{j_1+1} = s_{j_1}$ for the adjacent nose equilibria $v_2 = h_0(j_1+1)$ and $v_1= h_0(j_1)$.
Summing \eqref{eq:4.7} from $j=j_1-1$ to $j=j_1+1$ therefore implies
	\begin{equation} 
	\begin{aligned}
	&\zeta_{j_1+2} - \zeta_{j_1-1} =\\
	&\quad= (-1)^{j_1} \left( s_{j_1}- s_{j_1-1}\right) + (-1) ^{j_1+1}
		\left( s_{j_1+1}- s_{j_1}\right) + (-1)^{j_1+2}
		\left( s_{j_1+2}- s_{j_1+1}\right)=\\
	&\quad= (-1)^{j_1} 
		\left( -s_{j_1-1}+2s_{j_1}-2s_{j_1+1}+s_{j_1+2}\right)= \\
	&\quad= (-1)^{j_1}
		\left( s_{j_1+2}-s_{j_1-1}\right)=\\
	 &\quad = \check{\zeta}_{j_1} - \check{\zeta}_{j_1-1}\,.
	\end{aligned}
	\label{eq:4.8}
	\end{equation}
Here we have used $\check{s}_{j_1-1} = s_{j_1-1}$ and $\check{s}_{j_1}=s_{j_1+2}$ in the last equality.
This proves signed invariance of signed zero numbers under nose retraction, and also proves the proposition.
\end{proof}

From now on, and for the remaining paper, we return to a 3-meander template $\mathcal{M}$ with associated Sturm permutation $\sigma_f= \sigma$, Sturm attractor $\mathcal{A}_f$, Sturm complex $\mathcal{C}_f$, and boundary orders $h_\iota^f$ at $x=\iota = 0,1$.
See definition~\ref{def:1.3}.
Our first task is to work towards identifying $\mathcal{A}_f$ as a Sturm 3-ball, in theorem~\ref{thm:5.1} below.
As candidates ${\mathcal{E}'}_\pm^j$ for the equilibrium sets $\mathcal{E}_\pm^j$ in the signed hemisphere decomposition $\Sigma_\pm^j$ of the 2-sphere $\partial W^u(\mathcal{O})$, we define the following sets of vertices $\mathcal{E}$, alias equilibria $\mathcal{E}_f$:
	\begin{equation}
	\begin{aligned}
	{\mathcal{E}'_-}^0 &:=
		\lbrace h_\iota^f(1)\rbrace\,, \quad
		{\mathcal{E}'_+}^0 :=
		\lbrace h_\iota^f(N)\rbrace\,;\\
	{\mathcal{E}'_-}^1 &:=
		\lbrace v_-^1, v_-^2, \ldots , v_-^{2m-1}\rbrace\,, \quad
		{\mathcal{E}'_+}^1 :=
		\lbrace v_+^1, v_+^2,\ldots , v_+^{2n-1}\rbrace\,;\\
	{\mathcal{E}'}^j &:=
		\bigcup\limits_{k\leq j,\, \delta = \pm} {\mathcal{E}'_\delta}^k\,,
		\quad \text{clos } {\mathcal{E}'_\delta}^j :=
		{\mathcal{E}'_\delta}^j \cup {\mathcal{E}'}^{j-1}\,;\\
	{\mathcal{E}'_\delta}^2 &:=
		\lbrace v \in \mathcal{E} \smallsetminus {\mathcal{E}'}^1 \,|\,
		v \neq \mathcal{O}\,,\,\, \delta \mathcal{O} < \delta v \
		\text{ at }\ x=1\rbrace\,.
	\end{aligned}
	\label{eq:4.9}
	\end{equation}
Here $j=0,1,2$ and $\delta = \pm$.

\begin{defi}\label{def:4.2}
We define the \emph{East scoop} $\mathcal{M}_-$ with scooped Sturm permutation $\sigma_-$:= $(\check{h}_0)^{-1} \circ \check{h}_1$ as the result of the removal of $\mathcal{O} \cup {\mathcal{E}'_+}^2$, by successive nose retraction.
This leads to the replacement of the meander part
	\begin{equation}
	\begin{aligned}
	h_0&:\quad
	v_-^0 \ldots v_-^\mu \ldots v_-^{\mu'}w_-^1 \ldots v_-^{2m-1} \ldots 
	w_-^0\mathcal{O}w_+^0\ldots v_+^{2n-1} \ldots
	w_+^1 v_+^{\nu'}\ldots v_+^\nu \ldots v_+^0 \\ 
	&\text{ by }\\
	\check{h}_0&:\quad
	v_-^0 \ldots v_-^\mu \ldots v_-^{\mu'}w_-^1\ldots
	v_-^{2m-1} \ldots w_-^0 v_+^{2n-1} 
	\ldots v_+^{\nu'} \ldots v_+^\nu \ldots v_+^0\,.
	\end{aligned}
	\label{eq:4.10}
	\end{equation}
Similarly $\check{h}_1$ just skips the vertices $\mathcal{O} \cup {\mathcal{E}'_+}^2$.
Here $\check{h}_0$ and $\check{h}_1$ terminate along a full $\mathbf{S}$-polar $h_0$-serpent.

Analogously, the \emph{West scoop} $\mathcal{M}_+$, $\sigma_+$:= $(\check{h}_0)^{-1} \circ \check{h}_1$ removes $\mathcal{O} \cup {\mathcal{E}'_-}^2$ by successive nose retraction.
This replaces
	\begin{equation}
	\begin{aligned}
	h_0&:\quad
	v_-^0 \ldots v_-^\mu \ldots v_-^{\mu'}w_-^1 \ldots v_-^{2m-1} \ldots 
	w_-^0\mathcal{O}w_+^0\ldots v_+^{2n-1} \ldots
	w_+^1 v_+^{\nu'}\ldots v_+^\nu \ldots v_+^0 \\
	&\text{ by }\\
	\check{h}_0&:\quad
	v_-^0 \ldots v_-^\mu \ldots v_-^{\mu'} \ldots v_-^{2m-1}w_+^0 \ldots
	v_+^{2n-1} \ldots w_+^1v_+^{\nu'} \ldots v_+^\nu \ldots v_+^0\,.
	\end{aligned}
	\label{eq:4.11}
	\end{equation}
Similarly $\check{h}_1$ just skips the vertices $\mathcal{O} \cup {\mathcal{E}'_-}^2$.
Here $\check{h}_0$ and $\check{h}_1$ start along a full $\mathbf{N}$-polar $h_1$-serpent.
\end{defi}

To see how, say, the East scoop is actually feasible by successive nose retraction, let us consider fig.~\ref{fig:1.5} of a 3-meander template again.
We first note that all vertices of ${\mathcal{E}'_+}^2$ are located (nonstrictly) between $w_+^1$ and $w_+^0$ along the $h_1$-axis, excepting the vertices of type $v_-^j$:
	\begin{equation}
	{\mathcal{E}'_+}^2 = \lbrace w \in \mathcal{E}\, | \, w_+^1 \leq_1
	w \leq_1 w_+^0 \rbrace \smallsetminus {\mathcal{E}'_-}^1\,.
	\label{eq:4.12}
	\end{equation}
Here and below $<_\iota$ denotes the ordering at $x=\iota$, or by $h_\iota$, for $\iota =0,1$.
The reason for \eqref{eq:4.12} is the overlap of the polar serpents, by definition~\ref{def:1.3}(ii), together with extremality of $w_\pm^\iota$.
By successive nose retractions under the upper arcs of $v_-^2v_-^3, \ldots , v_-^{\mu-1} v_-^\mu$ we can achieve $\mu = 1$.
In other words, $v_-^1$ is the immediate $h_1$-successor of $w_+^0$.
We can then eliminate all vertices from $\mathcal{O}$ to $w_+^0$ by lower nose retraction, to arrive at the situation of definition~\ref{def:4.2}.
Analogous arguments justify the West scoop of
	\begin{equation}
	{\mathcal{E}'_-}^2 = \lbrace w \in \mathcal{E}\, | \, w_-^0 \leq_1
	w \leq w_+^1 \rbrace \smallsetminus {\mathcal{E}'_+}^1\,.
	\label{eq:4.13}
	\end{equation}

\begin{prop}\label{prop:4.3}
The scoops $\sigma_{f_\pm}= \sigma_\pm$ of definition~\ref{def:4.2} are Sturm permutations of planar Sturm attractors $\mathcal{A}_\pm$:= $\mathcal{A}_{f_\pm}$.
In particular equilibria $v \neq w$ in $\mathcal{E}_f$ satisfy
	\begin{align}
	v,w \in \text{clos } {\mathcal{E}'_\delta}^1 \quad
	\Longrightarrow \quad i(v) \leq 1,\, z(v-w) = 0\label{eq:4.14}\\
	v,w \in \text{clos } {\mathcal{E}'_\delta}^2 \quad
	\Longrightarrow \quad i(v) \leq 2,\, z(v-w) \leq 1
	\label{eq:4.15}
	\end{align}
for $\delta = \pm$, in the notation of \eqref{eq:4.9}.
\end{prop}

\begin{proof}[\textbf{Proof.}]
By definition~\ref{def:4.2}, the permutations $\sigma_\pm$ arise via successive nose reduction.
By proposition~\ref{prop:4.1}, the permutations are therefore Sturm.
Let $\sigma_{f_\pm} = \sigma_\pm$ with associated Sturm attractors $\mathcal{A}_\pm$.
By proposition~\ref{prop:4.1} again, all Morse numbers and zero numbers of $\sigma$ are inherited by $\sigma_\pm$.
By definition~\ref{def:1.3}(i) and the scooping of $\mathcal{O}$, the resulting Morse numbers $i_v=i(v)$ cannot exceed 2.
Therefore $\mathcal{A}_\pm$ are planar Sturm attractors.
In particular \eqref{eq:4.15} holds on $\mathcal{A}_+$ and on $\mathcal{A}_-$, respectively; this observation goes back as far as \cite{br90}.
This proves claim \eqref{eq:4.15} on ${\mathcal{E}'_\delta}^2 \subseteq \mathcal{A}_g$.

To prove claim \eqref{eq:4.14} we note that ${\mathcal{E}'_-}^1$ is part of the full $\mathbf{S}$-polar serpents $\check{h}_\iota$ after the West scoop.
By definition, \eqref{eq:4.14} holds along polar serpents.
The argument for ${\mathcal{E}'_+}^1$ is analogous.
This proves the proposition.
\end{proof}


\section{Sturm 3-balls from 3-meander templates}
\label{sec5}

We continue our analysis of the global attractor $\mathcal{A}_f$ associated to the Sturm permutation $\sigma_f = \sigma$ of the general 3-meander template $\mathcal{M},\sigma$ from definition~\ref{def:1.3} and fig.~\ref{fig:1.5}.
In theorem~\ref{thm:5.1} we state that $\mathcal{A}_f$ is a Sturm 3-ball.
In other words,
	\begin{equation}
	\mathcal{A}_f = \text{clos } W^u(\mathcal{O})
	\label{eq:5.1}
	\end{equation}
is the closure of the unstable manifold of the single equilibrium $\mathcal{O}$, at which the meander $\mathcal{M}$ crosses the horizontal $h_1$-axis with maximal Morse number $i_{\mathcal{O}} =3$; see definition~\ref{def:1.3}(i).
Our proof only requires to show the existence of heteroclinic orbits
	\begin{equation}
	\mathcal{O} \leadsto v\,,
	\label{eq:5.2}
	\end{equation}
for all equilibria $v \in \mathcal{E}_f \smallsetminus \lbrace \mathcal{O}\rbrace$.
In lemma~\ref{lem:5.2} we therefore recall the Wolfrum version of heteroclinicity in Sturm attractors, based on zero number.
The required input is collected in proposition~\ref{prop:5.3}, so that we can conclude this section with the proof of theorem~\ref{thm:5.1}.

\begin{thm}\label{thm:5.1}
Any 3-meander template $\mathcal{M},\sigma$ defines a Sturm 3-ball attractor $\mathcal{A}_f$ with Sturm permutation $\sigma_f=\sigma$ and meander $\mathcal{M}$.
\end{thm}

The notion of $k$-\emph{adjacency} is central for Wolfrum's reformulation, in \cite{wo02}, of the heteroclinicity results in \cite{firo96, firo99}.
We say two distinct equilibria $v_1, v_2$ are $k$-\emph{adjacenct} if there does not exist a third equilibrium $w$ between them, say at $x=0$, such that the signed zero numbers 
	\begin{equation}
	z\left( w-v_1\right) = k_\pm = z\left( v_2-w \right)
	\label{eq:5.3}
	\end{equation}
coincide with either $k_+$ or $k_-$, depending on the sign in $\pm(v_2(0)-v_1(0))>0$.

\begin{lem}[\textbf{[Wo02]}]\label{lem:5.2}
Let $\mathcal{A}_f$ be a Sturm global attractor with distinct equilibria $v_1, v_2 \in \mathcal{E}_f$.
Then $v_1 \leadsto v_2$ if, and only if, $i(v_1) > i(v_2)$ and $v_1, v_2$ are $z(v_2 -v_1)$-adjacent.
\end{lem}

We comment on the proof of this lemma in the appendix.
Suffice it here to recall how violation of $k$-adjacency, i.e. the existence of an in-between equilibrium $w$ with \eqref{eq:5.3}, blocks the existence of a heteroclinic orbit $u(t, \cdot )$ between $v_1$ and $v_2$.
Indeed the zero number $z(u(t, \cdot )-w)$ would have to drop strictly, when the boundary values of $u(t, \cdot )$ and $w$ cross each other at $x=0$ or at $x=1$.
For $t \rightarrow \pm\infty$, on the other hand, that zero number has to coincide with $k$.
For $k=0$, we have already encountered such a \emph{blocking argument} in the proof of proposition~\ref{prop:2.3}.
See also \cite{brfi89}.

Based on the decomposition \eqref{eq:4.9} of the equilibrium set
	\begin{equation}
	\mathcal{E}_f \smallsetminus \mathcal{O} =
	\bigcup\limits_{j\leq 2,\, \delta = \pm} 
	{\mathcal{E}'_\delta}^j\,,
	\label{eq:5.4}
	\end{equation}
in the Sturm attractor $\mathcal{A}_f$ of the 3-meander template $\mathcal{M}$ with $\sigma = \sigma_f$, we now collect information on the zero numbers on these sets.
This information coincides, verbatim, with the corresponding statements of \cite[proposition~3.1]{firo3d-1} on the hemisphere decomposition
	\begin{equation}
	\mathcal{E}_f \smallsetminus \mathcal{O} =
	\bigcup\limits_{j\leq 2,\, \delta = \pm} 
	\mathcal{E}_\delta^j
	\label{eq:5.5}
	\end{equation}
by the equilibrium sets $\mathcal{E}_\delta^j = \mathcal{E}_f \cap \Sigma_\delta^j$.
A posteriori, i.e. after theorem~\ref{thm:5.1} is proved and $\mathcal{A}_f$ is identified as a Sturm 3-ball, indeed, we will have arrived at the identification
	\begin{equation}
	{\mathcal{E'}}_\delta^j = \mathcal{E}_\delta^j
	\label{eq:5.6}
	\end{equation}
for all $0 \leq j \leq 2$ and both signs $\delta = \pm$.
For the moment, however, \cite[proposition~3.1]{firo3d-1} cannot be invoked and we must prove the following version, independently.
See fig.~\ref{fig:5.1} for an illustration of this result, but not its proof.

\begin{figure}[t!]
\centering \includegraphics[width=\textwidth]{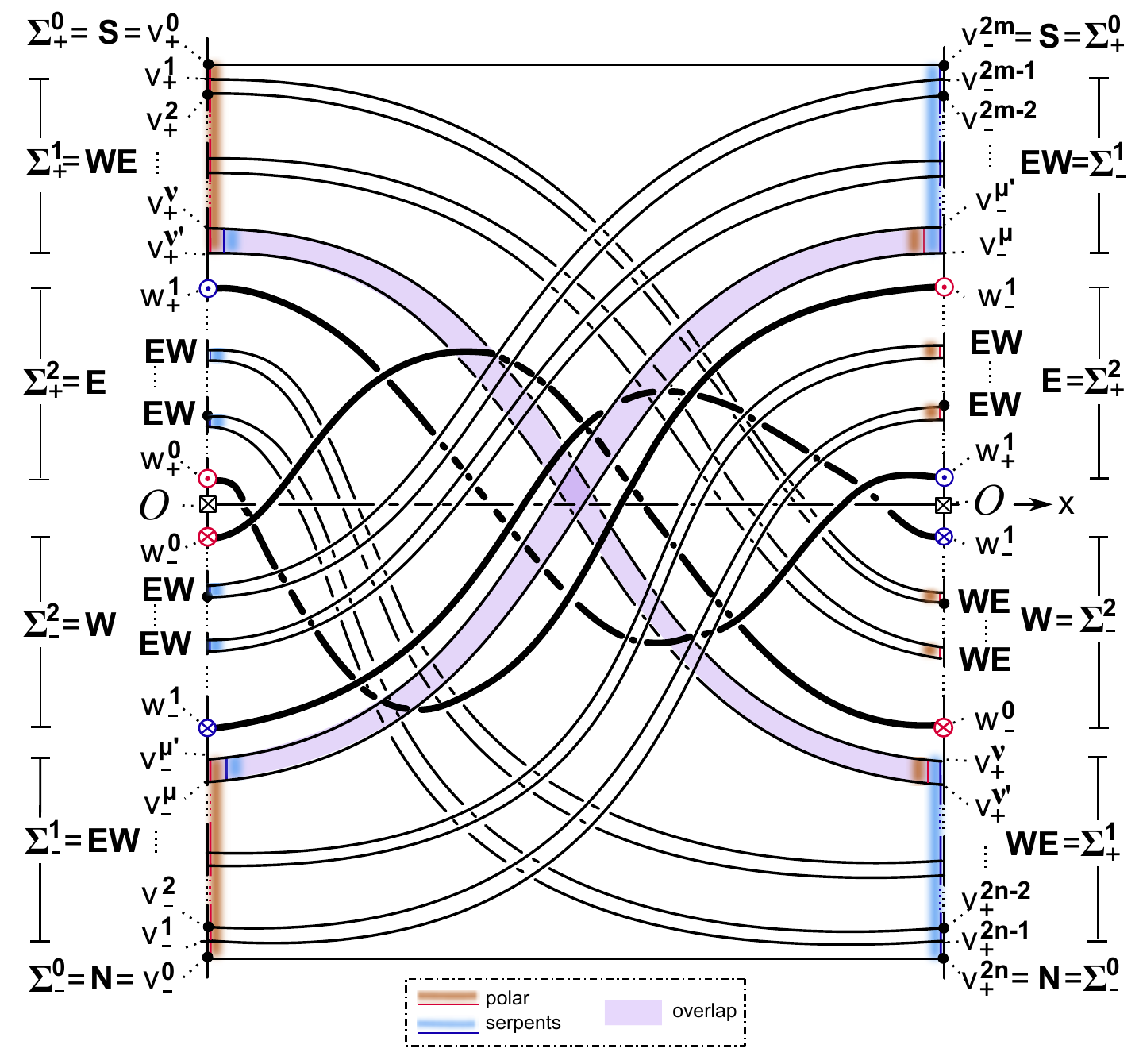}
\caption{\emph{
An impressionist sketch of the spatial profiles $v(x)$, for all equilibria $v\in \mathcal{E}_f$ of a general Sturm 3-meander template.
The drawing illustrates the results of proposition~\ref{prop:5.3}, as well as certain aspects of definition~\ref{def:1.3}.
For the specific case of a solid octahedron see also fig.~\ref{fig:1.1}.
}}
\label{fig:5.1}
\end{figure}

\begin{prop}\label{prop:5.3}
In the above setting and with the notation \eqref{eq:4.9} for the equilibrium sets ${\mathcal{E}'}_\pm^j$, the following statements hold true for all $0\leq j\leq 2$ and $\delta = \pm$.
\begin{align}
v \in {\mathcal{E}'}^j
\quad &\Longrightarrow \quad i(v) \leq j\tag{i}\\ 
v \in \text{clos } {\mathcal{E}'}^j
\quad &\Longrightarrow \quad z(v-\mathcal{O}) \leq j\tag{ii}\\
v \in {\mathcal{E}'}_\pm^j
\quad &\Longrightarrow \quad z(v-\mathcal{O})= j_\pm\tag{iii}\\
v_1, v_2 \in \text{clos } {\mathcal{E}'_\delta}^j
\quad &\Longrightarrow \quad z(v_1-v_2) < j\,.\tag{iv}
\end{align}
\end{prop}

\begin{proof}[\textbf{Proof.}]
Claim~(iv) is void for $j=0$.
For $j=1,2$, claims~(i),(iv) have already been proved in proposition~\ref{prop:4.3}.
Claims~(i),~(iii) for $j=0$ just reiterate $i_v=0$ for $v\in \lbrace h_\iota^f(1), \, h_\iota^f(N)\rbrace$, by dissipativeness; see \eqref{eq:1.17}, \eqref{eq:1.18}, \eqref{eq:2.3c}.
Claim~(ii) follows from claim~(iii), by definition \eqref{eq:4.9} of the sets $\text{clos }{\mathcal{E}'}^j$.

Therefore it only remains to prove claim~(iii).
Although it is possible to invoke scoops, except for the last nose retraction involving $\mathcal{O}$ itself, we proceed more directly this time.
With the abbreviations $\zeta_j$:= $z(h_0(j)-\mathcal{O})$, for the unsigned zero numbers, and with $s_j$:= $\tfrac{1}{2}\text{sign}( \sigma^{-1}(j)-h_1^{-1}(\mathcal{O}))$, the explicit recursion \eqref{eq:2.3c} reads
	\begin{equation}
	\zeta_{j+1} - \zeta_j = (-1)^{j+1} (s_{j+1} -s_j)\,;
	\label{eq:5.7}
	\end{equation}
see also \eqref{eq:4.7}. Here $1 \leq j <k$ and $k=h_0^{-1}(\mathcal{O})$ in \eqref{eq:2.3c}.
Note that $h_1^{-1} (\mathcal{O}) = h_1^{-1}h_0(k) = \sigma^{-1}(k)$.
We omit sub- and superscripts $f$ in this proof.
We only prove claim~(iii) for ${\mathcal{E}'_-}^j$; the cases of ${\mathcal{E}'_+}^j$ are analogous by the trivial equivalence $u \mapsto -u$.

The recursion \eqref{eq:5.7} is initialized with
	\begin{equation}
	\zeta_1 = 0\,,\quad s_1= -1/2\,,\quad h_0(1) <_0 \mathcal{O}\,,
	\label{eq:5.8}
	\end{equation}
by dissipativeness.
This proves claim~(iii) for the pole $\mathbf{N}= h_\iota(1)$ and settles $j=0$.

We follow the meander path of $h_0$ along the $\mathbf{N}$-polar $h_0$-serpent
	\begin{equation}
	h_0: \qquad 
	\mathbf{N}\ v_-^1v_-^2 \ldots v_-^\mu \ldots 
	v_-^{\mu'}\ldots
	\label{eq:5.9}
	\end{equation}
up to $v_-^{\mu'}$, next.
By definition~\ref{def:1.3}(iii), we have
	\begin{equation}
	\mathbf{N} = h_0(1) <_1 \mathcal{O} <_1 v_-^1 <_1 v_-^2 <_1
	\ldots <_1 v_-^\mu <_1 \ldots <_1 v_-^{\mu'}
	\label{eq:5.10}
	\end{equation}
along that serpent.
See also fig.~\ref{fig:1.5}.
Since $v_-^j= h_0(j+1)$, for $0 \leq j\leq \mu'$, this implies
	\begin{equation}
	s_2 = \ldots s_{\mu+1} = \ldots s_{\mu'+1} = 1/2\,.
	\label{eq:5.11}
	\end{equation}
With recursion \eqref{eq:5.7} and initialization \eqref{eq:5.8} this proves
	\begin{equation}
	\zeta_2= \ldots = \zeta_{\mu+1} = \ldots \zeta_{\mu'+1} = 1\,.
	\label{eq:5.12}
	\end{equation}

\begin{figure}[t!]
\centering \includegraphics[width=\textwidth]{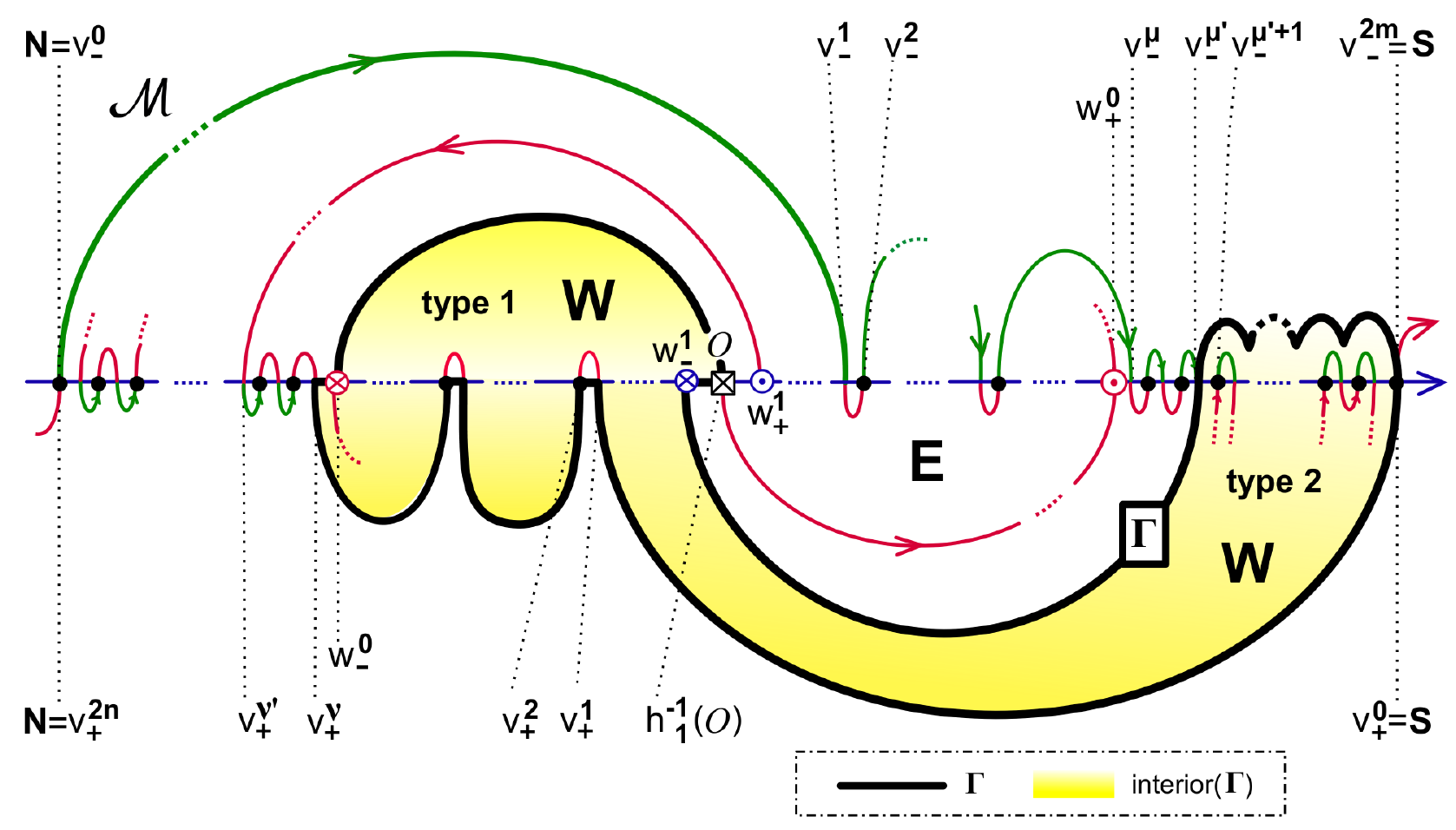}
\caption{\emph{
The trapping region interior$(\Gamma)$, in a West scoop, and the types of trapped equilibria of a 3-meander template; see the proof of proposition~\ref{prop:5.3}, \eqref{eq:5.14}--\eqref{eq:5.16}.
}}
\label{fig:5.2}
\end{figure}

By definition~\ref{def:1.3}(iv), the $\mathbf{N}$-polar $h_0$-serpent is terminated by $w_-^1$ or $w_+^1$.
In fact $\mathcal{O} \leadsto w_\pm^\iota$, because $h_\iota$-neighbors cannot be blocked.
Hence $i(\mathcal{O}) =3$ implies $i(w_\pm^\iota) = 2$ and, by \eqref{eq:2.3c},
	\begin{equation}
	z\left( w_\pm^\iota -\mathcal{O}\right) = 2_\pm\,.
	\label{eq:5.13}
	\end{equation}
This proves that the $h_0$-successor of the serpent termination $v_-^{\mu'}$ is $w_-^1$, rather than $w_+^1$; see fig.~\ref{fig:1.5} again.
Similarly, the $h_0$-predecessor $w_-^0$ of $\mathcal{O}$ terminates the $\mathbf{N}$-polar $h_1$-serpent $v_+^{2n} \ldots v_+^\nu$ along the $h_1$-axis.
By the Jordan curve property, this traps the meander segment $\mathcal{M}$, from the entry $w_-^1$ to the exit $w_-^0$, inside the trapping region defined by the Jordan curve
	\begin{equation}
	\Gamma: \qquad
	v_-^{\mu'} w_-^1 \ \mathcal{O}\ w_-^0 v_+^\nu \ldots
	v_+^2v_+^1\ \mathbf{S}\,.
	\label{eq:5.14}
	\end{equation}
See fig.~\ref{fig:5.2}.
Here $\Gamma$ consists of $h_0$- and $h_1$-arcs, alternatingly, and terminates with the part $v_+^\nu \ldots \mathbf{S}$ of the $\mathbf{S}$-polar $h_0$-serpent.
The Jordan curve $\Gamma$ is not closed.
We consider the remaining part
	\begin{equation}
	v_-^{\mu'+1} \ldots v_-^{2m-1}
	\label{eq:5.15}
	\end{equation}
of the $\mathbf{S}$-polar $h_1$-serpent to still be inside the trapping region $\Gamma$ of our meander $\mathcal{M}$ segment from $w_-^1$ to $w_-^0$.

Equilibrium vertices $v = h_0(j)$ inside $\Gamma$ consist of two types:
	\begin{equation}
	\begin{aligned}
	\text{type } 1&: \qquad
	w_-^0 \leq_1 v \leq_1 w_-^1 <_1 \mathcal{O}\,, \quad \text{and} \\
	\text{type } 2&: \qquad 
	v\in \lbrace v_-^{\mu'+1}\ , \ldots ,\, v_-^{2m-1}\rbrace
	>_1 \mathcal{O}\,.
	\end{aligned}
	\label{eq:5.16}
	\end{equation}
Suppose the meander path $h_0$ changes type along the $\mathcal{M}$-arc from $h_0(j)$ to $h_0(j+1)$.
We claim $j$ must be even.
Indeed, the trapping region $\Gamma$ ensures that a change of type can only occur via a lower $h_0$-arc of the meander $\mathcal{M}$.
Therefore the meander must cross the $h_1$-axis downward at $v_j$:= $h_0(j)$, upward at $v_{j+1}= h_0(j+1)$, and $j$ must be even.

The types distinguish the signs of $s_j$ to be
	\begin{equation}
	s_j=
	\left\lbrace 
	\begin{aligned}
	&-1/2\quad \text{for}\quad h_0(j) \text{ of type } 1\,;\\
	&+1/2\quad \text{for}\quad h_0(j) \text{ of type } 2\,.	
	\end{aligned}
	\right.
	\label{eq:5.17}
	\end{equation}
Indeed the relative ordering of $\sigma^{-1}(j) =h_1^{-1}(v_j)$ and $h_1 ^{-1}(\mathcal{O})$ distinguishes the type of $v_j= h_0(j)$.
In particular, the recursion \eqref{eq:5.7} determines the values $\zeta_j$ inside the trapping region, with the initialization $\zeta =2$ at $w_-^1$, to be
	\begin{equation}
	\zeta_j=
	\left\lbrace 
	\begin{aligned}
	&2\quad \text{for}\quad h_0(j) \text{ of type } 1\,;\\
	&1\quad \text{for}\quad h_0(j) \text{ of type } 2\,.	
	\end{aligned}
	\right.
	\label{eq:5.18}
	\end{equation}
Here we have used that $j$ is even at any type change from $v_j = h_0(j)$ to $v_{j+1}=h_0(j+1)$.
Hence \eqref{eq:5.17} implies a decrease of $\zeta_j$ by 1, upon passage from type~1 to type~2, and an increase by 1 upon return.
Without change of type, both $s_j$ and $\zeta_j$ remain unchanged.

By definition \eqref{eq:4.9} of ${\mathcal{E}'_\delta}^2$, we see how \eqref{eq:5.18} proves claim~(iii) for ${\mathcal{E}'_-}^2$.
Type~2, together with our previous observation \eqref{eq:5.12} proves claim~(iii) for ${\mathcal{E}'_-}^1$ and completes the proof of the proposition.
\end{proof}

\begin{proof}[\textbf{Proof of theorem 5.1.}]
It is sufficient to establish heteroclinic orbits $\mathcal{O} \leadsto v$ from the unique $i_{\mathcal{O}} = 3$ equilibrium to any other equilibrium $\mathcal{O} \neq v \in \mathcal{E}_f$.
By the Wolfrum lemma~\ref{lem:5.2} this is equivalent to showing that
	\begin{equation}
	\mathcal{O}, v \ \text{are} \
	z(v-\mathcal{O})\text{-adjacent}\,.
	\label{eq:5.19}
	\end{equation}
Note $i_v\leq 2$.
The relevant information on zero numbers $z$ is listed in proposition~\ref{prop:5.3}, for the decomposition
	\begin{equation}
	v\ \in \  \mathcal{E}_f \smallsetminus
	\lbrace \mathcal{O} \rbrace\  = \  
	\bigcup\limits_{j=0,1,2}
	\left( {\mathcal{E}'_-}^j \cup {\mathcal{E}'_+}^j\right)\,;
	\label{eq:5.20}
	\end{equation}
see \eqref{eq:4.9}, \eqref{eq:5.4}.
Let $v \in {\mathcal{E}'_\delta}^j$, $\delta=\pm$.
By proposition~\ref{prop:5.3}(iii) this is equivalent to $z(v-\mathcal{O})=j_\delta$.
To show $j$-adjacency of $\mathcal{O}, v$, as required by \eqref{eq:5.19}, we proceed indirectly.
Suppose there exists $w\in \mathcal{E}_f \smallsetminus \lbrace \mathcal{O},v\rbrace$ such that
	\begin{equation}
	z(w-\mathcal{O})=j_\delta = z(v-w)\,;
	\label{eq:5.21}
	\end{equation}
see \eqref{eq:5.3}.
Then the left equality and proposition~\ref{prop:5.3}(iii) imply $w\in {\mathcal{E}'_\delta}^j$.
Hence $v,w$ are both in ${\mathcal{E}'_\delta}^j$, and proposition~\ref{prop:5.3}(iv) implies
	\begin{equation}
	z(v-w) <j\,.
	\label{eq:5.22}
	\end{equation}		
This contradicts the right equality in \eqref{eq:5.21}, proves \eqref{eq:5.19}, establishes $\mathcal{O} \leadsto v$, and hence proves theorem~\ref{thm:5.1}.
\end{proof}


\section{Signed homeomorphisms for Sturm 3-balls}
\label{sec6}

In this section we prove theorems~\ref{thm:2.6} and~\ref{thm:2.7}.
Theorem~\ref{thm:2.6} establishes signed homeo\-morphisms $\Phi^s$ between abstract signed 3-cell templates $\mathcal{C}$ and the signed hemisphere decompositions of the Thom-Smale dynamic complex $\mathcal{C}_f^s$ of the associated Sturm global attractor $\mathcal{A}_f$.
Theorem~\ref{thm:1.2} is the unsigned corollary.

In theorem~\ref{thm:2.6} we pass from an abstract signed 3-cell template $\mathcal{C}^s$ of cells $c_v,\ v\in \mathcal{E}$, with a formally prescribed hemisphere decomposition $S_\pm^j(v)$, to a concrete signed Sturm complex $\mathcal{C}_f^s$ of unstable manifolds $W^u(v),\ v\in \mathcal{E}_f$, with hemisphere decomposition $\Sigma_\pm^j(v)$ such that the signed dynamic complex $\mathcal{C}_f^s = \mathcal{C}^s$ realizes the prescribed signed 3-cell template $\mathcal{C}^s$.
More precisely, we have to construct a cell preserving homeomorphism
	\begin{equation}
	\Phi^s: \quad\mathcal{C}^s \longrightarrow \mathcal{C}_f^s
	\label{eq:6.1}
	\end{equation}
such that the restrictions define bijections
	\begin{align}
	\Phi^s: &\quad\mathcal{E} \longrightarrow \mathcal{E}_f\,;
	\label{eq:6.2}\\
	\Phi^s: &\quad c_v \longrightarrow W^u\left(\Phi^s(v)\right)\,;
	\label{eq:6.3}\\
	\Phi^s: &\quad S_\delta^j(v)\longrightarrow
	\Sigma_\delta^j\left(\Phi^s(v)\right)\,;
	\label{eq:6.4}
	\end{align}
for all $v\in \mathcal{E}$ and $\delta = \pm$.
This is based on the specific construction of the SZS-pair of bijections
	\begin{equation}
	h_\iota :\quad\lbrace 1,\ldots ,N\rbrace \longrightarrow\mathcal{E}\,,
	\label{eq:6.5}
	\end{equation}
$\iota= 0,1$, which is associated to the signed 3-cell template $\mathcal{C}^s$ by definition~\ref{def:2.5}.
As a consequence,
	\begin{equation}
	\sigma:= h_0^{-1} \circ h_1
	\label{eq:6.6}
	\end{equation}
is associated to a 3-meander template.
See \cite[theorem~5.2]{firo3d-1}.
In theorem~\ref{thm:5.1} above we have established that any 3-meander template $\mathcal{M}, \sigma$ in fact defines, not just some Sturm attractor but, a Sturm 3-ball $\mathcal{A}_f$ via
	\begin{equation}
	\sigma_f := \sigma\,.
	\label{eq:6.7}
	\end{equation}
In particular $\mathcal{A}_f$ comes with boundary orders
	\begin{equation}
	h_\iota^f:\quad \lbrace 1,\ldots , N\rbrace \longrightarrow
	\mathcal{E}_f
	\label{eq:6.8}
	\end{equation}
of the equilibria $v(x)$ at $x=\iota=0,1$ and defines the Sturm 3-cell template $\mathcal{C}_f^s$.

Theorem~\ref{thm:2.7} then shows, conversely, that any two nonlinearities $f, g$ which satisfy \eqref{eq:6.1}--\eqref{eq:6.4} for respective signed homeomorphisms $\Phi_f^s, \Phi_g^s$, possess identical Sturm permutations
	\begin{equation}
	\sigma_f = \sigma_g\,.
	\label{eq:6.9}
	\end{equation}
In particular their global attractors $\mathcal{A}_f, \mathcal{A}_g$ are $C^0$ orbit-equivalent; see \cite{firo00}.
The homeomorphism
	\begin{equation}
	\Phi_g^s \circ \left( \Phi_f^s\right)^{-1}:
	\quad \mathcal{C}_f^s\longrightarrow \mathcal{C}_g^s
	\label{eq:6.10}
	\end{equation}
can be required to respect decompositions into fast unstable manifolds, as well.

\begin{proof}[\textbf{Proof of theorem 2.6.}]
We establish a signed homeomorphism $\Phi^s$: $\mathcal{C}^s \rightarrow \mathcal{C}_f^s$ as in \eqref{eq:6.1}--\eqref{eq:6.4}, by successive extension.
Our basic strategy is similar to the planar case discussed in section~\ref{sec3}; see in particular the proof of corollary~\ref{cor:3.2}.
As in \eqref{eq:3.4a} we start from the identical bijections
	\begin{equation}
	\Phi^s:= h_\iota^f \circ h_\iota^{-1}: \qquad 
	\mathcal{E} \longrightarrow\mathcal{E}_f\,,
	\label{eq:6.11}
	\end{equation}
for $\iota = 0,1$.
Indeed this map does not depend on $\iota$ because $h_0^{-1} \circ h_1 = \sigma = \sigma_f = (h_0^f)^{-1} \circ h_1^f$.
This proves claim \eqref{eq:6.2}.
To simplify notation we will use \eqref{eq:6.11} to identify barycenter vertices $v\in\mathcal{E}$ of the cells $c_v \in \mathcal{C}^s$, i.e. intersections of the meander $\mathcal{M}$ of $\sigma$ with the horizontal $h_1$-axis, with the equilibria $\Phi^s(v) \in \mathcal{E}_f$, i.e. with the corresponding intersection of $\mathcal{M}_f$ viewed as an equilibrium via the shooting curve of $f$.
In particular $\mathcal{E}_f = \mathcal{E}$ and
	\begin{equation}
	h_\iota^f = h_\iota\,.
	\label{eq:6.12}
	\end{equation}

In the remaining proof we will first invoke corollary~\ref{cor:3.2}(i), on planar Sturm attractors, to establish signed homeomorphisms between the two closed hemispheres
	\begin{equation}
	\Phi_\delta^s: \qquad
	\text{clos } S_\delta^2(\mathcal{O}) \longrightarrow
	\text{clos } \Sigma_\delta^2(\mathcal{O})\,,
	\label{eq:6.13}
	\end{equation}
for $\delta = \pm$.
We will then show how $\Phi_\pm^s$ can be assumed to coincide on the intersection meridian circle
	\begin{equation}
	S^1(\mathcal{O})= 
	\text{clos } S_+^2(\mathcal{O}) \cap
	\text{clos } S_-^2(\mathcal{O})\,.
	\label{eq:6.14}
	\end{equation}
In our final step we extend $\Phi_\pm^s$ to the interior of the unique 3-cell $c_\mathcal{O}$.

We have to show how $\Sigma_-^2(\mathcal{O}) = \mathbf{W}_f$ and $\Sigma_+^2(\mathcal{O}) = \mathbf{E}_f$, in the signed 3-cell template $\mathcal{C}_f^s$, coincide with the hemispheres $\mathbf{W}$ and $\mathbf{E}$ of the prescribed 3-cell template $\mathcal{C}^s$, respectively, via hemisphere homeomorphisms $\Phi_\pm^s$ as in \eqref{eq:6.12}.
We construct $\Phi_+^s$ for the closure $\text{clos } \mathbf{E}_f= \text{clos }\Sigma_+^2(\mathcal{O})$ of the eastern hemisphere by a West scoop; the East scoop for $\text{clos } \mathbf{W}_f= \text{clos }\Sigma_-^2(\mathcal{O})$ works analogously.
See definition~\ref{def:4.2}.
The construction of the signed homeomorphism
	\begin{equation}
	\Phi_+^s:\qquad
	\text{clos }\mathbf{E}= \text{clos }S_+^2(\mathcal{O}) \longrightarrow
	\text{clos }\Sigma_+^2(\mathcal{O}) =
	\text{clos }\mathbf{E}_f
	\label{eq:6.15}
	\end{equation}
for the planar Sturm attractor $\mathcal{A}_+$ of the scooped meander $\mathcal{M}_+, \sigma_+$ simply invokes corollary~\ref{cor:3.2}(i); see \eqref{eq:3.13} in particular.

This step requires to show the following claim.
Let $(h_0^+, h_1^+)$ be the ZS-pair of the complex
	\begin{equation}
	\mathcal{C}_+^s:=
	\text{clos }\mathbf{E} =
	\text{clos }S_+^2(\mathcal{O})\,,
	\label{eq:6.16}
	\end{equation}
viewed as a planar bipolar, and hence signed, complex.
Then the West scooped meander permutation $\sigma_+$ coincides with the Sturm permutation defined by $h_\iota^+$:
	\begin{equation}
	\sigma_+ = \left( h_0^+\right)^{-1} \circ h_1^+\,.
	\label{eq:6.17}
	\end{equation}
We will show this claim in lemma~\ref{lem:6.1} below.

In lemma~\ref{lem:6.2} we will then show how the signed Sturm dynamic complex $\mathcal{A}_+ = \mathcal{C}_{f_+}^s$ of $f_+$, with $\sigma_{f_+}= \sigma_+$, coincides with the restriction $(\mathcal{C}_f^s)_+$ of the signed Sturm dynamic complex $\mathcal{A} = \mathcal{C}_f^s$ to the closed hemisphere $\text{clos } \Sigma_+^2(\mathcal{O})$:
	\begin{equation}
	\mathcal{C}_{f_+}^s= ( \mathcal{C}_f^s)_+
	\label{eq:6.18}
	\end{equation}
Combined, \eqref{eq:6.16} and \eqref{eq:6.18} construct the homeomorphism \eqref{eq:6.14} on $\text{clos }\mathbf{E}$.

The construction for $\text{clos }\mathbf{W}$ is analogous, but might differ on the shared boundary meridian $S^1(\mathcal{O})$, see \eqref{eq:6.13}.
To remedy this point, let us recall the precise construction of the signed homeomorphisms $\Phi_\pm^s$ in the planar case.
By \eqref{eq:3.7} we first extend $\Phi_\pm^s$ to the 1-skeleta $\mathcal{C}_\pm^1$ before extending to faces.
The faces of $\mathcal{C}_\pm^s$ are disjoint.
On the shared boundary meridian $S^1(\mathcal{O})$, it is sufficient to construct $\Phi_+^s$ and then define $\Phi_-^s$:= $\Phi_+^s$, there.

This completes the construction of $\Phi^s$ as a signed homeomorphism on the 2-sphere
	\begin{equation}
	\Phi^s:\quad
	\begin{aligned}	
	S^2(\mathcal{O}) &= 
	\text{clos }S_-^2(\mathcal{O}) \cup
	\text{clos }S_+^2(\mathcal{O}) \longrightarrow \\
	\longrightarrow \Sigma^2(\mathcal{O}) &=
	\text{clos }\Sigma_-^2(\mathcal{O}) \cup
	\text{clos }\Sigma_+^2(\mathcal{O})\,.
	\end{aligned}
	\label{eq:6.19}
	\end{equation}
The radial extension to the respective interiors $c_\mathcal{O} \rightarrow W^u(\mathcal{O})$ is now obvious and completes the proof of theorem~\ref{thm:2.6}, up to the next two lemmas~\ref{lem:6.1} and \ref{lem:6.2}.
\end{proof}

\begin{lem}\label{lem:6.1}
Let $(h_0^+,h_1^+)$ be the ZS-pair of the planar signed complex $\mathcal{C}_+^s$ defined by the restriction of the 3-cell template $\mathcal{C}^s$ to the closed Eastern hemisphere $\text{clos } \mathbf{E} = \text{clos } S_+^2(\mathcal{O})$.
Let $\check{h}_\iota$ denote the West scooped paths $h_\iota$ of the SZS-pair $(h_0,h_1)$ for $\mathcal{C}^s$.

Then the paths $h_\iota^+$ and $\check{h}_\iota$ coincide,
	\begin{equation}
	\check{h}_\iota = h_\iota^+\,,
	\label{eq:6.20}
	\end{equation}
for $\iota=0,1$.
In particular, consider the Sturm permutation
	\begin{equation}
	\sigma_*:=\left( h_0^+\right)^{-1}\circ h_1^+
	\label{eq:6.21}
	\end{equation}
of the planar complex $\mathcal{C}_+^s$.
Then $\sigma_*$ coincides with the scooped meander permutation $\sigma_+$ of definition~\ref{def:4.2}, i.e.
	\begin{equation}
	\sigma_+ = \sigma_*\,,
	\label{eq:6.22}
	\end{equation}
as claimed in \eqref{eq:6.17}.

The analogous statements hold for the SZ-pair $(h_0^-, h_1^-)$ on $\text{clos } \mathbf{W} = \text{clos } S_-^2(\mathcal{O})$ and the East scooped paths $\check{h}_\iota$.
\end{lem}

\begin{proof}[\textbf{Proof.}]
Since the ZS-pair $(h_0^+, h_1^+)$ is unique, we only have to show that the West scooped pair $(\check{h}_0, \check{h}_1)$ of definition~\ref{def:4.2}, \eqref{eq:4.11} forms a ZS-pair in the closed hemisphere $\text{clos }\mathbf{E}$, according to definition~\ref{def:1.1}.
Let $(h_0, h_1)$ denote the original SZS-pair of the 3-cell template $\mathcal{C}^s$, prior to the West scoop.
By construction, the Hamiltonian paths $h_\iota$ form a ZS-pair in $\text{clos }\mathbf{E}$, from their respective first emergence vertex $w_+^\iota \in \mathbf{E}$ onwards.
Before, $h_0$ and $h_1$ follow the meridians $\mathbf{N} \cup \mathbf{EW}$ and $\mathbf{N} \cup \mathbf{WE}$, respectively, in bipolar order and with interspersed excursions into $\mathbf{W}$.
See figs.~\ref{fig:1.4} and \ref{fig:2.2}.
Omitting precisely these Western excursions, in the scooped pair $(\check{h}_0, \check{h}_1)$, generates the full $\mathbf{N}$-polar serpents
	\begin{equation}
	\begin{aligned}
	\check{h}_0 &:\qquad
	\mathbf{N}= v_-^0v_-^1\ldots v_-^{2m-1} w_+^0\ldots\\
	\check{h}_1 &:\qquad
	\mathbf{N} = v_+^{2n} v_+^{2n-1} \ldots v_+^1 w_+^1\ldots
	\end{aligned}
	\label{eq:6.23}
	\end{equation}

By \cite[lemma~2.7]{firo3d-1}, the $\mathbf{N}$-polar serpents $h_\iota^+$ of the ZS-pair $(h_0^+, h_1^+)$ in the East hemisphere $\text{clos }\mathbf{E}$ are also full.
Hence the scooped paths $\check{h}_\iota$ and the ZS paths $h_\iota$ coincide everywhere, $\check{h}_\iota = h_\iota$, as claimed in \eqref{eq:6.20}.
Indeed these paths coincide, both, in their initial $\mathbf{N}$-polar serpent parts before $w_+^\iota$, and from $w_+^\iota$ onwards, for $\iota=0,1$.
Since $\sigma_+ = \check{h}_0^{-1} \circ \check{h}_1$, by definition, \eqref{eq:6.20} proves \eqref{eq:6.22} and the lemma.
\end{proof}

\begin{lem}\label{lem:6.2}
As claimed in \eqref{eq:6.18}, the signed Sturm dynamic complex $\mathcal{A}_+= \mathcal{C}_{f_+}^s$ of the West scoop $\sigma_{f_+} = \sigma_+$ of $\sigma_f =\sigma$ coincides with the restriction $(\mathcal{C}_f^s)_+$ of the signed Sturm dynamic complex $\mathcal{C}_f^s$ to the closed Eastern hemisphere $\text{clos }\Sigma_+^s(\mathcal{O})$.

The analogous statement holds for $\mathcal{A}_-= \mathcal{C}_{f_-}^s$ of the East scoop $\sigma_{f_-} = \sigma_-$ and the Western restriction $(\mathcal{C}_f^s)_- = \text{clos }\Sigma_-^s(\mathcal{O})$.
\end{lem}

\begin{proof}[\textbf{Proof.}]
Consider the Eastern restriction $(\mathcal{C}_f^s)_+$ as a given abstract planar signed complex,
	\begin{equation}
	\mathcal{C}_+^s := (\mathcal{C}_f^s)_+\,.
	\label{eq:6.24}
	\end{equation}
We then have to show that the planar signed Sturm complex $\mathcal{A}_+ = \mathcal{A}_{f_+}= \mathcal{C}_{f_+}^s$ of $f_+$ with $\sigma_{f_+}=\sigma_+$ coincides with the abstract planar complex $\mathcal{C}_+^s$:
	\begin{equation}
	\mathcal{C}_{f_+}^s = \mathcal{C}_+^2\,.
	\label{eq:6.25}
	\end{equation}
But in lemma~\ref{lem:6.1} we have already observed how the defining scoop paths $(\check{h}_0, \check{h}_1)$ of $\sigma_{f_+} = \sigma_+$ coincide with the ZS-pair $(h_0^+, h_1^+)$ of the prescribed planar complex $\mathcal{C}_+^s$.
Therefore corollary~\ref{cor:3.2}(i), \eqref{eq:3.13} proves claim \eqref{eq:6.25} and the lemma.
\end{proof}

With the above two lemmas, the proof of theorem~\ref{thm:2.6} is now also complete.

\begin{proof}[\textbf{Proof of theorem 2.7.}]
By assumptions \eqref{eq:homsfg}, \eqref{eq:2.25} we have a signed  homeomorphism $\Phi^s$ which identifies the signed versions $\mathcal{C}_f^s,\ \mathcal{C}_g^s$ of two Sturm 3-ball dynamic complexes $\mathcal{C}_f,\ \mathcal{C}_g$.
In short,
	\begin{equation}
	\mathcal{C}_f^s = \mathcal{C}_g^s\,.
	\label{eq:6.26}
	\end{equation}
We have to show that the Sturm permutations $\sigma_f$ and $\sigma_g$ coincide; see \eqref{eq:2.26}.
Moreover, we have to show how $\Phi^s$ can be chosen to preserve the fast unstable manifolds; see \eqref{eq:2.27}.

To show the first claim, $\sigma_f = \sigma_g$, we only have to show that the boundary orders $h_\iota^f,\ h_\iota^g$ of the equilibria in $\mathcal{E}_f,\ \mathcal{E}_g$ at $x=\iota=0,1$ coincide.
Identifying $\mathcal{E}_f,\ \mathcal{E}_g$ via $\Phi^s$, we can write this claim as
	\begin{equation}
	h_\iota^f = h_\iota^g\,,
	\label{eq:6.27}
	\end{equation}
for $\iota = 0,1$.
Indeed \eqref{eq:6.27} implies \eqref{eq:2.26} by
	\begin{equation}
	\sigma_f = (h_0^f)^{-1} \circ h_1^f=
	(h_0^g)^{-1} \circ h_1^g = \sigma_g\,.
	\label{eq:6.28}
	\end{equation}

To prove claim \eqref{eq:6.27} we invoke proposition~\ref{prop:2.3}.
The signed homeomorphism $\Phi^s$: $\mathcal{C}_f^s \rightarrow\mathcal{C}_g^s$ identifies the equilibria $\mathcal{E}_f$ with $\mathcal{E}_g$, and all $\mathcal{E}_{f,\pm}^j(v) = \mathcal{E}_f \cap \Sigma_{f,\pm}^j(v)$ with their counterparts $\mathcal{E}_{g,\pm}^j(v) = \mathcal{E}_g \cap \Sigma_{g,\pm}^j(v)$.
In particular, $\Phi^s$ identifies all $f$-equilibria $w_f(\mathbf{s})$ with their $g$-counterparts $w_g(\mathbf{s})$, for identical sign sequences $\mathbf{s}= s_0 \ldots s_{n-1},\ s_k\in \lbrace \pm \rbrace$.
By the table of proposition~\ref{prop:2.3}, this shows that the boundary orders $h_\iota^f ,\ h_\iota^g$ of the respective equilibria coincide, as claimed in \eqref{eq:6.27}.

We show next how the signed homeomorphism $\Phi^s$ can be chosen to respect fast unstable manifolds $W^{j+1}(v)$, as claimed in \eqref{eq:2.27}.
Let $\Phi_f^s $: $\mathcal{C}^s \rightarrow \mathcal{C}_f^s$ denote the signed homeomorphism which describes $\mathcal{C}_f^s$ as an abstract 3-cell template $\mathcal{C}^s= \mathcal{C}_f^s = \mathcal{C}_g^s$.
See \eqref{eq:6.1}.
We only have to recall how $\Phi_f^s$ was constructed by ascending dimensions $i_v= i(v)$ of Thom-Smale cells $c_v \rightarrow W^u(v)$.
On the closed ball $\overline{c}_v$ with barycenter $v$ we extended $\Phi_f^s$ radially inwards from the boundary,
	\begin{equation}
	\Phi_f^s: \quad
	\partial c_v \longrightarrow \partial W_f^u(v) =
	\Sigma_f^{i(v)-1} (v)\,.
	\label{eq:6.29}
	\end{equation}
The fast unstable manifolds $W_f^{j+1}$, likewise, possess sphere boundaries and, by induction on cell dimension, we may assume
	\begin{equation}
	\Phi_f^s: \quad
	S^j(v) \longrightarrow \partial W_f^{j+1}(v) =
	\Sigma_f^j(v)\,,
	\label{eq:6.30}
	\end{equation}
for $0\leq j< i(v)$.
Since $\Phi_f^s$ is a signed homeomorphism, and passing to the notation of signed hemispheres, we have,
	\begin{equation}
	\Phi_f^s: \quad
	S_\delta^j(v) \longrightarrow \Sigma_\delta^j(v)\,,
	\label{eq:6.31}
	\end{equation}
for $\delta = \pm$.
The Schoenflies result \cite{firo13} provided extensions of \eqref{eq:6.31}, to the interior balls $B^{j+1}(v)$, such that the standard eigenspaces $E^{j+1}$ mapped to $W_f^{j+1}(v)$.
Similarly, positive and negative half spaces are mapped to the signed versions of $W_f^{j+1}(v)$, separated by $W_f^j(v)$, for $0\leq j < i(v)$.
Replacing radial extensions by this more refined construction of $\Phi_f^s$ we see how standard (half) eigenspaces just get mapped to (signed) fast unstable manifolds.
Since the same statement holds for $\Phi_g^s$: $\mathcal{C}^s\rightarrow \mathcal{C}_g^s$, on the same 3-cell template complex $\mathcal{C}^s$, the combined signed homeomorphism
	\begin{equation}
	\Phi^s = \Phi_g^s \circ \left(\Phi_f^s\right)^{-1}:\quad
	\mathcal{C}_f^s \longrightarrow \mathcal{C}_g^s
	\label{eq:6.32}
	\end{equation}
respects signed fast unstable manifolds.
This completes the proof of claim \eqref{eq:2.27}, and the proof of theorem~\ref{thm:2.7}.
\end{proof}


\section{Appendix: Wolfrum's lemma}\label{sec7}
In this technical appendix we comment on, and repair, a gap in the original proof of Wolfrum's lemma~\ref{lem:5.2}.

In \cite[theorem~2.1]{wo02} the lemma has first been formulated in the present form.
The gap in the proof arises, formally, by an overinterpretation of realization results in \cite{firo99} to provide templates for arbitrary sequences of saddle-node bifurcations. This is not what had been proved there.
The relevant result is \cite[lemma~3.1]{firo99}.
Already in the simplest case it is based, first, on a ``short arc'' \emph{nose retraction}, via a saddle-node bifurcation. Second, the resulting nose in the meander $\mathcal{M}$ has to be retracted \emph{counterclockwise} towards the lower, reduced, number of equilibria. See \cite[fig.~3]{firo99}. This brings the relevant Sturm shooting meanders $\mathcal{M}$ into canonical form, as specified in \cite{firo99}. The counterclockwise restriction in the second step has not been addressed in \cite{wo02}.

In fact, the results in \cite{firo99} do allow a nose removal by a saddle-node bifurcation which pushes its ``short arc'' of $\mathcal{M}$ nearly vertically through the horizontal axis. This addresses the first step, locally. Neither before, nor after, such a local sadlle-node bifurcation, however, would the resulting meander be in canonical form, globally. 

Therefore it remains crucial to lift the clockwise restriction in
the second step, towards canonical meanders. We use the global rigidity of Sturm attractors proved in \cite{firo00}:
global Sturm attractors $\mathcal{A}_f$ and $\mathcal{A}_g$ with identical Sturm permutations $\sigma_f = \sigma_g$ are $C^0$ orbit equivalent.
In view of that global rigidity, the Sturm permutations on either side of the local saddle-node bifurcation can therefore be realized by shooting curves, again, which are canonical meanders.
As a caveat we add that it is still unknown to us whether that second step can be achieved by a global parameter homotopy of Sturm nonlinearities $f$, within the PDE class \eqref{eq:1.1}.
Instead, the rigidity proof in \cite{firo00} used a discretization, and subsequent dimensional augmentation, to provide parameter homotopies in the potentially much wider ODE class of finite-dimensional Jacobi systems.
At any rate, this remedies both gaps in the proof of \cite[theorem~2.1]{wo02}.

The proof of Wolfrum's lemma is independent of a Conley index argument in \cite{firo96} which led to a weaker result.
See \cite[remark~4.1]{wo02}.
Above we have indicated how arguments of \cite{firo99, firo00} enter, instead.

\end{document}